\documentclass{amsart}

\usepackage[a4paper,hmargin=2.9cm,vmargin=4cm]{geometry}
\usepackage{appendix}
\usepackage{amsfonts,amssymb,amscd,amstext}
\usepackage{graphicx}
\usepackage[dvips]{epsfig}
\usepackage[colorlinks=true,linkcolor=blue,citecolor=red]{hyperref}

\renewcommand{\leq}{\leqslant}
\renewcommand{\geq}{\geqslant}
\newcommand{\ptl}{\partial}
\newcommand{\rr}{{\mathbb{R}}}
\newcommand{\la}{\lambda}
\newcommand{\h}{\mathcal{H}}
\newcommand{\sub}{\subset}
\newcommand{\subeq}{\subseteq}
\newcommand{\escpr}[1]{\big<#1\big>}
\newcommand{\Sg}{\Sigma} 
\newcommand{\Om}{\Omega}
\newcommand{\eps}{\varepsilon}
\newcommand{\ga}{\gamma}
\newcommand{\Ga}{\Gamma}
\newcommand{\mnh}{|N_{h}|}
\newcommand{\nuh}{\nu_{h}}
\newcommand{\m}{\mathbb{N}(\kappa)}
\newcommand{\e}{\mathbb{M}(\kappa)}
\newcommand{\mm}{\mathbb{M}}
\newcommand{\s}{\mathcal{S}}
\newcommand{\rrr}{\mathcal{R}}
\newcommand{\sla}{\mathcal{S}_\lambda}
\newcommand{\var}{\varphi}
\newcommand{\ric}{\text{Ric}}

\DeclareMathOperator{\divv}{div}

\newtheorem{theorem}{Theorem}[section]
\newtheorem{proposition}[theorem]{Proposition}
\newtheorem{lemma}[theorem]{Lemma}
\newtheorem{corollary}[theorem]{Corollary}

\theoremstyle{definition}

\newtheorem{remark}[theorem]{Remark}
\newtheorem{remarks}[theorem]{Remarks}
\newtheorem{definition}[theorem]{Definition}
\newtheorem{example}[theorem]{Example}

\numberwithin{equation}{section}

\begin{document}

\title[Pansu spheres in sub-Riemannian $3$-space forms]{Existence, characterization and stability \\ of Pansu spheres in sub-Riemannian $3$-space forms}

\author[A.~Hurtado]{Ana Hurtado}
\address{Departamento de Geometr\'{\i}a y Topolog\'{\i}a \\
Universidad de Granada \\ E--18071 Granada \\ Spain}
\email{ahurtado@ugr.es}

\author[C.~Rosales]{C\'esar Rosales}
\address{Departamento de Geometr\'{\i}a y Topolog\'{\i}a \\
Universidad de Granada \\ E--18071 Granada \\ Spain}
\email{crosales@ugr.es}

\date{\today}

\thanks{The authors have been supported by Mec-Feder grant MTM2010-21206-C02-01, Mineco-Feder grant MTM2013-48371-C2-1-P, and Junta de Andaluc\'ia grant FQM-325. The second author has been also supported by the grant PYR-2014-23 of the GENIL program of CEI BioTic GRANADA} 



\begin{abstract}
Let $M$ be a complete Sasakian sub-Riemannian $3$-manifold of constant Webster scalar curvature $\kappa$. For any point $p\in M$ and any number $\la\in\rr$ with $\lambda^2+\kappa>0$, we show existence of a $C^2$ spherical surface $\sla(p)$ immersed in $M$ with constant mean curvature $\lambda$. Our construction recovers in particular the description of Pansu spheres in the first Heisenberg group \cite{pansu1} and the sub-Riemannian $3$-sphere \cite{hr1}. Then, we study variational properties of $\sla(p)$ related to the area functional. First, we obtain uniqueness results for the spheres $\sla(p)$ as critical points of the area under a volume constraint, thus providing sub-Riemannian counterparts to the theorems of Hopf and Alexandrov for CMC surfaces in Riemannian $3$-space forms. Second, we derive a second variation formula for admissible deformations possibly moving the singular set, and prove that $\sla(p)$ is a second order minimum of the area for those preserving volume. We finally give some applications of our results to the isoperimetric problem in sub-Riemannian $3$-space forms.
\end{abstract}

\maketitle

\thispagestyle{empty}

\section{Introduction}
\label{sec:intro}
\setcounter{equation}{0} 

The study of variational questions related to the area functional in sub-Riemannian geometry is a focus of recent attention, especially in the Heisenberg group $\mathbb{H}^n$, which is the simplest non-trivial example of a flat sub-Riemannian manifold. Indeed, the \emph{isoperimetric problem} in $\mathbb{H}^n$, where we seek regions minimizing the (sub-Riemannian) perimeter under a volume constraint, is one of the main open problems in this field with several important contributions in the last ten years. 

The existence and boundedness of isoperimetric solutions in $\mathbb{H}^n$ are consequence of more general results of Leonardi and Rigot for Carnot groups~\cite{lr}. Moreover, the first variational formulas for area and volume imply that the boundary of any $C^2$ solution is a compact hypersurface with constant mean curvature (CMC) in sub-Riemannian sense, see \cite{chmy} and \cite{rr1}. In 1982 it was conjectured by Pansu~\cite{pansu1} that the isoperimetric regions in $\mathbb{H}^1$ are congruent to the topological balls enclosed by a family $\{\sla\}_{\la>0}$ of CMC spheres of revolution with $C^2$ regularity, see \cite{monti}, \cite{leomas} and \cite{rr1} for explicit expressions of $\sla$. These spheres are geometrically obtained by rotating about the vertical axis in $\mathbb{H}^1$ a Carnot-Carath\'eodory geodesic (CC-geodesic) connecting two points on this axis. Such a geodesic is a Euclidean helix with curvature depending only on the vertical distance between its extreme points. Pansu's conjecture has been supported by many partial results, where further hypotheses involving regularity and/or symmetry of the solutions are assumed. We refer the reader to \cite[Ch.~8]{survey} and the Introduction of \cite{ritore-calibrations} for a precise account and description of the related works. More recently, Monti~\cite{monti-rearrangements} has analyzed symmetrization in $\mathbb{H}^n$, and Cheng, Chiu, Hwang and Yang~\cite{cchy} has used a notion of umbilicity to characterize Pansu's spheres in $\mathbb{H}^n$ with $n\geq 2$. It is worth mentioning that the optimal regularity for isoperimetric sets in $\mathbb{H}^n$ is still an open question. Indeed, though the conjectured solutions are $C^2$ smooth, there exist non-compact area-minimizing surfaces in $\mathbb{H}^1$ with less regularity, see \cite{pauls-regularity}, \cite{chy}, \cite{mscv} and \cite{r2}.

Besides the Heisenberg group, it is also natural to investigate the isoperimetric question in other sub-Riemannian $3$-spaces, especially in those representing the most simple and symmetric examples of the theory. These are the \emph{sub-Riemannian $3$-space forms}, defined in Section~\ref{subsec:sf} as complete Sasakian sub-Riemannian $3$-manifolds of constant Webster scalar curvature. By a result of Tanno~\cite{tanno} there is, up to isometries, a unique simply connected sub-Riemannian $3$-space form $\e$ of Webster scalar curvature $\kappa$. The space $\e$ is the Heisenberg group $\mathbb{H}^1$ for $\kappa=0$, the group of unit quaternions $\mathbb{S}^3\sub\rr^4$ for $\kappa=1$, and the universal cover of the special linear group $\text{SL}(2,\rr)$ for $\kappa=-1$. It is also known that the spaces $\e$ are the unique simply connected, homogeneous, contact sub-Riemannian $3$-manifolds with the largest possible isometry group, see for instance \cite{falbel1}. For an arbitrary sub-Riemannian $3$-space form $M$, standard Riemannian arguments show that $M$ is isometric to $\e/G$, where $\kappa$ is the Webster scalar curvature of $M$ and $G$ is a subgroup of isometries of $\e$ isomorphic to the fundamental group of $M$. This permits the construction of $3$-space forms with non-trivial topology like flat and hyperbolic cylinders $\rr^2\times\mathbb{S}^1$, as well as a sub-Riemannian model of the real projective space $\rr\mathbb{P}^3$, see Example~\ref{ex:topology} for details.

Our main motivation in this paper is to analyze a possible extension of Pansu's conjecture inside any sub-Riemannian $3$-space form $M$. For that, we must produce first CMC spheres $\sla(p)$ playing the same role in $M$ as Pansu's spheres in the Heisenberg group $\mm(0)$. This was already done in the sub-Riemannian $3$-sphere $\mathbb{M}(1)$, where the authors employed the explicit expression of the CC-geodesics to construct $\sla(p)$, see \cite{hr1}. In the present work we avoid the computation of CC-geodesics; instead of that, we will use \emph{CC-Jacobi fields} (introduced in Section~\ref{subsec:ccgeo}) systematically. This approach allows us to introduce and study in a unified way the spheres $\sla(p)$ for general sub-Riemannian $3$-space forms.

In Section~\ref{sec:spheres} we define $\sla(p)$ as follows. For a fixed point $p\in M$ and a number $\la\in\rr$ we consider the flow $F(\theta,s):=\ga_\theta(s)$ of all the CC-geodesics of curvature $\la$ in $M$ leaving from $p$ with unit horizontal velocity $\theta$. Thus, in order to produce a spherical surface, we require the CC-geodesics $\ga_\theta$ to meet again in $M$. This is clearly equivalent to that the associated vector fields $V_\theta(s):=(\ptl F/\ptl\theta)(s,\theta)$ vanish for some positive value. The remarkable fact is that $V_\theta$ is a CC-Jacobi field along $\ga_\theta$ with $V_\theta(0)=0$ and so, it is explicitly computable since the Webster scalar curvature of $M$ is constant. From this calculus we deduce in Lemma~\ref{lem:condition} that the arithmetic condition $\la^2+\kappa>0$ is equivalent to the existence of a first cut point $p_\la$ for the CC-geodesics $\ga_\theta(s)$ when $s:=\pi/\sqrt{\la^2+\kappa}$. This leads us to consider the set
\[
\sla(p):=\{\ga_\theta(s)\,;0\leq s\leq \pi/\sqrt{\la^2+\kappa}\},
\]
where $\theta$ ranges over the unit horizontal circle of $M$ at $p$. In the model spaces $\mm(0)$ and $\mm(1)$ the previous definition recovers the spherical surfaces described in \cite{pansu1} and \cite{hr1}. In the hyperbolic space $\e$ with $\kappa<0$ we show in Theorem~\ref{th:sla} that $\sla(p)$ is also an embedded $C^2$ surface congruent to a sphere of revolution with $C^\infty$ regularity off of the poles $\{p,p_\la\}$. So, by taking into account that any sub-Riemannian space form $M$ is a quotient of some $\e$, we are able to prove in Theorem~\ref{th:sla} that $\sla(p)$ is a $C^2$ oriented sphere immersed in $M$ with $C^\infty$ regularity off of the singular set $\{p,p_\la\}$, where the tangent plane equals the horizontal plane. Moreover, in Theorem~\ref{th:sla} we exhibit some geometric properties of $\sla(p)$ regarding its symmetry and embeddedness. Precisely, $\sla(p)$ is always Alexandrov embedded and invariant under the isometries of $M$ fixing $p$. As a difference with respect to the spaces $\e$ the sphere $\sla(p)$ need not be embedded in $M$ for small values of $\la$, see Remarks~\ref{re:pozo}. After some computations in the coordinates $(\theta,s)$ we deduce that $\sla(p)$ has CMC equal to $\la$ off of the poles. Furthermore, if $N_h$ denotes the horizontal projection of a unit normal $N$ along $\sla(p)$, then the function $\mnh^{-1}$ turns out to be integrable, see Corollary~\ref{cor:integrability}. This is a non-trivial fact (since $\mnh$ vanishes at the poles) that will be used frequently through this work.

Once we have constructed the spheres $\sla(p)$ it is natural to analyze if they provide nice candidates to solve the isoperimetric problem in any sub-Riemannian $3$-space form $M$. For that, we first observe that $\sla(p)$ is \emph{volume-preserving area-stationary}, i.e., a critical point of the area functional for deformations preserving volume. This is an immediate consequence of the characterization result in \cite{rr2}, see Theorem~\ref{th:vpstationary}, since $\sla(p)$ is a $C^2$ oriented CMC surface with finitely many singular points. Now, an interesting question is to characterize $\sla(p)$ in the family of volume-preserving area stationary surfaces, possibly under additional hypotheses. In Section~\ref{subsec:classresults} we prove some uniqueness results in this line providing sub-Riemannian counterparts to the well-known Hopf and Alexandrov theorems for CMC surfaces in Riemannian manifolds. 

Recall that Hopf's theorem \cite{hopf} establishes that a topological sphere immersed in $\rr^3$ with CMC is a round sphere. This was extended by Chern~\cite{chern} to arbitrary complete Riemannian $3$-manifolds of constant sectional curvature. In Theorem~\ref{th:hopf} we show that a $C^2$ oriented sphere with CMC immersed in a sub-Riemannian $3$-space form $M$ must be a spherical surface $\sla(p)$. Recently Cheng, Chiu, Hwang and Yang~\cite{cchy} have obtained a Hopf theorem for umbilic hypersurfaces in the Heisenberg group $\mathbb{H}^n$ with $n\geq 2$. On the other hand, Alexandrov's theorem~\cite{alexandrov} states that a compact CMC surface embedded in $\rr^3$ is a round sphere. The proof, based on the so-called Alexandrov reflection technique, can be readily generalize to the hyperbolic space and a hemisphere of $\mathbb{S}^3$. In Theorem~\ref{th:alexandrov} we characterize the spheres $\sla(p)$ as the unique compact volume-preserving area-stationary surfaces of class $C^2$ in $\mm(\kappa)$ with $\kappa\leq 0$, or inside an open hemisphere of $\mm(\kappa)$ with $\kappa>0$. This theorem was previously established in \cite{rr2} for the particular case of $\mm(0)$. The vertical Clifford tori in $\mathbb{S}^3$ illustrate that the Alexandrov theorem is no longer true if the surface is not entirely contained inside an open hemisphere. As we see in Remark~\ref{re:tekel} the theorem also fails for sub-Riemannian $3$-space forms with non-trivial topology. The proofs of Theorems~\ref{th:hopf} and \ref{th:alexandrov} rely on results in \cite{chmy} and \cite{rr2} about geometric, topological and variational properties of $C^2$ critical points of the area under a volume constraint, see Section~\ref{subsec:statsurf} for precise statements. These were employed in \cite{rr2} and \cite{hr1} to describe complete $C^2$ volume-preserving area-stationary surfaces with non-empty singular set in $\mm(0)$ and $\mm(1)$. With similar techniques Galli characterized $C^2$ area-minimizing and area-stationary surfaces in the roto-translation group \cite{galli}, and in the sub-Riemannian Sol manifold \cite{galli2}. Recently, Bernstein type problems and local minimizers of the area with $C^1$ or less regularity have been studied in the Heisenberg groups $\mathbb{H}^n$, see the Introduction of \cite{galli-ritore2} for a complete account of related works.

Once we know that the spheres $\sla(p)$ are volume-preserving area-stationary surfaces, we analyze in Section~\ref{sec:stability} if they are also \emph{second order minima of the area under a volume constraint}, thus providing natural candidates to solve the isoperimetric problem. As a fundamental tool, we need the second derivative of the functional $A+2\la V$ stated in Theorem~\ref{th:2ndsla} for certain deformations of $\sla(p)$. The second variation of the sub-Riemannian area has appeared in several contexts, see for instance \cite{chmy}, \cite{bscv}, \cite{dgn}, \cite{montefalcone}, \cite{mscv}, \cite{hrr}, \cite{ch2}, \cite{hp2}, \cite{galli}, \cite{montefalcone2} and \cite{galli-ritore2}. The derivation of the second variation formula involves a technical problem since, in contrast to the first variation, the differentiation under the integral sign is not always possible. Nevertheless, this problem is overcome if we require the employed deformations to be \emph{admissible}. These are introduced in Definition~\ref{def:admissible} as those satisfying the conditions necessary to apply the usual Leibniz's rule for differentiating under the integral sign, see Lemma~\ref{lem:difint}. We must remark that, unlike some previous approaches, our deformations are only assumed to be admissible; they are no restricted to fixing the singular set nor having a special form. Indeed, the notion of admissibility permits us to obtain a general second variation formula for CMC surfaces in arbitrary Sasakian sub-Riemannian $3$-manifolds, see Theorem~\ref{2ndgeneral}. The proof of Theorem~\ref{th:2ndsla} comes from Theorem~\ref{2ndgeneral} with an additional effort to demonstrate that all the divergence terms in the general second variation formula vanish for $\sla(p)$. For non-admissible variations the second derivative $A''(0)$ may exist or not and, in case of existence, it is not easy to compute it. Some particular examples of these situations for CMC surfaces with singular curves were discussed in \cite{hrr}, \cite{ch2} and \cite{galli}. 

From Theorem~\ref{th:2ndsla} the second derivative of the functional $A+2\la V$ associated to an admissible variation of a sphere $\sla(p)$ is controlled by a quadratic form $\mathcal{I}$, called the \emph{index form}, which involves the normal component of the velocity vector field and some geometric quantities defined off of the poles. Surprisingly, the expression in polar coordinates of $\mathcal{I}$ does not depend on the ambient manifold $M$ nor the point $p\in M$, which suggests that all the spheres $\sla(p)$ will share the same behaviour with respect to stability conditions. Indeed in Proposition~\ref{prop:strictstability} we deduce that $\sla(p)$ is always \emph{strongly stable}, that is $(A+2\la V)''(0)\geq 0$, for any admissible variation vanishing at the poles or along the equator. This generalizes a previous result of Montefalcone~\cite{montefalcone-spheres} for some particular variations of Pansu's spheres in the Heisenberg group $\mm(0)$. Nevertheless the spheres $\sla(p)$ are not strongly stable for arbitrary admissible variations. To see this it is enough to consider the deformation of $\sla(p)$ by Riemannian parallel surfaces, see Remark~\ref{re:chispa} for details. However, by using analytical arguments based on double Fourier series, we are able to prove in Theorem~\ref{th:stability} that $\sla(p)$ is \emph{stable under a volume constraint}, i.e., it has non-negative second derivative of the area under volume-preserving admissible deformations. In the Heisenberg group $\mathbb{H}^n$, Montefalcone~\cite{montefalcone-spheres} analyzed a $1$-dimensional eigenvalues problem to establish that the Pansu spheres satisfy the same stability property for certain radially symmetric deformations. It is worth mentioning that, though our stability results are proved for admissible deformations of $\sla(p)$, they are still valid for any variation for which the second variation formula in Theorem~\ref{th:2ndsla} holds. Moreover, as it is shown in Appendix~\ref{sec:appendix2}, the family of admissible variations of $\sla(p)$ is very large.

We finish the present work with Section~\ref{sec:isoperimetric}, where we come back to the isoperimetric problem inside any sub-Riemannian $3$-space form $M$. Note that, if $M$ is homogeneous, then we can apply the results of Galli and Ritor\'e~\cite{galli-ritore} to ensure existence and boundedness of isoperimetric regions for any given volume. Unfortunately, the regularity of the minimizers is still an open question. However, it is clear that the boundary of a minimizer must be volume-preserving area-stationary. In particular, if $\Om$ is a $C^2$ isoperimetric region in the model spaces $\e$ with $\kappa\leq 0$, then we can invoke Theorem~\ref{th:alexandrov} to deduce that $\ptl\Om$ must be one of the spherical surfaces $\sla(p)$. This extends to the sub-Riemannian hyperbolic models the uniqueness of $C^2$ solutions in $\mm(0)$ proved in \cite{rr2}. In the spherical models $\e$ with $\kappa>0$ the same conclusion holds for $C^2$ isoperimetric regions contained inside an open hemisphere of $\mathbb{S}^3$. In \cite{hr3} we use the second variation formula to discard the existence in $\e$ with $\kappa>0$ of stable tori under a volume constraint, so that, for arbitrary $C^2$ isoperimetric regions, we get the same uniqueness result as in $\e$ with $\kappa\leq 0$. From this we may suggest that the spheres $\sla(p)$ uniquely minimize the perimeter under a volume constraint in $\e$. This is a generalization of Pansu's conjecture in the Heisenberg group $\mm(0)$. Indeed, our stability result in Theorem~\ref{th:stability} provides further evidence supporting this extended conjecture. On the other hand, in a homogeneous space form $M$ of non-trivial topology, a sphere $\sla(p)$ need not bound a minimizer. This is shown in Example~\ref{ex:mendicuti}, where we find vertical tori inside a Sasakian flat cylinder $\rr^2\times\mathbb{S}^1$ which are isoperimetrically better than the spheres $\sla(p)$ of the same volume. So, an interesting question in this setting is the following: are the topological balls enclosed by $\sla(p)$ isoperimetric solutions for small volumes?

The paper is organized into six sections and two appendices. The second section contains background material about sub-Riemannian $3$-manifolds. In Section~\ref{sec:spheres} we construct the spherical surfaces $\sla(p)$ and show some of their geometric properties. In the fourth section we establish uniqueness results for $\sla(p)$ and describe volume-preserving area-stationary surfaces with non-empty singular set. In Section~\ref{sec:stability} we obtain stability results for admissible variations of $\sla(p)$. Some applications to the isoperimetric problem in sub-Riemannian $3$-space forms are given in Section~\ref{sec:isoperimetric}. The second variation formula is proved in detail in Appendix~\ref{sec:appendix1}. Finally, in Appendix~\ref{sec:appendix2} we produce several examples of admissible variations. 

\noindent
\textbf{Acknowledgements.} The authors wish to thank the referee for his/her detailed report and useful suggestions to improve the presentation of the paper.

\section{Preliminaries}
\label{sec:preliminaries}
\setcounter{equation}{0}

In this section we introduce the notation and establish some results that will be used throughout the paper.

\subsection{Sasakian sub-Riemannian $3$-manifolds}
\label{subsec:ssR3m}
\noindent

A \emph{contact sub-Riemannian manifold} is a connected manifold $M$ with empty boundary together with a Riemannian metric $g_h$ defined on an oriented contact distribution $\h$, which we refer to as the \emph{horizontal distribution}. We denote by $J$ the orientation-preserving $90$ degree rotation in $(\h,g_h)$. This is a contact structure on $\h$ since $J^2=-\text{Id}$. 

The \emph{normalized form} is the contact $1$-form $\eta$ on $M$ such that $\text{Ker}(\eta)=\h$ and the restriction of the $2$-form $d\eta$ to $\h$ equals the area form on $\h$. We shall always choose the orientation of $M$ induced by the $3$-form $\eta\wedge d\eta$. The \emph{Reeb vector field} associated to $\eta$ is the vector field $T$ transversal to $\h$ defined by $\eta(T)=1$ and $d\eta(T,U)=0$, for any $U$. We extend $J$ to the tangent space to $M$ by setting $J(T):=0$. A vector field $U$ is \emph{horizontal} if it coincides with its projection $U_h$ onto $\h$. If $U$ is always proportional to $T$ then we say that $U$ is \emph{vertical}. 

The \emph{canonical extension} of $g_h$ is the Riemannian metric $g=\escpr{\cdot\,,\cdot}$ on $M$ for which $T$ is a unit vector field orthogonal to $\h$. For this metric it is easy to check that
\begin{equation}
\label{eq:conmute}
\escpr{J(U),V}+\escpr{U,J(V)}=0,
\end{equation}
for any $U$ and $V$. We say that $M$ is \emph{complete} if $(M,g)$ is a complete Riemannian manifold. 

By a \emph{local isometry} between contact sub-Riemannian $3$-manifolds we mean a smooth ($C^\infty$) map $\phi:M\to M'$ whose differential at any $p\in M$ is an orientation-preserving linear isometry from $\h_p$ to $\h'_{\phi(p)}$. These maps preserve the normalized forms, the Reeb vector fields, the associated complex structures and the canonical extensions of the sub-Riemannian metrics. An \emph{isometry} is a local isometry which is also a diffeomorphism. 
A contact sub-Riemannian $3$-manifold $M$ is \emph{homogeneous} if the group $\text{Iso}(M)$ of isometries of $M$ acts transitively on $M$. Such a manifold must be complete. We say that two subsets $S_1$ and $S_2$ of $M$ are \emph{congruent} if there is $\phi\in\text{Iso}(M)$ such that $\phi(S_1)=S_2$.

By a \emph{Sasakian sub-Riemannian $3$-manifold} we mean a contact sub-Riemannian $3$-manifold $M$ where $g_h$ is a \emph{Sasakian metric}: this means that any diffeomorphism of the one-parameter group associated to $T$ is an isometry. This implies, see \cite[p.~67, Cor.~6.5 and Thm.~6.3]{blair}, that the Levi-Civit\`a connection $D$ associated to $g$ satisfies the equalities
\begin{align}
\nonumber
D_UT&=J(U),
\\
\label{eq:dujv}
D_U\left(J(V)\right)&=J(D_UV)+\escpr{V,T}\,U-\escpr{U,V}\,T,
\end{align}
for any vector fields $U$ and $V$ on $M$. In particular, any integral curve of $T$ is a geodesic in $(M,g)$ parameterized by arc-length. We refer to these curves as \emph{vertical geodesics}.

Now we recall some curvature properties of a Sasakian $3$-manifold. The curvature tensor in $(M,g)$ is defined by
\[
R(U,V)W:=D_{V}D_{U}W-D_{U}D_{V}W+D_{[U,V]}W,
\]
where $[U,V]$ is the Lie bracket. From \cite[Prop.~7.3]{blair} and the fact that $\escpr{U,T}=\eta(U)$, we have
\begin{equation}
\label{eq:ruvt}
R(U,V)T=\escpr{U,T}\,V-\escpr{V,T}\,U.
\end{equation}

The \emph{Webster scalar curvature} $K$ of $M$ is the sectional curvature of $\h$ with respect to the Tanaka connection \cite[Sect.~10.4]{blair}. If $K_h$ denotes the sectional curvature of $\h$ in $(M,g)$, then
\[
K=\frac{1}{4}\,\big(K_h+3\big),
\]
which is clearly invariant under local isometries. The following formulas valid for $U$ horizontal and $V$ arbitrary were proved in \cite[Lem.~2.1]{rosales}
\begin{align}
\label{eq:ruvu}
R(U,V)U&=(4K-3)\,\escpr{V,J(U)}\,J(U)+|U|^2\,\escpr{V,T}\,T,
\\
\label{eq:ricvv}
\text{Ric}(V,V)&=(4K-2)\,|V_h|^2+2\,\escpr{V,T}^2,
\end{align}
where $\text{Ric}$ is the Ricci curvature in $(M,g)$ defined, for any two vectors fields $U$ and $V$, as the trace of the map $W\mapsto R(U,W)V$. 

Let $M$ be a contact sub-Riemannian $3$-manifold. The \emph{homothetic deformation} of $g_h$ is the one-parameter family of sub-Riemannian metrics $\{(g_h)_\eps\}_{\eps>0}$ on $\h$ defined by
\begin{equation}
\label{eq:homothetic}
(g_h)_\eps:=\eps^2 g_h.
\end{equation}
This induces new structures of contact sub-Riemannian $3$-manifold on $M$, see \cite[p.~103]{blair} and \cite{tanno2}. The normalized form $\eta_\eps$, the Reeb vector field $T_\eps$, the complex structure $J_\eps$ and the canonical extension $g_\eps$ are, respectively, given by $\eta_\eps=\eps^2\,\eta$, $T_\eps=T/\eps^2$, $J_\eps=J$, and
\begin{equation}
\label{eq:ga}
g_\eps(U,V)=\eps^2\,g(U,V)+\eps^2\,(\eps^2-1)\,g(U,T)\,g(V,T).
\end{equation}
It is easy to check that any isometry for $g_h$ is also an isometry for $(g_h)_\eps$. Therefore $(g_h)_\eps$ is a Sasakian metric if and only if the same holds for $g_h$. Moreover, in such a case, we obtain
\[
K_\eps=\frac{K}{\eps^2},
\]
where $K_\eps$ is the Webster scalar curvature for $(g_h)_\eps$.

\subsection{Three-dimensional space forms}
\label{subsec:sf}
\noindent

For $\kappa=-1,0,1$, we denote by $\m$ the complete, simply connected, Riemannian surface of constant sectional curvature $4\kappa$ described as follows. If $\kappa=1$, then $\mathbb{N}(1)$ is the unit sphere $\mathbb{S}^2$ with its standard Riemannian metric scaled by  $1/4$. If $\kappa=-1,0$, then $\m:=\{p\in\rr^2\,;\,|p|<1/|\kappa|\}$ endowed with the Riemannian metric $\rho^2\,(dx^2+dy^2)$, where $\rho(x,y):=(1+\kappa(x^2+y^2))^{-1}$. Hence $\mathbb{N}(-1)$ is the Poincar\'e model of the hyperbolic plane and $\mathbb{N}(0)$ is the Euclidean plane. 

For $\kappa=-1,0$ we denote $\e:=\m\times\rr$. If $\kappa=1$, then we set $\mathbb{M}(1)=\mathbb{S}^3$. The space $\e$ admits a structure of complete Sasakian sub-Riemannian manifold of Webster scalar curvature $\kappa$, see \cite[Sect.~2.2]{rosales} and \cite[Sect.~7.4]{blair} for an explicit description. The sub-Riemannian metric $g_h$ in $\e$ is of \emph{bundle type} in the sense of \cite[p.~18]{montgomery}. In fact, there is a Riemannian submersion $\mathcal{F}:\e\to\m$ for which $\h=(\text{Ker}(d\mathcal{F}))^\bot$. Here the orthogonal complement is taken with respect to the canonical extension $g$ of $g_h$. The map $\mathcal{F}$ is the Euclidean projection $\mathcal{F}(x,y,t):=(x,y)$ if $\kappa=-1,0$, and the Hopf fibration $\mathcal{F}(x_1,y_1,x_2,y_2):=(x^2_{1}+y^2_{1}-x_{2}^2-y_{2}^2,2\,(x_{2}y_{1}-x_{1}y_{2}),2\,(x_{1}x_{2}+y_{1}y_{2}))$ if $\kappa=1$. 
We shall always consider the orientation on $\m$ for which the restriction of $d\mathcal{F}$ to $\h$ is an orientation-preserving isometry. The oriented fiber of $\mathcal{F}$ through $p$ is the vertical geodesic in $(\e,g)$ parameterized by $p+t\,T_p$ if $\kappa=-1,0$ or $e^{it}p$ if $\kappa=1$. We call \emph{vertical axis} the fiber through $o:=(0,0,0)$ if $\kappa\leq 0$ or through $o:=(1,0,0,0)$ if $\kappa>0$. The diffeomorphisms associated to $T$ are translations along the fibers. By means of the homothetic deformation in \eqref{eq:homothetic} we can construct, for arbitrary $\kappa\in\rr$, the Sasakian space $\e$ of constant Webster scalar curvature $\kappa$.

In $\e$ there is a product $*$ such that $(\e,*)$ is a Lie group. More precisely, we get the Heisenberg group when $\kappa=0$, the group of unit quaternions when $\kappa>0$, and the universal covering of the special linear group $\text{SL}(2,\rr)$ when $\kappa<0$. The associated left translations (resp. right translations) are isometries of $\e$  when $\kappa\leq 0$ (resp. $\kappa>0$). Hence the space $\e$ is homogeneous and, in particular, complete. Any Euclidean vertical rotation (a rotation about the vertical axis) is also an isometry of $\e$. 

By a result of Tanno~\cite{tanno} the space $\e$ is, up to isometries, the unique complete, simply connected, Sasakian sub-Riemannian $3$-manifold with constant Webster scalar curvature $\kappa$. This result together with a standard argument involving Riemannian coverings (see the proof of Prop.~4.3 in \cite[Chap.~8]{dcriem} for details) allows to deduce the following fact. 

\begin{proposition}
\label{prop:model}
If $M$ is a complete Sasakian sub-Riemannian $3$-manifold of constant Webster scalar curvature $\kappa$, then there is a surjective local isometry $\Pi:\e\to M$. Moreover, there is a subgroup $G$ of isometries of $\e$ such that $M$ is isometric to the quotient $\e/G$, and the fundamental group of $M$ is isomorphic to $G$. 
\end{proposition}

In the sequel, by a \emph{$3$-dimensional space form} we mean a complete Sasakian sub-Riemannian $3$-manifold of constant Webster scalar curvature. 

\begin{example}
\label{ex:topology}
The previous proposition suggests that one may construct $3$-space forms with non-trivial topology by taking a non-trivial subgroup $G$ of isometries of $\e$ such that the Sasakian structure and the group operation in $\e$ descend to the quotient $M:=\e/G$. In $\e$ with $\kappa\leq 0$ this can be done if we consider the subgroup $G\cong\mathbb{Z}$ of $\text{Iso}(\e)$ given by $G:=\{\phi_n\,;\,n\in\mathbb{Z}\}$, where $\phi_n(p):=p+(2n\pi)\,T_p$. The resulting space form is homeomorphic to the cylinder $\rr^2\times\mathbb{S}^1$. This construction also holds in $\mathbb{M}(\kappa)$ with $\kappa>0$ for the subgroup $G\cong\mathbb{Z}_2$ of $\text{Iso}(\e)$ given by the identity map and the antipodal map in $\mathbb{S}^3$. This leads to a sub-Riemannian model of the $3$-dimensional real projective space $\rr\mathbb{P}^3$.
\end{example} 

For a piecewise smooth curve $\alpha:I\to\m$ we define a \emph{lift} of $\alpha$ as a piecewise smooth curve $\widetilde{\alpha}:I\to\e$ such that $\mathcal{F}\circ\widetilde{\alpha}=\alpha$. From basic properties of principal bundles \cite[p.~88]{kobayashi} there is a unique lift with horizontal tangent vector wherever it exists (\emph{horizontal lift}) and passing through a given point of $\e$. The next lemma states that the holonomy displacement associated to a horizontal lift of a Jordan curve in $\m$ coincides with the area enclosed by the curve. The cases $\kappa=0,1$ are found in \cite[Sect.~1.3]{montgomery} and \cite[Proof of Prop.~1]{pinkall}. In general, the proof follows by adapting the arguments in \cite[Thm.~1, p.~191]{singer-thorpe}.

\begin{lemma}
\label{lem:holonomy}
Let $D\subset\m$ be a disk of area $A$ bounded by a piecewise smooth Jordan curve. Let $\alpha:[0,L]\to\m$ be a parameterization by arc-length of $\ptl D$ which reverses the orientation induced by $D$. If $\widetilde{\alpha}:[0,L]\to\e$ is a horizontal lift of $\alpha$, then the oriented distance between $\widetilde{\alpha}(0)$ and $\widetilde{\alpha}(L)$ along the fiber of $\mathcal{F}$ through $\widetilde{\alpha}(0)$ equals $2A$. 
\end{lemma}

\subsection{Carnot-Carath\'eodory geodesics and Jacobi fields}
\label{subsec:ccgeo}
\noindent

Let $M$ be a Sasakian sub-Riemannian $3$-manifold. A \emph{horizontal curve} in $M$ is a $C^1$ curve $\ga$ whose velocity vector $\dot{\ga}$ lies in $\h$. The \emph{length} of $\ga$ over a compact interval $I\sub\rr$ is given by $\int_I|\dot{\ga}(s)|\,ds$, where $|\dot{\ga}(s)|$ denotes the sub-Riemannian length of $\dot{\ga}(s)$. The \emph{Carnot-Carath\'eodory distance} between $p,q\in M$ is the infimum of the lengths of all $C^\infty$ horizontal curves joining $p$ and $q$. Since $\h$ is a bracket generating distribution such curves exist by Chow's connectivity theorem \cite[Sect.~1.2.B]{gromov-cc}. 

Suppose that $\ga$ is a $C^2$ horizontal curve parameterized by arc-length. We say that $\ga$ is a \emph{Carnot-Carath\'eodory geodesic}, or simply a \emph{$CC$-geodesic}, if $\ga$ is a critical point of length under $C^2$ variations by horizontal curves. Following variational arguments, see for example \cite[Prop. 3.1]{rr2}, we deduce that $\ga$ is a CC-geodesic if and only if there is a constant $\la\in\rr$, called the \emph{curvature} of $\ga$, such that $\ga$ satisfies the second order ODE
\begin{equation}
\label{eq:geoeq}
\dot{\ga}'+2\la\,J(\dot{\ga})=0,
\end{equation}
where the prime $'$ stands for the covariant derivative along $\ga$ in $(M,g)$. If $p\in M$ and $v\in\h_p$ with $|v|=1$, then the unique solution $\ga$ to \eqref{eq:geoeq} with $\ga(0)=p$ and $\dot{\ga}(0)=v$ is a CC-geodesic of curvature $\la$ since the functions $\escpr{\dot{\ga},T}$ and $|\dot{\ga}|^2$ are constant along $\ga$. The fact that any CC-geodesic is defined on $\rr$ is equivalent to the completeness of $M$ by \cite[Thm.~1.2]{falbel4}.

The next result shows the behaviour of CC-geodesics with respect to local isometries and the homothetic deformation in \eqref{eq:homothetic}. The proof comes from \eqref{eq:ga} and the Koszul formula \cite[p.~55]{dcriem}.

\begin{lemma}
\label{lem:geohomo}
Let $M$ be a Sasakian sub-Riemannian $3$-manifold and $\ga:I\to M$ a CC-geodesic of curvature $\la$ defined over an open interval $I\subeq\rr$. Then we have 
\begin{itemize}
\item[(i)] If $\phi:M\to M'$ is a local isometry then $\phi\circ\ga$ is a CC-geodesic of curvature $\la$ in $M'$.
\item[(ii)] For any $\eps>0$, the curve $\ga_\eps:I_\eps\to M$ defined on $I_\eps:=\eps I$ by $\ga_\eps(s):=\ga(s/\eps)$ is a CC-geodesic of curvature $\la/\eps$ for the sub-Riemannian metric $(g_h)_\eps:=\eps^2g_h$. 
\end{itemize}
\end{lemma} 

In a model space $\e$ any CC-geodesic of curvature $\la$ is a horizontal lift for the submersion $\mathcal{F}:\e\to\m$ of a curve with geodesic curvature $-2\la$ in $\m$, see \cite[Lem.~3.1]{rosales} and \cite[Thm.~1.26]{montgomery}. In this paper we will only need the following facts about CC-geodesics in $\e$ that can be proved without using their explicit expressions. 

\begin{proposition}
\label{prop:geomodel}
Let $\ga$ be a complete CC-geodesic in $\e$. Then we have
\begin{itemize}
\item[(i)] If $\kappa\leq 0$ then $\ga$ leaves any compact set of $\e$ in finite time.
\item[(ii)] If $\kappa>0$ then $\ga$ cannot be contained inside an open hemisphere of $\mathbb{S}^3$.
\item[(iii)] If $\kappa>0$ and $\ga$ is closed, then the length of $\ga$ is greater than or equal to $2\pi/\sqrt{\kappa}$. Moreover, equality holds if and only if $\ga$ is a horizontal great circle of $\mathbb{S}^3$. 
\end{itemize}
\end{proposition}

\begin{proof}
For any $\kappa\neq 0$ we can apply Lemma~\ref{lem:geohomo} (ii) to get that a CC-geodesic of curvature $\lambda$ and length $L$ in $\e$ is a reparameterization of a CC-geodesic of curvature $\lambda/\sqrt{|\kappa|}$ and length $L\,\sqrt{|\kappa|}$ in $\mathbb{M}(\kappa/|\kappa|)$. So, we only have to prove the claim for $\kappa=-1,0,1$. After a translation in the Lie group $(\e,*)$ we can suppose that $\ga(0)$ is the identity element $o$. 
By changing the orientation of $\ga$ we can admit that $\la\geq 0$. 

We first treat the case $\kappa=-1,0$. If $\la=0$ then $\ga$ is a horizontal Riemannian geodesic, i.e., a straight line containing $o$ inside the plane $t=0$. If $\la\neq 0$ then we know from \cite[Lem.~3.1]{rosales} that $\ga$ is a horizontal lift for the submersion $\mathcal{F}:\e\to\m$ of an arc $c_\la$ of a circle containing $o$. This arc meets $\ptl\m$ when $\kappa=-1$ and $|\la|\leq 1$. For any $z\in c_\la$, $z\neq o$, let $\Om_z\subset\m$ be the circular sector bounded by the arc $\widehat{oz}$ and the segment $\overline{oz}$. Then, the horizontal lift of $\widehat{oz}$ through $o$ is contained in $\ga$, and we deduce by Lemma~\ref{lem:holonomy} that the $t$-coordinate of the point of $\ga$ projecting onto $z$ is proportional to the area enclosed by $\Om_z$. Statement (i) then follows since this area tends to $\infty$ when the arc-length of $\widehat{oz}$ goes to $\infty$.

Suppose now that $\kappa=1$. Again by \cite[Lem.~3.1]{rosales} we can find a circle $c_\la$ in $\mathbb{N}(1)$, bounding a disk of area $A_0\leq\pi/2$, and such that $\ga$ is the horizontal lift of $c_\la$ through $o$. Let $\Ga$ be the oriented fiber of $\mathcal{F}$ containing $o$. From Lemma~\ref{lem:holonomy} the intersection $\ga\cap\Ga$ is a family of  points at a constant distance $2A_0\leq\pi$. Assume that $\ga$ is contained in an open hemisphere $E^+$ associated to a linear hyperplane $E$ of $\rr^4$. It is easy to check that $\Ga\cap E$ consists of two different points at distance $\pi$ along $\Ga$. Hence, it would be possible to find $p\in\ga\cap\Ga$ such that $p\in E^-$, a contradiction. This proves (ii). Finally, note that $\ga$ is closed if and only if there are $n,m\in\mathbb{N}$ such that $nA_0=m\pi$. From here we get $n\geq 2m$ since $A_0\leq\pi/2$. Moreover, we can take $n$ in such a way that the lengths $L$ and $L_0$ of $\ga$ and $c_\la$ satisfy $L=nL_0$. Hence the isoperimetric inequality in the surface $\mathbb{N}(1)$, which is a $2$-sphere of curvature $4$, implies 
\[
L^2=n^2L_0^2\geq n^2\,(4\pi A_0-4A_0^2)=4\pi^2m\,(n-m)\geq 4\pi^2m^2\geq4\pi^2,
\]
with equality if and only if $m=1$, $n=2$ and $A_0=\pi/2$, i.e., $\ga$ is a horizontal great circle. 
\end{proof}

It was proved in \cite[Lem.~3.5]{rr2} that the infinitesimal vector field associated to a $C^2$ family of CC-geodesics with the same curvature in the Heisenberg group $\mathbb{M}(0)$ satisfies a second order ODE analogous to the Jacobi equation in Riemannian geometry. The following result in \cite[Lem.~3.3]{rosales} establishes the same fact  under weaker regularity conditions in arbitrary Sasakian $3$-manifolds. 

\begin{lemma}
\label{lem:ccjacobi}
Let $M$ be a Sasakian sub-Riemannian $3$-manifold. We consider a $C^1$ curve $\alpha:I\to M$ defined on some open interval $I\subseteq\rr$, and a $C^1$ unit horizontal vector field $U$ along $\alpha$. For fixed $\lambda\in\rr$, suppose that we have a well-defined map $F:I\times I'\to M$ given by $F(\eps,s):=\ga_{\eps}(s)$, where $I'$ is an open interval containing the origin and $\ga_{\eps}(s)$ is the CC-geodesic of curvature $\la$ with initial conditions $\ga_{\eps}(0)=\alpha(\eps)$ and $\dot{\ga}_{\eps}(0)=U(\eps)$. Then, the vector field $V_{\eps}(s):=(\ptl F/\ptl\eps)(\eps,s)$ satisfies the following properties:
\begin{itemize}
\item[(i)] $V_\eps$ is $C^\infty$ along $\ga_\eps$ with $[\dot{\ga}_\eps,V_\eps]=0$,
\item[(ii)] the function $\la\,\escpr{V_\eps,T}+\escpr{V_\eps,\dot{\ga}_\eps}$
is constant along $\ga_\eps$, 
\item[(iii)] $V_\eps$ satisfies the $CC$-Jacobi equation
\[
V_{\eps}''+R(\dot{\ga}_{\eps},V_{\eps})\dot{\ga}_{\eps}
+2\la\,\big(J(V'_{\eps})-\escpr{V_{\eps},\dot{\ga}_{\eps}}\,T\big)=0, 
\]
where $R$ is the curvature tensor in $(M,g)$ and the prime $'$ denotes the covariant derivative along the geodesic $\ga_{\eps}$,  
\item[(iv)] the vertical component of $V_\eps$ satisfies
\begin{align}
\label{eq:jacobi1}
\escpr{V_\eps,T}'&=2\,\escpr{V_\eps,J(\dot{\ga}_\eps)},
\\
\label{eq:jacobi2}
\escpr{V_\eps,T}'''&+4\,(\lambda^2+K)\,\escpr{V_\eps,T}'=0,
\end{align}
where $K$ is the Webster scalar curvature of $M$ and $'$ is the derivative with respect to $s$.  
\end{itemize}
\end{lemma}

\begin{remark}
\label{re:ricase}
Following \cite[Lem.~2.1]{hrr} we can prove similar properties for the vector field $V_\eps(s):=(\ptl F/\ptl\eps)(\eps,s)$ associated to a variation $F(\eps,s):=\ga_\eps(s)$, where $\ga_\eps(s)$ is the Riemannian geodesic in $(M,g)$ with $\ga_\eps(0)=\alpha(\eps)$ and $\dot{\ga}_\eps(0)=U(\eps)$. In this case, the second order equation in Lemma~\ref{lem:ccjacobi} (iii) becomes the classical Jacobi equation
\begin{equation}
\label{eq:rijacobi}
V_{\eps}''+R(\dot{\ga}_{\eps},V_{\eps})\dot{\ga}_{\eps}=0.
\end{equation}
\end{remark}

If $\ga$ is a CC-geodesic in $M$ then we define a \emph{CC-Jacobi field} along $\ga$ as a solution to the CC-Jacobi equation in Lemma~\ref{lem:ccjacobi} (iii). If the Webster scalar curvature of $M$ is constant along $\ga$ then, an easy integration from \eqref{eq:jacobi2}, gives us the following result.

\begin{lemma}
\label{lem:jacobicon}
Let $\ga(s)$ be a $CC$-geodesic of curvature $\la$ in a Sasakian sub-Riemannian $3$-manifold $M$. Let $V$ be the $CC$-Jacobi field associated to a variation of $\ga$ by $CC$-geodesics of curvature $\la$ as in Lemma~\ref{lem:ccjacobi}. Suppose that the Webster scalar curvature of $M$ is a constant $\kappa$ along $\ga$. If we denote $\tau:=4\,(\la^2+\kappa)$ and $v(s):=\escpr{V,T}(s)$, then we have
\begin{itemize}
\item[(i)] For $\tau<0$, 
\[
v(s)=\frac{1}{\sqrt{-\tau}}\left(a\,\sinh(\sqrt{-\tau}\,s)+b\,\cosh(\sqrt{-\tau}\,s)\right)+c,
\]
where $a:=v'(0)$, $b:=(1/\sqrt{-\tau})\,v''(0)$ and $c:=v(0)+(1/\tau)\,v''(0)$.
\item[(ii)] For $\tau=0$, 
\[
v(s)=as^2+bs+c,
\]
where $a=(1/2)\,v''(0)$, $b=v'(0)$ and $c=v(0)$.
\item[(iii)] For $\tau>0$,   
\[
v(s)=\frac{1}{\sqrt{\tau}}\,\big(a\,\sin(\sqrt{\tau}\,s)-b\,\cos(\sqrt{\tau}\,s)\big)+c,
\]
where $a=v'(0)$, $b=(1/\sqrt{\tau})\,v''(0)$ and $c=v(0)+(1/\tau)\,v''(0)$.
\end{itemize}
\end{lemma}

\begin{remark}
\label{re:mm}
Previous work on the CC-geodesics equation, the second variation of length and CC-Jacobi fields in contact sub-Riemannian $3$-manifolds is found in \cite{hughen} and \cite{rumin}. In \cite{chanillo-yang}, Chanillo and Yang studied the CC-Jacobi fields associated to $C^2$ families of CC-geodesics of the same curvature leaving from the same point. 
\end{remark}

\subsection{Surfaces and singular set}
\label{subsec:surfaces}
\noindent

Let $M$ be a contact sub-Riemannian $3$-manifold and $\Sg$ a $C^1$ surface immersed in $M$. Unless explicitly stated we shall always assume that $\ptl\Sg=\emptyset$. The \emph{singular set} of $\Sg$ is the set $\Sg_0:=\{p\in\Sg\,;\,T_p\Sg=\h_p\}$. Since $\h$ is a completely nonintegrable distribution, it follows by the Frobenius theorem that $\Sg_0$ is closed and has empty interior in $\Sg$. Hence the \emph{regular set} $\Sg-\Sg_0$ of $\Sg$ is open and dense in $\Sg$.  By using the arguments in \cite[Lem.~1]{d2}, see also \cite[Thm.~1.2]{balogh} and \cite[App.~A]{hp2}, we can deduce that, for a $C^2$ surface $\Sg$, the Hausdorff dimension of $\Sg_{0}$ in $(M,g)$ is less than or equal to $1$.  In particular, the Riemannian area of $\Sg_{0}$ vanishes. 

If $\Sg$ is orientable and $N$ is a unit vector normal to $\Sg$ in $(M,g)$, then we define the \emph{area} of $\Sg$ by
\begin{equation}
\label{eq:area}
A(\Sg):=\int_{\Sg}|N_{h}|\,d\Sg,
\end{equation}
where $d\Sg$ is the Riemannian area element. If $\Sg$ bounds a set $\Om\subset M$ then the Riemannian divergence theorem implies that $A(\Sg)$ coincides with the \emph{sub-Riemannian perimeter \'a la De Giorgi} of $\Om$, which can be introduced by following \cite{cdg1}, see equation \eqref{eq:per} for a precise definition. It is clear that $\Sg_{0}=\{p\in\Sg\,;\,N_h(p)=0\}$, where $N_{h}:=N-\escpr{N,T}\,T$. Hence, the differentiation of the area functional could be problematic when $\Sg_0\neq\emptyset$. In the regular part $\Sg-\Sg_0$, we can define the \emph{horizontal Gauss map} $\nu_h$ and the \emph{characteristic vector field} $Z$, by
\begin{equation}
\label{eq:nuh}
\nu_h:=\frac{N_h}{|N_h|}, \qquad Z:=J(\nuh).
\end{equation}
As $Z$ is horizontal and orthogonal to $\nu_h$, then $Z$ is tangent to $\Sg$.  Hence $Z_{p}$ generates $T_{p}\Sg\cap\h_{p}$ for any $p\in\Sg-\Sg_0$. The integral curves of $Z$ in $\Sg-\Sg_0$ are called $(\emph{oriented}\,)$ \emph{characteristic curves} of $\Sg$.  They are
both tangent to $\Sg$ and to $\h$.  If we define
\begin{equation}
\label{eq:ese}
S:=\escpr{N,T}\,\nu_h-|N_h|\,T,
\end{equation}
then $\{Z_{p},S_{p}\}$ is an orthonormal basis of $T_p\Sg$ whenever $p\in\Sg-\Sg_0$.  Moreover, for any $p\in\Sg-\Sg_{0}$ we have the orthonormal basis of $T_{p}M$ given by $\{Z_{p},(\nuh)_{p},T_{p}\}$.  From here we deduce the following identities on $\Sg-\Sg_{0}$
\begin{equation}
\label{eq:relations}
|N_{h}|^2+\escpr{N,T}^2=1, \quad (\nu_{h})^\top=\escpr{N,T}\,S, 
\quad T^\top=-|N_{h}|\,S,
\end{equation}
where $U^\top$ stands for the projection of a vector field $U$ onto the tangent plane to $\Sg$.

If $\Sg$ is an oriented $C^2$ surface immersed in $M$ then, for any $p\in\Sg-\Sg_0$ and $v\in T_p M$, the following equalities hold (see \cite[Lem.~3.5]{hrr} and \cite[Lem.~4.2]{rosales}): 
\begin{align}
\label{eq:vmnh}
v\,(|N_h|)&=\escpr{D_{v}N, \nu_{h}}+\escpr{N,T}\,
\escpr{v,Z},
\\
\label{eq:vnt}
v(\escpr{N,T})&=\escpr{D_{v}N,T}+\escpr{N,J(v)}, \\
\label{eq:dvnuh}
D_{v}\nu_h&=|N_h|^{-1}\, \big(\escpr{D_vN,Z}-\escpr{N,T}\,
\escpr{v,\nuh}\big)\,Z+\escpr{v,Z}\,T.
\end{align}

We denote by $B$ the \emph{shape operator} of $\Sg$ in $(M,g)$.  It is defined, for any vector $U$ tangent to $\Sg$, by $B(U):=-D_{U}N$. We will say that $\Sg$ is \emph{complete} if it is complete with respect to the Riemannian metric induced by $(M,g)$.

\subsection{Differentiation under the integral sign}
\label{subsec:difint}
\noindent

In the next lemma we recall the classical Leibniz's rule for differentiating the integral of a function which depends on a real parameter. This will be used in Sections~\ref{sec:1stvar} and \ref{sec:stability} to compute the derivatives of the sub-Riemannian volume and area functionals. A proof can be obtained by following the arguments in \cite[Thm.~2.27]{folland}.

\begin{lemma}
\label{lem:difint}
Let $(\Om,\mu)$ be a measure space and $I\subseteq\rr$ an open interval. Consider a real valued function $f:I\times\Om\to\rr$ such that the following conditions hold:
\begin{itemize}
\item[(i)] for any $s\in I$, the function $f(s,\cdot):\Om\to\rr$ is integrable,
\item[(ii)] the function $f(\cdot,x):I\to\rr$ is absolutely continuous for almost every $x\in\Om$,
\item[(iii)] there is an integrable function $h:\Om\to\rr$ such that 
$|(\ptl f/\ptl s)(s,x)|\leq h(x)$ for almost every $s\in I$ and almost every  $x\in\Om$,
\item[(iv)] there is $s_0\in I$ such that $(\ptl f/\ptl s)(s_0,x)$ exists for almost every $x\in\Om$.
\end{itemize} 
Then, for the function $A:I\to\rr$ defined by $A(s):=\int_\Om f(s,x)\,d\mu(x)$, the derivative $A'(s_0)$ exists, and we have $A'(s_0)=\int_\Om (\ptl f/\ptl s)(s_0,x)\,d\mu(x)$.
\end{lemma}

\begin{example}
Here we show that Lemma~\ref{lem:difint} need not hold if condition (ii) fails. Consider the function $f:(-1,1)\times (-1,1)\to\rr$ given by
\[
f(s,x):=
\begin{cases}
-1, \quad -1<s<x,
\\
1,\hspace{0.87cm} x<s<1.
\end{cases}
\]
Clearly $A(s)=\int_{-1}^1 f(s,x)\,dx=2s$, so that $A'(s)=2$ for any $s\in(-1,1)$. Note that the conditions (i), (iii) and (iv) in Lemma~\ref{lem:difint} are satisfied. However, differentiation under the integral sign fails; otherwise we would obtain $A'(s)=0$ for any $s\in (-1,1)$, a contradiction.
\end{example}

\section{The spherical surfaces $\s_\la(p)$}
\label{sec:spheres}

Pansu conjectured in \cite[p.~172]{pansu1} that any isoperimetric surface in the first Heisenberg group $\mm(0)$ is congruent to a  sphere of revolution $\s_\la$ obtained by vertical rotations of a CC-geodesic of curvature $\la\neq 0$ connecting two points $p$ and $p_\la$ of the vertical axis. Since vertical rotations are isometries of $\mm(0)$, it follows from Lemma~\ref{lem:geohomo} (i) that $\s_\la$ coincides with the union of all the CC-geodesics of curvature $\la$ joining $p$ and $p_\la$. This approach was also employed in \cite[Sect.~5.1]{hr1} to produce spherical surfaces in the sub-Riemannian $3$-sphere $\mm(1)$. In this section we extend the construction inside any $3$-dimensional space form $M$. For that, we must study first when the CC-geodesics of a fixed curvature $\la\in\rr$ leaving from a given point $p\in M$ meet again at another point $p_\la\in M$. 

\begin{lemma}
\label{lem:condition}
Let $M$ be a $3$-dimensional space form of Webster scalar curvature $\kappa$. Fix a point $p\in M$ and a number $\la\in\rr$. Then, all the CC-geodesics of curvature $\la$ leaving from $p$ meet again in $M$ if and only if $\la^2+\kappa>0$. In such a case, the length of any CC-geodesic of curvature $\la$ between $p$ and the first cut point $p_\la$ equals $\pi/\sqrt{\la^2+\kappa}$.
\end{lemma}

\begin{proof}
Fix a positive orthonormal basis $\{e_1,e_2\}$ in the contact plane $\h_p$. For any $\theta\in\rr$, let $\ga_\theta:\rr\to M$ be the CC-geodesic  of curvature $\la$ with $\ga_\theta(0)=p$ and $\dot{\ga}_\theta(0)=U(\theta)$, where $U(\theta):=(\cos\theta)\,e_1+(\sin\theta)\,e_2$. The $C^\infty$ map $F:\rr^2\to M$ defined by $F(\theta,s):=\ga_\theta(s)$ is the flow of the CC-geodesics of curvature $\la$ leaving from $p$. We denote $V_\theta(s):=(\ptl F/\ptl\theta)(\theta,s)$. This is a CC-Jacobi field along $\ga_\theta$ by Lemma~\ref{lem:ccjacobi} (iii). Let $v_\theta:=\escpr{V_\theta,T}$. By using Lemma~\ref{lem:ccjacobi} (ii), equation \eqref{eq:jacobi1} and that $V_\theta(0)=0$, we deduce
\begin{equation}
\label{eq:vtheta}
V_\theta=-(\la\,v_\theta)\,\dot{\ga}_\theta+(v'_\theta/2)\,J(\dot{\ga}_\theta)+v_\theta\,T.
\end{equation}
Evaluating at $s=0$ we get $v_\theta'(0)=0$. We compute $V'_\theta(0)$ from \eqref{eq:vtheta}, so that we obtain
\[
V_\theta'(0)=\frac{v''_\theta(0)}{2}\,J\big(U(\theta)\big).
\]
On the other hand, $V_\theta'=D_{\dot{\ga}_\theta}V_\theta=D_{V_\theta}\dot{\ga}_\theta$ since $[\dot{\ga}_\theta,V_\theta]=0$ along $\ga_\theta$. Therefore
\[
V_\theta'(0)=\frac{d}{d\theta}\bigg|_\theta\,\dot{\ga}_\theta(0)=-(\sin\theta)\,e_1+(\cos\theta)\,e_2=J\big(U(\theta)\big),
\]
so that $v''_\theta(0)=2$. Since $M$ has constant Webster scalar curvature, we apply Lemma~\ref{lem:jacobicon} to get
\begin{equation}
\label{eq:vthetacases}
v_\theta(s)=
\begin{cases}
\frac{1}{2(\la^2+\kappa)}\,\big(1-\cosh(2\,\sqrt{-(\la^2+\kappa)}\,s)\big), \ \text{ if } \ 
\la^2+\kappa<0,
\\
s^2, \hspace{5.20cm} \text{ if } \ \la^2+\kappa=0,
\\
\frac{1}{2(\la^2+\kappa)}\,\big(1-\cos(2\,\sqrt{\la^2+\kappa}\,s)\big),
\hspace{0.92cm} \text{ if } \ \la^2+\kappa>0.
\end{cases}
\end{equation}
It follows that $v_\theta$ does not depend on $\theta$, and that $v_\theta(s)\neq 0$ for any $s\neq 0$ if $\la^2+\kappa\leq 0$. For $\la^2+\kappa>0$ it is clear that the first $s_0>0$ for which $v_\theta(s_0)=0$ is given by $s_0:=\pi/\sqrt{\la^2+\kappa}$.

Let us prove the statement. If $\la^2+\kappa>0$ then \eqref{eq:vtheta} implies $V_\theta(s_0)=0$ for any $\theta\in\rr$ since $v_\theta(s_0)=v_\theta'(s_0)=0$. This is equivalent to the existence of a first point $p_\la\in M$ such that $\ga_\theta(s_0)=p_\la$ for any $\theta\in\rr$. Conversely, suppose that all the CC-geodesics $\ga_\theta$ meet again at $p_\la\in M$. This means that there is a $C^1$ function $s(\theta)>0$ such that $p_\la=\ga_\theta(s(\theta))=F(\theta,s(\theta))$ for any $\theta\in\rr$. Note that $s(\theta)$ is the length of $\ga_\theta$ between $p$ and $p_\la$. By taking derivatives with respect to $\theta$, we obtain
\[
0=V_\theta\big(s(\theta)\big)+s'(\theta)\,\dot{\ga}_\theta\big(s(\theta)\big),
\]
and so the vertical component $v_\theta$ of $V_\theta$ vanishes at $s(\theta)$. From the explicit computation of $v_\theta$ in \eqref{eq:vthetacases}, we conclude that $\la^2+\kappa>0$ and $s(\theta)$ must be constant. Furthermore, the first cut point of the CC-geodesics $\ga_\theta$ is attained for $s_0=\pi/\sqrt{\la^2+\kappa}$.
\end{proof}

The previous lemma motivates the following definition. Let $M$ be a $3$-dimensional space form of Webster scalar curvature $\kappa$. For any $p\in M$ and any $\la\geq 0$ such that $\la^2+\kappa>0$, we define the \emph{spherical surface} $\s_\la(p)$ as the set
\begin{equation}
\label{eq:sla}
\s_\la(p):=\big\{\ga_v(s)\,;\,v\in\h_p,\,|v|=1,\,0\leq s\leq\pi/\sqrt{\la^2+\kappa}\big\},
\end{equation}
where $\ga_v:\rr\to M$ is the CC-geodesic of curvature $\la$ with $\ga_v(0)=p$ and $\dot{\ga}_v(0)=v$. The \emph{south and north poles} of $\s_\la(p)$ are respectively given by $p$ and $p_\la:=\ga_v(\pi/\sqrt{\la^2+\kappa})$ (which does not depend on $v\in\h_p$ with $|v|=1$).  For further reference we define the \emph{equator} of $\sla(p)$ by $\{\ga_v\big(\pi/(2\,\sqrt{\la^2+\kappa})\big)\,;\,v\in\h_p,\,|v|=1\}$ and the \emph{open hemispheres} as the sets
\begin{align*}
\s_\la(p)^-&:=\big\{\ga_v(s)\,;\,v\in\h_p,\,|v|=1,\,0\leq s<\pi/(2\,\sqrt{\la^2+\kappa})\big\},
\\
\s_\la(p)^+&:=\big\{\ga_v(s)\,;\,v\in\h_p,\,|v|=1,\,\pi/(2\,\sqrt{\la^2+\kappa})<s\leq\pi/\sqrt{\la^2+\kappa}\big\}.
\end{align*}

As an immediate consequence of Lemma~\ref{lem:geohomo} we can deduce the behaviour of $\s_\la(p)$ with respect to local isometries and the homothetic deformation in \eqref{eq:homothetic}.

\begin{lemma}
\label{lem:slaiso}
Let $M$ be a $3$-dimensional space form of Webster scalar curvature $\kappa$. For any $p\in M$ and any $\la\geq 0$ such that $\la^2+\kappa>0$, we have
\begin{itemize}
\item[(i)] If $\phi:M\to M'$ is a local isometry then $\phi\big(\s_\la(p)\big)=\s'_\la\big(\phi(p)\big)$.
\item[(ii)] For any $\eps>0$, the set $\s_\lambda(p)$ associated to the sub-Riemannian metric $g_h$ coincides with the set $\s_{\la/\eps}(p)$ for the sub-Riemannian metric $(g_h)_\eps:=\eps^2 g_h$. 
\end{itemize}
\end{lemma}

Our main goal in this section is to prove some properties of the sets $\s_\la(p)$ regarding their symmetry, embeddeness, topology and smoothness. They are gathered in the next result.

\begin{theorem}
\label{th:sla}
Let $M$ be a $3$-dimensional space form of Webster scalar curvature $\kappa$. Fix a point $p\in M$ and a number $\la\geq 0$ such that $\la^2+\kappa>0$. Then we have
\begin{itemize} 
\item[(i)] The set $\s_\la(p)$ is invariant under any $\phi\in\text{Iso}(M)$ with $\phi(p)=p$. Moreover, if $M$ is homogeneous, then any two sets $\s_\la(p)$ and $\s_\la(q)$ are congruent.
\item[(ii)] If $M=\e$ then $\s_\la(p)$ is an embedded $C^2$ surface congruent to a sphere of revolution with respect to the vertical axis.
\item[(iii)] The set $\s_\la(p)$ is a $C^2$ sphere immersed in $M$. This sphere is also $C^\infty$ off of the singular set, which consists of the south and north poles. 
\item[(iv)] The sphere $\s_\la(p)$ is always Alexandrov embedded. Moreover, there is $\la_0=\la_0(p)\geq 0$ such that $\la_0^2+\kappa>0$ and $\s_\la(p)$ is embedded for any $\la>\la_0$.
\end{itemize}
\end{theorem}

\begin{remark}
A compact surface $\Sg$ immersed in $M$ is \emph{Alexandrov embedded} if we can extend the immersion $\varphi_0:\Sg\to M$ to an immersion $\widetilde{\varphi}_0:\mathcal{D}\to M$ of a compact $3$-manifold $\mathcal{D}$ with $\ptl\mathcal{D}=\Sg$.
\end{remark}

\begin{proof}[Proof of Theorem~\ref{th:sla}]
Statement (i) is an immediate consequence of Lemma~\ref{lem:slaiso} (i). By using that $\e$ is homogeneous and Lemma~\ref{lem:slaiso} (ii) it suffices to prove (ii) when $\kappa=-1,0,1$ and $p$ equals the identity element $o$ in $(\e,*)$. We shall denote the associated set by $\s^\kappa_\la(o)$.

The cases $\kappa=0,1$ are well known, see for example \cite[Ex.~3.3]{rr2} and \cite[Prop.~5.1, Thm.~6.4]{hr1}. For the remaining case $\kappa=-1$ we proceed as follows. Fix $\la>1$. Let $\ga$ be a CC-geodesic in $\mathbb{M}(-1)$ of curvature $\la$ and length $\pi/\sqrt{\la^2-1}$ leaving from $o$. By using \cite[Lem.~3.1]{rosales} or \cite[Thm.~1.26]{montgomery} we deduce that the Euclidean projection of $\ga$ onto the hyperbolic plane $\mathbb{N}(-1)$ is a clockwise parameterized circle $c_\la$ of radius $1/(2\la)$. Take a point $q:=(z,t)\in\ga$. Let $\Om_z\subset\mathbb{N}(-1)$ be the circular sector bounded by the arc $\widehat{oz}$ of $c_\la$ and the segment $\overline{oz}$. By Lemma~\ref{lem:holonomy} we get that $t/2$ equals the area of $\Om_z$ in $\mathbb{N}(-1)$. To compute this area as a function of $|z|$ we can use the Gauss-Bonnet formula.  It follows that $\s^\kappa_\la(o)$ is the union of two radial graphs over the horizontal disk $D_\la:=\{z\in\mathbb{R}^2\,;\,|z|\leq 1/\la\}$. In fact, a suitable vertical translation maps $\s^\kappa_\la(o)$ into $\{(z,t)\in D_\la\times\rr\,;\,t=\pm f(|z|)\}$, where $f:[0,1/\la]\to\rr$ is the function given by 
\[
f(x):=\frac{\pi}{2}\left(1-\frac{\la}{\sqrt{\la^2-1}}\right)+\frac{\la}{\sqrt{\la^2-1}}\,\arcsin\left(\frac{\sqrt{\la^2-1}\,x}{\sqrt{1-x^2}}\right)-\arctan\left(\frac{\la\,x}{\sqrt{1-\la^2 x^2}}\right).
\]
From here it is easy to check that $\s^\kappa_\la(o)$ is a $C^2$ embedded sphere of revolution with singular points only at $(0,\pm f(0))$.  Moreover, $\s^\kappa_\la(o)$ is $C^\infty$ off of these points, where it fails to be $C^3$. This is also known in $\mathbb{M}(0)$, see the remark after Thm.~8.16 in \cite{survey}, and in $\mathbb{M}(1)$ for $\la\neq 0$, see \cite[Rem.~5.2]{hr1}.

Now we can prove (iii) and (iv). We apply Proposition~\ref{prop:model} to the $3$-dimensional space form $M$, so that we find a surjective local isometry $\Pi:\e\to M$. Given a point $p_0\in\e$ with $\Pi(p_0)=p$, we have by Lemma~\ref{lem:slaiso} (i) that $\s_\la(p)=\Pi\big(\s^\kappa_\la(p_0)\big)$. This allows to deduce (iii) from (ii). Moreover, $\s_\la(p)$ is Alexandrov embedded since the sphere $\s^\kappa_\la(p_0)$ bounds a topological $3$-ball in $\e$. Finally, the fact that $\Pi:\e\to M$ is a local isometry implies the existence of an open ball $\mathcal{B}:=B^\kappa(p_0,R_0)$ for the Carnot-Carath\'eodory distance in $\e$ such that $\Pi:\mathcal{B}\to\Pi(\mathcal{B})$ is an isometry. We denote by $\la_0$ the unique non-negative number with $\la^2_0+\kappa>0$ and $\pi/\sqrt{\la^2_0+\kappa}=R_0$. By definition of CC-distance it is clear that $\s^\kappa_\la(p_0)\subset\mathcal{B}$ for any $\la>\la_0$. In particular, the sphere $\s_\la(p)=\Pi\big(\s^\kappa_\la(p_0)\big)$ is embedded for any $\la>\la_0$. This completes the proof.
\end{proof}

\begin{remarks}
\label{re:pozo}
1. In a homogeneous Sasakian $3$-manifold $M$ the number $\la_0=\la_0(p)$ in statement (iv) does not depend on $p$. So, the spheres $\s_\la(p)$ are embedded for any $\la>\la_0$ and any $p\in M$. 

2. If $M$ has non-trivial topology, then the spheres $\s_\la(p)$ need not be embedded for all values of $\la$. To illustrate this we consider a Sasakian cylinder $M:=\e/G$ as the ones in Example~\ref{ex:topology}. Take a sphere of revolution $\s_\la^\kappa(o)$ in $\e$ with poles identified under the action of $G$. If we denote $p:=\Pi(o)$ then $\s_\la(p)=\Pi(\s_\la^\kappa(o))$, so that $\s_\la(p)$ is immersed in $M$ but not embedded. 

3. The arguments in Theorem~\ref{th:sla} allow also the construction of immersed planes inside any $3$-dimensional space form $M$.  Given $p\in M$ and $\la\geq 0$ such that $\la^2+\kappa\leq 0$, we define the set
\begin{equation}
\label{eq:lla}
\mathcal{L}_\la(p):=\big\{\ga_v(s)\,;\,v\in\h_p,\,|v|=1,\,s\geq 0\},
\end{equation}
where $\ga_v:\rr\to M$ is the CC-geodesic of curvature $\la$ with $\ga_v(0)=p$ and $\dot{\ga}_v(0)=v$. By equation \eqref{eq:vthetacases} we know that the vertical component of the CC-Jacobi field given in \eqref{eq:vtheta} never vanishes for $s>0$. Hence, the set $\mathcal{L}_\la(p)$ is a complete $C^2$ plane immersed in $M$ whose singular set is $\{p\}$. 

4. It is natural to ask if our construction of $\sla(p)$ would provide spherical surfaces inside an arbitrary complete, contact, sub-Riemannian $3$-manifold $M$ of constant Webster scalar curvature. Observe that, if $M$ is not Sasakian, then the characteristic curves of a CMC surface are not necessarily CC-geodesics, but they still satisfy a second order ODE similar to \eqref{eq:geoeq}, see \cite[Prop.~4.1]{galli}. Hence we may consider the set $\s_\la(p)$ of all the solutions to such ODE leaving from $p$ and study, as in Lemma~\ref{lem:condition}, when these solutions meet again. However, the analysis of the associated CC-Jacobi fields is much more involved since their vertical component does not only depend on the Webster scalar curvature but also on the Webster-Tanaka torsion \cite[Prop.~4.3]{galli}.
\end{remarks}

We conclude this section by introducing \emph{polar coordinates} inside the spherical surfaces $\s_\la(p)$ and giving some computations in such coordinates that will be helpful in the sequel. 

\begin{lemma}
\label{lem:slaF}
Let $M$ be a $3$-dimensional space form of Webster scalar curvature $\kappa$. Given a point $p\in M$, a positive orthonormal basis $\{e_1,e_2\}$ in $\h_p$, and a number $\la\geq 0$ with $\la^2+\kappa>0$, we define the map $F:\rr^2\to M$ by $F(\theta,s):=\ga_\theta(s)$, where $\ga_\theta$ is the CC-geodesic of curvature $\la$ with $\ga_\theta(0)=p$ and $\dot{\ga}_\theta(0)=(\cos\theta)\,e_1+(\sin\theta)\,e_2$. Denote $\tau:=\sqrt{\la^2+\kappa}$. The following properties are satisfied: 
\begin{itemize}
\item[(i)] the restriction of $F$ to $C_\la:=[0,2\pi]\times(0,\pi/\tau)$ is a $C^\infty$ immersion whose range is $\s_\la(p)$ minus the poles,
\item[(ii)] there is a Riemannian unit normal $N$ along $\s_\la(p)$ such that any CC-geodesic $\ga_\theta(s)$ with $s\in(0,\pi/\tau)$ is a characteristic curve of $\s_\la(p)$, 
\item[(iii)] with respect to the coordinates $(\theta,s)\in C_\la$, these equalities hold
\begin{align*}
d\sla(p)&=\frac{\sin(\tau s)}{\tau^2}\,\sqrt{1+(\tau^2-1)\cos^2(\tau s)}\ d\theta\,ds,
\\
\mnh&=\frac{\sin(\tau s)}{\sqrt{1+(\tau^2-1)\cos^2(\tau s)}}, 
\hspace{3cm} \escpr{N,T}=\frac{\tau\,\cos(\tau s)}{\sqrt{1+(\tau^2-1)\cos^2(\tau s)}},
\\
\escpr{B(Z),Z}&=2\la\,\mnh, \quad \escpr{B(Z),S}=(1-\tau^2)\,\mnh^2,
\quad \escpr{B(S),S}=\frac{\la\,\tau^2\,\mnh}{1+(\tau^2-1)\cos^2(\tau s)},
\end{align*}
where $d\sla(p)$ is the Riemannian area element, $\{Z,S\}$ is the orthonormal basis in \eqref{eq:nuh} and \eqref{eq:ese}, and $B$ is the Riemannian shape operator with respect to $N$.
\end{itemize}
\end{lemma}

\begin{proof}
The map $F:\rr^2\to M$ is $C^\infty$ since, for fixed $\la$, the solutions of \eqref{eq:geoeq} depends differentiably ($C^\infty$) on the initial data. Note that $(\ptl F/\ptl s)(\theta,s)=\dot{\ga}_\theta(s)$, which is a horizontal vector. On the other hand, $V_\theta(s):=(\ptl F/\ptl\theta)(\theta,s)$ is a CC-Jacobi field along $\ga_\theta$ by Lemma~\ref{lem:ccjacobi} (iii). From the computations in the proof of Lemma~\ref{lem:condition} we know that
\[
V_\theta=-(\la\,v)\,\dot{\ga}_\theta+
(v'/2)\,J(\dot{\ga}_\theta)+v\,T,
\]
where the prime $'$ denotes the derivative with respect to $s$, and $v$ is the function given by
\begin{equation}
\label{eq:slav}
v(s):=\frac{1-\cos(2\tau s)}{2\tau^2}=\frac{\sin^2(\tau s)}{\tau^2}.
\end{equation}
The fact that $v(s)>0$ for $s\in (0,\pi/\tau)$ implies that the differential of $F$ has rank two for any $(\theta,s)\in C_\la$, and so $F:C_\la\to M$ is a $C^\infty$ immersion. By the definition of $\s_\la(p)$ in \eqref{eq:sla} it is clear that $F(C_\la)=\s_\la(p)-\{p,p_\la\}$. This proves (i). Now, we define 
\begin{equation}
\label{eq:slan}
N:=\frac{-v\,J(\dot{\ga}_\theta)+(v'/2)\,T}{\sqrt{v^2+(v'/2)^2}},
\end{equation}
which provides a Riemannian unit normal along $S_\la(p)-\{p,p_\la\}$. By using \eqref{eq:nuh} we get $\nuh=-J(\dot{\ga}_\theta)$ and $Z=\dot{\ga}_\theta$. This means that any $\ga_\theta(s)$ with $s\in (0,\pi/\tau)$ is a characteristic curve of $\s_\la(p)$, proving statement (ii). On the other hand, it is clear that
\begin{align*}
d\sla(p)&=\sqrt{|V_\theta|^2-\la^2 v^2} \ d\theta\,ds=
\sqrt{v^2+(v'/2)^2} \ d\theta\,ds
\\
&=\frac{\sin(\tau s)}{\tau^2}\,\sqrt{\sin^2(\tau s)+\tau^2\cos^2(\tau s)}\ d\theta\,ds=\frac{\sin(\tau s)}{\tau^2}\,\sqrt{1+(\tau^2-1)\cos^2(\tau s)}\ d\theta\,ds.
\end{align*}
Moreover, the expressions of $\mnh$ and $\escpr{N,T}$ in statement (iii) can be obtained from \eqref{eq:slan} and \eqref{eq:slav}. Note that $0=Z(\escpr{N,Z})=\escpr{D_ZN,Z}+\escpr{N,D_ZZ}$. From equation \eqref{eq:geoeq}, we get
\[
\escpr{B(Z),Z}=\escpr{N,D_ZZ}=-2\la\,\escpr{N,J(Z)}=2\la\,\escpr{N,\nuh}=2\la\,\mnh.
\]
By taking into account \eqref{eq:vnt}, we obtain
\begin{align*}
Z(\escpr{N,T})&=\escpr{D_ZN,T}+\escpr{N,D_ZT}=\mnh\,\escpr{B(Z),S}-\mnh,
\\
S(\escpr{N,T})&=\escpr{D_SN,T}+\escpr{N,D_ST}=\mnh\,\escpr{B(S),S},
\end{align*}
since the tangent projection of $T$ is $-\mnh\,S$, $D_ZT=-\nuh$ and $D_ST=\escpr{N,T}\,Z$. It follows that
\begin{align*}
\escpr{B(Z),S}&=\frac{Z(\escpr{N,T})}{\mnh}+1=\frac{\escpr{N,T}'}{\mnh}+1=\frac{(1-\tau^2)\,\sin^2(\tau s)}{1+(\tau^2-1)\cos^2(\tau s)}
=(1-\tau^2)\,\mnh^2,
\\
\escpr{B(S),S}&=\frac{S(\escpr{N,T})}{\mnh}=-\la\,Z(\escpr{N,T})=-\la\,\escpr{N,T}'=\frac{\la\,\tau^2\sin(\tau s)}{\big(1+(\tau^2-1)\cos^2(\tau s)\big)^{3/2}},
\end{align*}
where in the second equation we have used that $\escpr{N,T}$ does not depend on $\theta$, and that 
\begin{equation}
\label{eq:esepolar}
S=\frac{-1}{\sqrt{v(s)^2+(v'(s)/2)^2}}\,V_\theta-\la\,\mnh\,Z.
\end{equation}
This completes the proof.
\end{proof}

An easy consequence of Lemma~\ref{lem:slaF} is the following fact, that will play a key role in this work.

\begin{corollary}
\label{cor:integrability}
Let $\s_\la(p)$ be a spherical surface inside a $3$-dimensional space form. Then, for any Riemannian unit normal vector $N$ along $\s_\la(p)$, the function $\mnh^{-1}$ is integrable along $\s_\la(p)$ with respect to the Riemannian area element $d\sla(p)$.
\end{corollary}

\section{Characterization results and classification of stationary surfaces}
\label{sec:1stvar}

Let $M$ be a $3$-dimensional space form. Our main goal in this section is to establish characterization results for the spherical surfaces $\s_\la(p)$ introduced in \eqref{eq:sla} as critical points of the sub-Riemannian area under a volume constraint. More precisely, we will be able to prove sub-Riemannian counterparts to the classical theorems of Hopf and Alexandrov for constant mean curvature surfaces in  Riemannian $3$-space forms. We will also provide the description of complete volume-preserving area-stationary surfaces in $M$ with non-empty singular set, thus extending previous results in \cite{rr2} and \cite{hr1}. We first introduce basic definitions and review some properties of stationary surfaces that will be key ingredients throughout this section.

\subsection{Volume-preserving area-stationary surfaces}
\label{subsec:statsurf}
\noindent

Let $M$ be a Sasakian sub-Riema\-nnian $3$-manifold and $\varphi_0:\Sg\to M$ an oriented $C^2$ surface immersed in $M$. By a \emph{variation} of $\Sg$ we mean a map $\varphi:I\times\Sg\to M$ (which is assumed to be $C^2$ unless otherwise stated), where $I\subeq\rr$ is an open interval containing the origin, and $\varphi$ satisfies:
\begin{itemize}
\item[(i)] $\varphi(0,p)=\varphi_0(p)$ for any $p\in\Sg$,
\item[(ii)] the map $\varphi_s:\Sg\to M$ given by $\varphi_s(p):=\varphi(s,p)$ is an immersion for any $s\in I$, 
\item[(iii)] there is a compact set $C\subseteq\Sg$ such that $\varphi_{s}(p)=\varphi_0(p)$ for any $s\in I$ and any $p\in\Sg-C$.
\end{itemize}
The \emph{velocity vector field} associated to $\varphi$ is the $C^{1}$ vector field along $\Sg$ defined by $U_p:=(\ptl\varphi/\ptl s)(0,p)$. The \emph{area functional} is given by $A(s):=A(\Sg_{s})=A(\varphi_s(\Sg))$.  We define the \emph{volume functional} $V(s)$ as in \cite[Sect.~2]{bdce}, by taking the signed volume enclosed between $\Sg$ and $\Sg_s$. More precisely
\begin{equation}
\label{eq:volume}
V(s):=\int_{[0,s]\times C}\varphi^*(dM),
\end{equation}
where $dM$ is the Riemannian volume element in $(M,g)$. The variation is \emph{volume preserving} if $V(s)$ is constant for any $s$ small enough. We say that $\Sg$ is \emph{volume-preserving area-stationary} or \emph{area-stationary under a volume constraint} if $A'(0)=0$ for any volume-preserving variation of $\Sg$.  

Let us recall the computation of $A'(0)$ and $V'(0)$. From \eqref{eq:area} and the fact that $\Sg_0$ has vanishing Riemannian area, we have
\begin{equation}
\label{eq:duis}
A(s)=\int_{\Sg-\Sg_0}\mnh_p(s)\,
|\text{Jac}\,\varphi_s|_p\,d\Sg.
\end{equation}
Here $\mnh_p(s):=\mnh\big(\varphi_s(p)\big)$, where $N$ is a $C^1$ vector field along the variation whose restriction to any $\Sg_s$ provides the Riemannian unit normal $N_s$ compatible with the orientations of $\Sg$ and $M$. If $p\in\Sg$ and $\{e_1,e_2\}$ is any orthonormal basis in $T_p\Sg$ then $|\text{Jac}\,\varphi_s|_p$ denotes the squared root of the determinant of the positive definite matrix $G(s):=\big(\escpr{e_i(\varphi_s),e_j(\varphi_s)}\big)_{ij}$ with $i,j=1,2$. Note that $\mnh_p(s)$ is a positive Lipschitz function, for $s$ small enough. Moreover, we have
\[
\mnh_p'(s)=U(\mnh)=\escpr{D_UN_h,\nuh}, \quad |\text{Jac}\,\varphi_s|_p'(s)=\frac{1}{2}\,\sqrt{\text{det}\,G(s)}\,\,
\text{trace}(G'(s)\,G(s)^{-1}),
\]
where $\nuh$ is the horizontal Gauss map defined in \eqref{eq:nuh}. The previous expressions are uniformly bounded as functions of $s$ and $p$ near $s=0$. Hence we can use Lemma~\ref{lem:difint} to differentiate under the integral sign in \eqref{eq:duis}. By reproducing the arguments in \cite[Lem.~4.3]{rr2}, we get
\begin{equation}
\label{eq:aprima}
A'(0)=-2\int_{\Sg-\Sg_0} H\,\escpr{U,N}\,\,d\Sg
-\int_{\Sg-\Sg_0}\divv_\Sg\big(\escpr{U,N}\,(\nuh)^\top\big)\,d\Sg,
\end{equation}
provided the function
\begin{equation}
\label{eq:mc}
-2H:=\divv_\Sg\nuh
\end{equation}
is locally integrable with respect to $d\Sg$. Here $\divv_\Sg\nuh$ and $(\nuh)^\top$ are the Riemannian divergence relative to $\Sg$ and the tangent projection of $\nuh$, respectively. The function $H$ in \eqref{eq:mc} is the (\emph{sub-Riemannian}) \emph{mean curvature} of $\Sg$, as defined in \cite{rr1}, \cite{hrr} and \cite{rosales}. On the other hand, it is known \cite[Lem.~2.1]{bdce} that $V'(0)$ is given by
\begin{equation}
\label{eq:vprima}
V'(0)=\int_\Sg\escpr{U,N}\,d\Sg.
\end{equation} 

We say that $\Sg$ is a \emph{constant mean curvature surface} (CMC surface) if the function $H$ in \eqref{eq:mc} is constant on $\Sg-\Sg_{0}$. As an easy consequence of \eqref{eq:aprima} and \eqref{eq:vprima} it follows that $\Sg$ is volume-preserving area-stationary if and only if there is a constant $\lambda\in\rr$ such that $(A+2\la V)'(0)=0$ for any variation of $\Sg$. In particular $\Sg$ has constant mean curvature $\la$, see \cite[Prop.~4.3]{rosales} for details. 

The regular part $\Sg-\Sg_0$ of a CMC surface is ruled by CC-geodesics of $M$, see for example \cite{rr2}, \cite{chmy} and \cite{hp1}. In precise terms we get this result, which can be proved as in \cite[Thm.~4.8]{rr2}.

\begin{proposition}
\label{prop:ruling}
Let $\Sg$ be an oriented $C^2$ surface immersed inside a Sasakian sub-Riemannian $3$-manifold.  Then, in the regular part $\Sg-\Sg_0$, we have equality $D_ZZ=(2H)\,\nuh$. Therefore, if $\Sg$ has constant mean curvature $H$, then any characteristic curve of $\Sg$ is a CC-geodesic of curvature $H$.
\end{proposition}  

The singular set $\Sg_0$ of a CMC surface $\Sg$ is well understood by the results in \cite{chmy} for arbitrary pseudo-Hermitian $3$-manifolds, see also \cite[Sect.~5]{galli}. By using the ruling property in Proposition~\ref{prop:ruling}, the description of $\Sg_0$ given in \cite{chmy} can be stated as follows.

\begin{theorem}[{\cite[Thm.~B]{chmy}}]
\label{th:structure}
Let $\Sg$ be an oriented $C^2$ surface of constant mean curvature $H$ immersed in a Sasakian sub-Riemannian $3$-manifold $M$. Then, the singular set $\Sg_0$ consists of isolated points and $C^1$ curves with non-vanishing tangent vector (singular curves). Moreover, we have:
\begin{itemize}
\item[(i)] $($\cite[Thm.~3.10 and Lem.~3.8]{chmy}$)$ if $p\in\Sg_{0}$ is isolated, then there exist $r>0$ and $\la\in\rr$ with $|\la|=|H|$ such that the set $D_{r}(p):=\{\gamma_{v}(s)\,;\,v\in\h_p,\,|v|=1,\,s\in [0,r)\}$
is an open neighborhood of $p$ in $\Sg$. Here $\ga_v$ denotes the CC-geodesic of curvature $\la$ in $M$ with $\ga_v(0)=p$ and $\dot{\ga}_v(0)=v$. Moreover, the index of the vector field $\mnh\,Z$ at $p$ equals $+1$.
\vspace{0,1cm}
\item[(ii)]  $($\cite[Prop.~3.5 and Cor.~3.6]{chmy}$)$ if $p$ is
contained in a $C^1$ curve $\Ga\subset\Sg_{0}$ then there is a neighborhood $D$ of $p$ in $\Sg$ such that $D-\Gamma$ is the union of two disjoint connected open sets $D^+, D^-\subsetneq\Sg-\Sg_0$. Moreover, for any $q\in\Ga\cap D$ there are exactly two CC-geodesics $\ga_{1}\sub D^+$ and $\ga_{2}\sub D^-$ of curvature $\la$ leaving from $q$ and meeting transversally $\Ga$ at $q$ with opposite initial velocities. The curvature $\la$ does not depend on $q$ and satisfies $|\la|=|H|$.
\end{itemize} 
\end{theorem}

\begin{remarks}
\label{re:layH}
1. The relation between $\la$ and $H$ depends on $N_p$. If $N_p=T_p$ then $\la=H$, and the CC-geodesics in the statement are characteristic curves of $\Sg$. If $N_p=-T_p$ then $\la=-H$.

2. The regular and singular sets of a $C^1$ surface in the Heisenberg group $\mathbb{M}(0)$ satisfying the constant mean curvature equation in a weak sense have been studied in \cite{chy2} and \cite{chmy2}, see also \cite{galli-ritore2}.
\end{remarks}

The previous result together with the Hopf index theorem for line fields implies the following topological restriction proved in \cite[Thm.~E]{chmy}, see also \cite[Thm.~5.7]{galli}.

\begin{theorem}
\label{th:topology}
Let $\Sg$ be a compact, connected, oriented $C^2$ surface immersed in a Sasakian sub-Riemannian $3$-manifold. If $\Sg$ is a CMC surface, then the genus of $\Sg$ equals $0$ or $1$. Moreover, $\Sg$ is homeomorphic to a sphere if and only if $\Sg$ contains at least one isolated singular point.
\end{theorem} 

In sub-Riemannian geometry a CMC surface need not be volume-preserving area-stationary. Indeed, we have the following characterization result, see \cite[Thm.~4.17 and Prop.~4.20]{rr2} for a proof in $\mathbb{M}(0)$ and \cite{chy}, \cite{ch2}, \cite{hp2}, \cite{galli} for generalizations with low regularity assumptions and other settings.

\begin{theorem}
\label{th:vpstationary}
Let $\Sg$ be an oriented $C^2$ surface immersed in a Sasakian sub-Riemannian $3$-manifold. Then $\Sg$ is area-stationary under a volume constraint if and only if $\Sg$ has constant mean curvature and the characteristic curves meet orthogonally the singular curves when they exist. Moreover, in such a case, any singular curve in $\Sg$ is of class $C^2$.
\end{theorem}

\begin{example}
\label{ex:slacmc}
Let $M$ be a $3$-dimensional space form. Consider a spherical surface $\s_\la(p)$ as defined in \eqref{eq:sla}. By using Lemma~\ref{lem:slaF} (ii), equation~\eqref{eq:geoeq} and Proposition~\ref{prop:ruling} we deduce that $\s_\la(p)$ has constant mean curvature $H=\la$. Moreover, as $\sla(p)$ has two singular points we conclude from Theorem~\ref{th:vpstationary} that $\sla(p)$ is area-stationary under a volume constraint. The same argument shows that any plane $\mathcal{L}_\la(p)$ as in \eqref{eq:lla} is volume-preserving area-stationary with $H=0$. 
\end{example}

\subsection{Characterization results}
\label{subsec:classresults}
\noindent

Now we can prove a first characterization result for the spherical surfaces $\s_\la(p)$. This may be seen as a sub-Riemannian counterpart to Hopf uniqueness theorem, see \cite[Chap.~6]{hopf} and \cite{chern}, which states that any constant mean curvature sphere immersed inside a complete Riemannian $3$-manifold of constant sectional curvature must be round.

\begin{theorem}[A Hopf theorem for sub-Riemannian $3$-space forms]
\label{th:hopf}
If $\Sg$ is a topological sphere of class $C^2$ and with constant mean curvature immersed inside a $3$-dimensional space form $M$, then $\Sg$ coincides with one of the surfaces $\s_\la(p)$ defined in \eqref{eq:sla}.
\end{theorem}

\begin{proof}
First we apply Theorem~\ref{th:topology} to deduce that $\Sg$ contains an isolated singular point $p$. We choose the unit normal $N$ to $\Sg$ such that $N_p=T_p$. Let $\kappa$ be the Webster scalar curvature of $M$, and $\la$ the mean curvature of $\Sg$ with respect to $N$. By Theorem~\ref{th:structure} (i) and Remarks~\ref{re:layH} we can parameterize a small open neighborhood of $p$ in $\Sg$ by means of the flow of all the CC-geodesics $\ga_v$ of curvature $\la$ in $M$ leaving from $p$. By completeness any $\ga_v$ can be extended as a characteristic curve of $\Sg$ until it meets a singular point of $\Sg$. If $\la^2+\kappa\leq 0$ then equation \eqref{eq:vthetacases} would imply that the vertical component of the associated CC-Jacobi fields never vanishes along $\Sg-\{p\}$. Thus, the point $p$ would be the unique singular point of $\Sg$, a contradiction with the Hopf index theorem for line fields since $\Sg$ is topologically a sphere. Hence $\la^2+\kappa>0$. If $\la\geq 0$, then it is clear that $\Sg=\s_\la(p)$. In the case $\la<0$ we obtain $\Sg=\s_{-\la}(p_\la)$, where $p_\la$ is the first cut point of the CC-geodesics $\ga_v$.
\end{proof}

\begin{remark}
Recently Cheng, Chiu, Hwang and Yang have proved a Hopf theorem in the Heisenberg group $\mathbb{H}^n$ with $n\geq 2$ under the additional assumption that $\Sg$ is umbilic \cite[Cor.~$A'$]{cchy}.
\end{remark}

The ideas in the previous proof can be used to obtain the classification of CMC surfaces with isolated singular points in $3$-dimensional space forms. 

\begin{theorem}
\label{th:isolated}
Let $\Sg$ be a complete, connected and orientable $C^2$ surface immersed inside a $3$-dimensional space form $M$. If $\Sg$ has constant mean curvature and at least one isolated singular point, then $\Sg$ is either a spherical surface $\s_\la(p)$ as in \eqref{eq:sla}, or a plane $\mathcal{L}_\la(p)$ as in \eqref{eq:lla}.
\end{theorem}

In \cite[Thm.~6.8]{rr2} and \cite[Thm.~5.9]{hr1} it was proved, by means of an explicit computation of the CC-geodesics, that any singular curve of a $C^2$ volume-preserving area-stationary surface in the Heisenberg group $\mm(0)$ or in the unit sphere $\mm(1)$ is a piece of a CC-geodesic. Here we are able to generalize this fact to arbitrary $3$-dimensional space forms by using the explicit expression of the vertical component of a CC-Jacobi field given in Lemma~\ref{lem:jacobicon}.

\begin{theorem}
\label{th:singcurves}
Let $\Sg$ be a complete and orientable $C^2$ surface immersed inside a $3$-dimensional space form $M$. If $\Sg$ is volume-preserving area-stationary then any singular curve of $\Sg$ is a complete CC-geodesic in $M$. 
\end{theorem}

\begin{proof}
Let $\Ga:I\to \Sg$ be a connected and relatively compact open piece of a singular curve in $\Sg$. By Theorems~\ref{th:vpstationary} and~\ref{th:structure} we can suppose that $\Ga$ is a $C^2$ horizontal curve parameterized by arc-length. This implies that $\dot{\Ga}'=h\,J(\dot{\Ga})$, where $h:I\to\rr$ is the continuous function defined by $h:=\escpr{\dot{\Ga}',J(\dot{\Ga})}$. To prove the statement it suffices by \eqref{eq:geoeq} to show that $h$ is a constant function.

We take a unit normal $N$ to $\Sg$ for which $N=T$ along $\Ga$. We know that $\Sg$ has constant mean curvature $\la$ with respect to $N$. Let $F:I\times [0,+\infty)\to M$ be the $C^1$ map defined by $F(\eps,s):=\ga_\eps(s)$, where $\ga_\eps:[0,+\infty)\to M$ is the CC-geodesic of curvature $\la$ with $\ga_\eps(0)=\Ga(\eps)$ and $\dot{\ga}_\eps(0)=J(\dot{\Ga}(\eps))$. By using Theorem~\ref{th:structure} (ii), Remarks~\ref{re:layH}, Theorem~\ref{th:vpstationary} and the completeness of $\Sg$, we deduce that any $\ga_\eps(s)$ with $s>0$ is a characteristic curve of $\Sg$ until it meets a singular point. 
We denote $V_\eps(s):=(\ptl F/\ptl\eps)(\eps,s)$.  From Lemma~\ref{lem:ccjacobi} (iii) this is a CC-Jacobi field along $\ga_\eps$. Let $v_\eps:=\escpr{V_\eps,T}$.  By Lemma~\ref{lem:ccjacobi} (ii), equation \eqref{eq:jacobi1} and equality $V_\eps(0)=\dot{\Ga}(\eps)$, we get
\begin{equation}
\label{eq:vtheta2}
V_\eps=-(\la\,v_\eps)\,\dot{\ga}_\eps+(v'_\eps/2)\,J(\dot{\ga}_\eps)+v_\eps\, T.
\end{equation}

\vspace{0,1cm}
\emph{Claim $(*)$.} Suppose that there is an open interval $I'\subeq I$ and a positive function $s(\eps)$ defined on $I'$ such that $v_\eps(s(\eps))=0$ and $V_\eps(s(\eps))\neq 0$ for any $\eps\in I'$. Then $s(\eps)$ is constant on $I'$.

To prove the claim we proceed as follows. Consider the curve $\Ga_0(\eps):=F(\eps,s(\eps))$, $\eps\in I'$. Two tangent vectors to $\Sg$ along $\Ga_0$ are given by $(\ptl F/\ptl\eps)(\eps,s(\eps))=V_\eps(s(\eps))$ and $(\ptl F/\ptl s)(\eps,s(\eps))=\dot{\ga}_\eps(s(\eps))$. These vectors are horizontal and orthogonal by \eqref{eq:vtheta2}. As a consequence, $\Ga_0(\eps)$ is a singular curve of $\Sg$. Up to a reparameterization we can suppose by Theorem~\ref{th:vpstationary} that $\Ga_0$ is a $C^2$ curve. Since the differential of $F:I'\times [0,+\infty)\to\Sg$ has rank two for any pair $(\eps,s(\eps))$ we deduce from the inverse function theorem that $s(\eps)$ is a $C^1$ function. Moreover, by the definition of $\Ga_0(\eps)$ we have
\[
\dot{\Ga}_0(\eps)=V_\eps(s(\eps))+s'(\eps)\,\dot{\ga}_\eps(s(\eps)).
\]
As $\Sg$ is volume-preserving area-stationary, we conclude from the orthogonality condition in Theorem~\ref{th:vpstationary} that $\dot{\Ga}_0(\eps)\,\bot\,\dot{\ga}_\eps(s(\eps))$ and so, $s'(\eps)=0$ for any $\eps\in I'$. This proves the claim.

\vspace{0,1cm}
Now we compute the initial conditions of $v_\eps$. Clearly $v_\eps(0)=0$ since $V_\eps(0)=\dot{\Ga}(\eps)$. Evaluating \eqref{eq:vtheta2} at $s=0$ we see that $v_\eps'(0)=-2$. From \eqref{eq:vtheta2} we also obtain
\[
V_\eps'(0)=(2\la)\,\dot{\ga}_\eps(0)+(v_\eps''(0)/2)\,J(\dot{\ga}_\eps(0))
-J(\dot{\ga}_\eps)'(0)-2\,T_{\ga_\eps(0)}.
\]
Note that $J(\dot{\ga}_\eps)'=(2\la)\,\dot{\ga}_\eps-T$ by \eqref{eq:dujv} and \eqref{eq:geoeq}. Therefore
\[
V_\eps'(0)=-(v_\eps''(0)/2)\,\dot{\Ga}(\eps)-T_{\Ga(\eps)}.
\]
On the other hand, Lemma~\ref{lem:ccjacobi} (i) implies $[\dot{\ga}_\eps,V_\eps]=0$, so that $V_\eps'=D_{\dot{\ga}_\eps}V_\eps=D_{V_\eps}\dot{\ga}_\eps$ along $\ga_\eps$. Hence
\[
V_\eps'(0)=D_{V_\eps}J(\dot{\Ga})=J\big(\dot{\Ga}'(\eps)\big)-T_{\Ga(\eps)}=-\big(h(\eps)\,\dot{\Ga}(\eps)+T_{\Ga(\eps)}\big),
\]
where we have used \eqref{eq:dujv} and $\dot{\Ga}'=h\,J(\dot{\Ga})$. The two previous equalities for $V_\eps'(0)$ yield $v''_\eps(0)=2\,h(\eps)$. 

Let $\kappa$ be the Webster scalar curvature of $M$ and denote $\tau:=4\,(\la^2+\kappa)$. From Lemma~\ref{lem:jacobicon} we can obtain the explicit expression of $v_\eps$ depending on the sign of $\tau$. Suppose that $\tau<0$. Then, we have
\[
v_\eps(s)=\frac{2}{\sqrt{-\tau}}\left[\frac{h(\eps)}{\sqrt{-\tau}}\,
\big(\!\cosh(\sqrt{-\tau}\,s)-1\big)-\sinh(\sqrt{-\tau}\,s) \right].
\]
It follows that $v_\eps(s)=0$ for some $s>0$ if and only if $\phi(s)=h(\eps)/\sqrt{-\tau}$, where $\phi:(0,+\infty)\to\rr$ is defined by $\phi(s):=\sinh(\sqrt{-\tau}\,s)/(\cosh(\sqrt{-\tau}\,s)-1)$. Note that $\phi$ is decreasing on $(0,+\infty)$ and $\phi(0,+\infty)=(1,+\infty)$. Hence, there is a positive solution of equation $v_\eps(s)=0$ if and only if $h(\eps)>\sqrt{-\tau}$. If $h=\sqrt{-\tau}$ on $I$ then the statement is proved. Otherwise, we can find $\eps_0\in I$ such that $h(\eps_0)\neq\sqrt{-\tau}$.  After reversing the orientation of $\Ga$ if necessary we can suppose that $h(\eps_0)>\sqrt{-\tau}$. Consider the set $A:=\{\eps\in I\,;\,h(\eps)=h(\eps_0)\}$, which is clearly closed in $I$. Let us see that $A$ is also open in $I$, then proving that $h$ is constant. Take $\eps_1\in A$. By continuity there is an open interval $I'\subeq I$ with $\eps_1\in I'$ and $h>\sqrt{-\tau}$ on $I'$. Thus, for any $\eps\in I'$, there is a unique $s(\eps)>0$ for which $v_\eps(s(\eps))=0$. From the definition of $s(\eps)$ we get $v_\eps'(s(\eps))=2$ and so, $V_\eps(s(\eps))=J(\dot{\ga}_\eps(s(\eps)))\neq 0$ by \eqref{eq:vtheta2}. By applying the Claim $(*)$ we conclude that $s(\eps)$ is constant on $I'$. Hence $h(\eps)$ is constant on $I'$ and $I'\subeq A$, as desired. This completes the proof when $\tau<0$. If $\tau>0$ or $\tau=0$, then the function $v_\eps(s)$ is respectively given by $h(\eps)\,s^2-2s$ or $2\tau^{-1/2}\,\big[h(\eps)\,\tau^{-1/2}\,(1-\cos(\sqrt{\tau}\,s))-\sin(\sqrt{\tau}\,s)\big]$, and we can reason as in the case $\tau<0$ to finish the proof. 
\end{proof}

Now we can establish a sub-Riemannian counterpart to a classical theorem of Alexandrov \cite{alexandrov}, which states that any compact and embedded constant mean curvature surface inside the Euclidean space, the hyperbolic space or an open hemisphere of $\mathbb{S}^3$ must be a round sphere. The result was already proved for the Heisenberg group $\mm(0)$ in \cite[Thm.~6.10]{rr2}.  

\begin{theorem}[An Alexandrov theorem in $\e$]
\label{th:alexandrov}
Let $\Sg$ be a compact, connected, orientable $C^2$ surface immersed in some model space $\e$ with $\kappa\leq 0$ or inside an open hemisphere of $\e$ with $\kappa>0$. If $\Sg$ is volume-preserving area-stationary then $\Sg$ coincides with one of the spherical surfaces $\s_\la(p)$ defined in \eqref{eq:sla}.
\end{theorem}

\begin{proof}
Suppose that $\Sg$ has empty singular set or that it contains a singular curve. In both cases, Proposition~\ref{prop:ruling} or Theorem~\ref{th:singcurves} would imply that there is a complete CC-geodesic of $\e$ included in $\Sg$. By Proposition~\ref{prop:geomodel} we get a contradiction with the compactness of $\Sg$ or with the fact that $\Sg$ is contained inside an open hemisphere. Hence, we conclude that $\Sg$ must contain an isolated singular point. In particular, $\Sg$ is topologically a sphere by Theorem~\ref{th:topology}. Moreover, $\Sg$ has constant mean curvature since it is volume-preserving area-stationary. We finish by using Theorem~\ref{th:hopf}.
\end{proof}

\begin{remark}
\label{re:tekel}
In Theorem~\ref{th:alexandrov} the hypothesis that $\Sg$ is contained inside an open hemisphere of $\e$ when $\kappa>0$ cannot be removed. In fact, any Clifford torus given by the product of two circles with suitable radii is a volume-preserving area-stationary surface in $\e$, see \cite[Ex.~4.6]{hr1}. Other examples of embedded CMC tori with rotational symmetry in $\mathbb{M}(1)$ were found in \cite[Rem.~6.6]{hr1}.  Of course, none of these tori is inside an open hemisphere of $\e$. Note that Theorem~\ref{th:alexandrov} also fails in arbitrary $3$-dimensional space forms of Webster scalar curvature $\kappa\leq 0$. Let $M:=\e/G$ be the Sasakian cylinder defined in Example~\ref{ex:topology}. It is known, see \cite[Sect.~4.2]{rosales}, that any Euclidean vertical cylinder $\Sg$ about the vertical axis is a CMC surface in $\e$ with empty singular set. In $M$  the quotient surface $\Sg/G$ provides a volume-preserving area-stationary embedded torus.
\end{remark}

The results in Theorems~\ref{th:singcurves}, \ref{th:vpstationary} and \ref{th:structure} also allow to describe and classify the complete volume-preserving area-stationary surfaces with singular curves in any $3$-dimensional space form. By using the same ideas as in \cite[Thm.~6.11]{rr2} and \cite[Thm.~5.9]{hr1} we can prove the following result.

\begin{theorem}
\label{th:cmula}
Let $\Sg$ be a complete, connected and orientable $C^2$ surface immersed in a $3$-dimen\-sional space form. If $\Sg$ is volume-preserving area-stationary and contains at least one singular curve, then $\Sg$ coincides with one of the surfaces $\mathcal{C}_{\mu,\la}(\Ga)$ defined in Example~\ref{ex:cmula} below.
\end{theorem}

\begin{example}
\label{ex:cmula}
We follow the construction in \cite[Ex.~6.7]{rr2} and \cite[Ex.~5.8]{hr1}. Let $M$ be a $3$-dimensional space form and $\Ga$ a complete CC-geodesic of curvature $\mu$. At any point of $\Ga$ there are two unit horizontal vectors orthogonal to $\Ga$. Then, for any $\la\in\rr$, we can consider two families of CC-geodesic rays of curvature $\la\in\rr$ leaving orthogonally from $\Ga$. We extend these geodesics until they meet a new singular curve ``parallel" to $\Ga$. Now, we repeat the construction with the new singular curves and so on. The resulting set $\mathcal{C}_{\mu,\la}(\Ga)$ is a volume-preserving area-stationary surface of constant mean curvature $\la$ whose singular set consists of CC-geodesics of curvature $\mu$. 
\end{example}

Note that Theorems~\ref{th:isolated} and ~\ref{th:cmula} provide the classification of complete volume-preserving area-stationary $C^2$ surfaces with non-empty singular set in any $3$-dimensional space form. This generalizes previous results in \cite[Sect.~6]{rr2} for $\mm(0)$ and in \cite[Sect.~5]{hr1} for $\mm(1)$. Some characterizations of CMC surfaces with empty singular set in $\e$ can be found in \cite{ch}, \cite{bscv}, \cite{rr1}, \cite{rr2}, \cite{hr1} and \cite{rosales}.

\section{Stability properties of $\sla(p)$}
\label{sec:stability}

In this section we show that the spherical surfaces $\sla(p)$ introduced in \eqref{eq:sla} are second order minima of the sub-Riemannian area under certain deformations. We first state a second variation formula for $\sla(p)$ and discuss the technical problems arising in the computation.  

Let $\Sg$ be an oriented $C^2$ surface immersed inside a Sasakian sub-Riemannian $3$-manifold $M$. For any variation $\varphi:I\times\Sg\to M$ we know from \eqref{eq:duis} that the associated area functional is given by
\[
A(s)=\int_{\Sg-\Sg_0}f(s,p)\,d\Sg,
\]
where $f(s,p):=\mnh_p(s)\,|\text{Jac}\,\varphi_s|_p$. In Section~\ref{subsec:statsurf} we used differentiation under the integral sign, as stated in Lemma~\ref{lem:difint}, to get $A'(0)=\int_{\Sg-\Sg_0}(\ptl f/\ptl s)(0,p)\,d\Sg$. This was possible since, for fixed $p\in\Sg-\Sg_0$, the map $s\mapsto f(s,p)$ is absolutely continuous, the derivative $(\ptl f/\ptl s)(0,p)$ exists, and $(\ptl f/\ptl s)(s,p)$ is uniformly bounded as a function of $s$ and $p$ near $s=0$. If we want to compute $A''(0)$ from Lemma~\ref{lem:difint} we need some conditions that the derivatives $(\ptl f/\ptl s)(s,p)$ and $(\ptl^2 f/\ptl s^2)(s,p)$ must satisfy. This leads us to the notion of admissible variation that we now introduce.

\begin{definition}
\label{def:admissible}
Let $\varphi:I\times\Sg\to M$ be a $C^1$ variation of $\Sg$. Recall that there is a compact set $C\subeq\Sg$ such that $\varphi_s(p)=\varphi_0(p)$ for any $s\in I$ and any $p\in\Sg-C$. We say that $\varphi$ is \emph{admissible} if the function $f:I\times\Sg\to\rr$ defined by $f(s,p):=\mnh_p(s)\,|\text{Jac}\,\varphi_s|_p$ satisfies these properties:
\begin{itemize}
\item[(i)] $(\ptl f/\ptl s)(s,p)$ exists for any $s\in I$ and any $p\in\Sg-\Sg_0$,
\item[(ii)] the map $s\mapsto (\ptl f/\ptl s)(s,p)$ is absolutely continuous for almost every $p\in\Sg-\Sg_0$,
\item[(iii)] there is a function $h:\Sg\to\rr$, which is integrable on $C$ with respect to $d\Sg$, and such that $|(\ptl^2 f/\ptl s^2)(s,p)|\leq h(p)$ for almost every $p\in\Sg-\Sg_0$ and almost every $s\in I$,
\item[(iv)] the derivative $(\ptl^2 f/\ptl s^2)(0,p)$ exists for almost every $p\in\Sg-\Sg_0$.
\end{itemize} 
\end{definition}

Clearly we can replace $\Sg-\Sg_0$ with $C\cap(\Sg-\Sg_0)$ in the previous definition. Note that condition (i) holds provided the variation \emph{preserves the regular set} of $\Sg$, i.e., $\varphi_s(\Sg-\Sg_0)\subseteq\Sg_s-(\Sg_s)_0$ for any $s\in I$. In particular it is easy to check that, if $\Sg-\Sg_0$ is a $C^3$ surface and $\varphi$ is a $C^3$ variation compactly supported on $\Sg-\Sg_0$ (which implies that $\varphi$ leaves $\Sg_0$ invariant) then there is an interval $I'\sub\sub I$ such that the restriction of $\varphi$ to $I'\times\Sg$ is admissible. In Appendix~B we will show several examples of admissible variations for which $\Sg_0$ is not fixed.

For a spherical surface $\sla(p)$ the family of admissible variations moving the poles is very large. In order to analyze if $\sla(p)$ is a second order minimum of the area we must obtain an explicit expression of $A''(0)$ for all these variations. In \cite[Thm.~5.2]{rosales} this was done for variations compactly supported off of the poles and having the form $\varphi_s(p):=\exp_p(\omega(s,p)\,N_p)$. In the next result we extend the second variation formula to the case of arbitrary admissible variations. 

\begin{theorem}[Second variation formula for $\sla(p)$] 
\label{th:2ndsla} 
Let $M$ be a $3$-dimensional space form of Webster scalar curvature $\kappa$. Consider a spherical surface $\sla(p)$ for some $p\in M$ and $\la\geq 0$ with $\la^2+\kappa>0$. Let $\varphi:I\times\sla(p)\to M$ be an admissible variation of class $C^3$ off of the poles. Denote $u:=\escpr{U,N}$, where $U$ is the velocity vector field and $N$ is the unit normal in Lemma~\ref{lem:slaF} (ii). Then, the functional $A+2\la V$ is twice differentiable at $s=0$, and we have
\begin{equation}
\label{eq:2ndsla}
(A+2\la V)''(0)=\mathcal{I}(u,u),
\end{equation}
where $\mathcal{I}$ is the quadratic form on $C^1(\sla(p))$ defined by
\begin{equation}
\label{eq:indexform}
\mathcal{I}(v,w):=\int_{\sla(p)}|N_{h}|^{-1}\left\{Z(v)\,Z(w)-\big(1+(\tau^2-1)\,\mnh^2\big)^2\,v w\right\}d\sla(p).
\end{equation}
In the previous formula $Z$ is the characteristic field introduced in \eqref{eq:nuh} and $\tau:=\sqrt{\la^2+\kappa}$.
\end{theorem}

That the expression in \eqref{eq:2ndsla} coincides with the one obtained in \cite[Thm.~5.2]{rosales} is a consequence of Lemma~\ref{lem:slaF} (iii), from which it is easy to check that
\begin{equation}
\label{eq:potential}
|B(Z)+S|^2+4\,(\kappa-1)\,\mnh^2=\big(1+(\tau^2-1)\,\mnh^2\big)^2
\end{equation}
along $\sla(p)$. Indeed, formula \eqref{eq:2ndsla} also holds for some particular admissible variations with regularity less than $C^3$ off of the poles, see Remarks~\ref{re:seno} for precise references. Note that the integrability of $\mnh^{-1}$ stated in Corollary~\ref{cor:integrability} implies that $\mathcal{I}(v,w)$ is finite for any $v,w\in C^1(\sla(p))$. We refer to $\mathcal{I}$ as the \emph{index form} associated to $\sla(p)$ by analogy with the Riemannian situation studied in \cite{bdce}. 

The proof of Theorem~\ref{th:2ndsla} is deduced from the more general second variation formula in Theorem~\ref{2ndgeneral} for admissible deformations possibly moving the singular set. After differentiation under the integral sign and a long computation, we get that $(A+2\la V)''(0)$ equals $\mathcal{I}(u,u)$ adding the divergence of some tangent vector fields defined off of the poles. Then we use a generalized divergence theorem together with the integrability of $\mnh^{-1}$ to show that these extra terms vanish, thus proving \eqref{eq:2ndsla}. All the details are found in Appendix~A. 

The next ingredient that we need is the expression of the index form with respect to the polar coordinates introduced in Lemma~\ref{lem:slaF}. 

\begin{lemma}
\label{lem:indexpolar}
Let $M$ be a $3$-dimensional space form of Webster scalar curvature $\kappa$. Consider a spherical surface $\sla(p)$ for some $p\in M$ and $\la\geq 0$ with $\la^2+\kappa>0$. For any function $u\in C^1(\sla(p))$ we denote $\overline{u}:=u\circ F$, where $F(\theta,s):=\ga_\theta(s)$ is the flow of CC-geodesics of curvature $\la$ leaving from $p$. Then, the index form of $\sla(p)$ defined in \eqref{eq:indexform} satisfies
\[
\mathcal{I}(u,u)=\frac{1}{\tau}\int_{[0,2\pi]\times [0,\pi]}\left\{\Big(\frac{\ptl\xi}{\ptl x}\Big)^2-\xi^2\right\}d\theta\,dx,
\]
where $\tau:=\sqrt{\la^2+\kappa}$ and $\xi(\theta,x):=\sqrt{1+(\tau^2-1)\cos^2(x)}\,\,\overline{u}(\theta,x/\tau)$.
\end{lemma}

\begin{proof}
Let $N$ be the unit normal along $\sla(p)$ for which any CC-geodesic $\ga_\theta(s)$ with $s\in (0,\pi/\tau)$ is a characteristic curve of $\sla(p)$. From the definition of $\mathcal{I}(u,u)$ and the expressions for $\mnh$ and $d\sla(p)$ in Lemma~\ref{lem:slaF} (iii), we obtain
\begin{equation}
\label{2ndslacoord} 
\mathcal{I}(u,u)=\int_{[0,2\pi]\times [0,\pi/\tau]}\Big\{\frac{1+(\tau^2-1)\cos^2(\tau s)}{\tau^2}\,\Big(\frac{\ptl\overline{u}}{\ptl s}\Big)^2-\frac{\tau^2}{1+(\tau^2-1)\cos^2(\tau s)}\,\overline{u}^2\Big\}\,d\theta\,ds.
\end{equation}
We define the $C^1$ function $\omega(\theta,s):=\sqrt{1+(\tau^2-1)\cos^2(\tau s)}\,\,\overline{u}(\theta,s)$. An easy computation gives us
\begin{align}
\label{deru2}
\frac{1+(\tau^2-1)\,\cos^2(\tau s)}{\tau^2}\,\Big(\frac{\ptl\overline{u}}{\ptl s}\Big)^2&=\frac{1}{\tau^2}\,\Big(\frac{\ptl\omega}{\ptl s}\Big)^2-\frac{(\tau^2-1)^2\,\cos^2(\tau s)\,\sin^2(\tau s)}{1+(\tau^2-1)\cos^2(\tau s)}\,\overline{u}^2
\\
\nonumber 
&+\frac{2\,(\tau^2-1)\,\cos(\tau s)\,\sin(\tau s)}{\tau}\,\frac{\ptl\overline{u}}{\ptl s}\,\overline{u}.
\end{align}
On the other hand, for fixed $\theta\in[0,2\pi]$, we can apply integration by parts to deduce
\begin{equation}
\label{intparts}
\int_0^{\pi/\tau}2\,\cos(\tau s)\,\sin(\tau s)\,\,\frac{\ptl\overline{u}}{\ptl s}\,\,\overline{u}\,ds=-\tau\int_0^{\pi/\tau}\big\{\!\cos^2(\tau s)-\sin^2(\tau s)\big\}\,\overline{u}^2\,ds.
\end{equation}
Now, we use Fubini's theorem and substitute equations \eqref{deru2} and \eqref{intparts} into \eqref{2ndslacoord}. After simplifying, we conclude that
\begin{equation*}
\mathcal{I}(u,u)=\int_{[0,2\pi]\times [0,\pi/\tau]}
\Big\{\frac{1}{\tau^2}\,\Big(\frac{\ptl\omega}{\ptl
s}\Big)^2-\omega^2\Big\}\,d\theta\,ds.
\end{equation*}
The claim follows from the change of variables $x:=\tau s$ and the fact that $\xi(\theta,x)=\omega(\theta,x/\tau)$.
\end{proof}

\begin{remark}
The previous lemma has an interesting consequence: the expression in polar coordinates of $\tau\,\mathcal{I}$ is a quadratic form not depending on the sphere $\sla(p)$ nor the ambient manifold $M$. This fact might suggest that all the spheres $\sla(p)$ will share the same behaviour with respect to stability conditions. We will illustrate this property rigorously in the main results of this section.  
\end{remark}

Recall that any sphere $\sla(p)$ is a volume-preserving area-stationary surface, see Example~\ref{ex:slacmc}. As we saw in Section~\ref{subsec:statsurf} this implies that $(A+2\la V)'(0)=0$ for any variation. Indeed, in the next statement we prove that $\sla(p)$ is a second order minimum of $A+2\la V$ for certain variations which may preserve volume or not.

\begin{proposition}
\label{prop:strictstability}
Let $M$ be a $3$-dimensional space form of Webster scalar curvature $\kappa$. Consider a spherical surface $\sla(p)$ for some $p\in M$ and $\la\geq 0$ with $\la^2+\kappa>0$. Then, the index form of $\sla(p)$ satisfies $\mathcal{I}(u,u)\geq 0$ for any $u\in C^1(\sla(p))$ vanishing at the poles or along the equator. As a consequence, $\sla(p)$ is strongly stable, in the sense that $(A+2\la V)''(0)\geq 0$ for any admissible variation which is $C^3$ off of the poles and fixes either the poles or the equator.
\end{proposition}

\begin{proof}
Let $u\in C^1(\sla(p))$. From Lemma~\ref{lem:indexpolar} and Fubini's theorem, we have
\[
\mathcal{I}(u,u)=\frac{1}{\tau}\int_0^{2\pi}\left[\int_0^\pi\left\{\xi_\theta'(x)^2-\xi_\theta(x)^2\right\}dx\right]d\theta,
\]
where $\tau:=\sqrt{\la^2+\kappa}$ and $\xi_\theta(x):=\sqrt{1+(\tau^2-1)\cos^2(x)}\,\,\overline{u}(\theta,x/\tau)$. Here the primes $'$ denote the derivatives with respect to $x$. If $u$ vanishes at the poles, then $\xi_\theta(0)=\xi_\theta(\pi)=0$, and we can apply Wirtinger's inequality for functions vanishing at the boundary (see for instance \cite[p.~143]{mitrinovic}) to get $\int_0^\pi\xi_\theta'(x)^2\,dx\geq\int_0^\pi\xi_\theta(x)^2\,dx$ for any $\theta\in [0,2\pi]$. Suppose that $u$ vanishes along the equator. Then $\overline{u}(\theta,\pi/(2\tau))=0$ for any $\theta\in[0,2\pi]$, and we can write $\xi_\theta(x)=\cos(x)\,\overline{v}_\theta(x)$ for some continuous function $\overline{v}_\theta:[0,\pi]\to\rr$. A straightforward computation gives us
\[
\xi_\theta'(x)^2-\xi_\theta(x)^2=-\cos(2x)\,\overline{v}_\theta(x)^2+\overline{v}_\theta'(x)^2\cos^2(x)-(1/2)\,\sin(2x)\,(\overline{v}_\theta(x)^2)',
\]
for any $x\neq \pi/2$. Then, integration by parts shows that the integrals on $(0,\pi/2)$ and $(\pi/2,\pi)$ of the functions $-(1/2)\sin(2x)\,(\overline{v}_\theta(x)^2)'$ and $\cos(2x)\,\overline{v}_\theta(x)^2$ coincide. Therefore
\[
\mathcal{I}(u,u)=\frac{1}{\tau}\int_0^{2\pi}\left[\int_0^\pi \overline{v}_\theta'(x)^2\cos^2(x)\,dx\right],
\]
which is clearly nonnegative. The last statement follows from Theorem~\ref{th:2ndsla}.
\end{proof}

The next result is an immediate consequence of the previous proposition.

\begin{corollary}
\label{cor:yeso}
A sphere $\sla(p)$ minus the poles is strongly stable, in the sense that $(A+2\la V)''(0)$ is nonnegative for any $C^3$ variation. A sphere $\sla(p)$ removing the equator is strongly stable, in the sense that $(A+2\la V)''(0)\geq 0$ for any admissible variation which is $C^3$ off of the poles. In particular, the open hemispheres $\sla(p)^-$ and $\sla(p)^+$ are strongly stable.
\end{corollary}

\begin{remark}
Previous results related to Proposition~\ref{prop:strictstability} were obtained for the Heisenberg group $\mm(0)$. In this particular case, it is well-known that a spherical surface $\sla(o)$ coincides with the union of two radial graphs defined over a horizontal disk $D_\la$ and meeting along the equator. On the one hand, Montefalcone~\cite[Cor.~4.11 and App.~A]{montefalcone-spheres} showed strong stability of $\sla(o)$ for certain variations with velocity vector field $U=(\mnh\,u)\,N$ or $U=u\,T$, where $u$ is a $C^1$ function with compact support inside an open hemisphere. On the other hand, the strong stability of $\sla(o)$ for admissible variations fixing the equator can be deduced from \cite[Eq.~(3.7)]{ritore-calibrations}, where Ritor\'e proved that $\sla(o)$ is a global minimizer of the functional $A+2\la V$ among finite perimeter sets $\Om$ of $\mm(0)$ such that $\Om\sub D_\la\times\rr$ and $D_\la\times\{0\}\sub\Om$.
\end{remark}

\begin{remark}
\label{re:chispa}
In spite of the previous results, the whole sphere $\sla(p)$ is not strongly stable for arbitrary admissible variations. To see this, it is enough to consider the deformation by Riemannian parallel surfaces $\varphi_s(p):=\exp_p(sN_p)$. This is an admissible variation by Lemma~\ref{lem:parallels} with velocity vector $U=N$. So, by taking into account Theorem~\ref{th:2ndsla}, we get
\[
(A+2\la V)''(0)=\mathcal{I}(1,1)=-\int_{\sla(p)}\mnh^{-1}\big(1+(\tau^2-1)\,\mnh^2\big)^2\,d\sla(p)<0.
\] 
\end{remark}

Clearly the variation considered in the previous example does not preserve volume. Indeed, our main result in this section establishes the following.

\begin{theorem}
\label{th:stability} 
Let $M$ be a $3$-dimensional space form of Webster scalar curvature $\kappa$. Consider a spherical surface $\sla(p)$ for some $p\in M$ and $\la\geq 0$ with $\la^2+\kappa>0$. Then, the index form of $\sla(p)$ satisfies $\mathcal{I}(u,u)\geq 0$, for any $u\in C^1(\sla(p))$ with $\int_{\sla(p)}u\,d\sla(p)=0$. As a consequence, $\sla(p)$ is stable under a volume constraint, in the sense that $A''(0)\geq 0$ for any volume-preserving admissible variation which is  $C^3$ off of the poles.
\end{theorem}

\begin{remark}
For a volume preserving variation of $\sla(p)$ with velocity vector $U$ such that $\escpr{U,N}=u$, we have $\int_{\sla(p)}u\,d\sla(p)=0$ by \eqref{eq:vprima} and $A''(0)=(A+2\la V)''(0)=\mathcal{I}(u,u)$ by \eqref{eq:2ndsla}.
\end{remark}

\begin{proof}[Proof of Theorem~\ref{th:stability}]
By standard approximation arguments, it suffices to prove the statement for test functions of class $C^2$. To this aim we will use double Fourier series on the rectangle $\rrr:=[-\pi,\pi]\times [-\pi/2,\pi/2]$. 

Take $u\in C^2(\sla(p))$ with $\int_{\sla(p)}u\,d\sla(p)=0$. If we define $\vartheta:=\theta-\pi$ and $t:=x-\pi/2$, then the equality in Lemma~\ref{lem:indexpolar} reads
\begin{equation*}
\tau\,\mathcal{I}(u,u)=\mathcal{Q}(\psi,\psi),
\end{equation*}
where $\mathcal{Q}$ is the symmetric bilinear form on $C^1(\rrr)$ given by
\begin{equation*}
\mathcal{Q}(\psi_1,\psi_2):=\int_\rrr\left\{\frac{\ptl\psi_1}{\ptl t}\,\frac{\ptl\psi_2}{\ptl t}-\psi_1\,\psi_2\right\}d\vartheta\,dt,
\end{equation*}
and $\psi:\rrr\to\rr$ is the function $\psi(\vartheta,t):=\xi(\vartheta+\pi,t+\pi/2)$. Note that $\psi(\vartheta,-\pi/2)=u(p)$ and $\psi(\vartheta,\pi/2)=u(p_\la)$ for any $\vartheta\in [-\pi,\pi]$, whereas $\psi(-\pi,t)=\psi(\pi,t)$ for any $t\in[-\pi/2,\pi/2]$. On the other hand, the mean zero condition for $u$ together with the expression for $d\sla(p)$ in polar coordinates (see Lemma~\ref{lem:slaF} (iii)) and the change of variables theorem, gives us 
\begin{equation*}
0=\frac{1}{\tau^3}\int_{[0,2\pi]\times [0,\pi]}\xi(\theta,x)\,\sin(x)\,d\theta\,dx=\frac{1}{\tau^3}\int_\rrr\psi(\vartheta,t)\,\cos(t)\,d\vartheta\,dt.
\end{equation*}
So, in order to to prove the claim, we must show that $\mathcal{Q}(\psi,\psi)\geq 0$ for any $\psi\in C^2(\rrr)$ such that: 
\begin{enumerate}
\item[(i)] $\psi(\vartheta,-\pi/2)$ and $\psi(\vartheta,\pi/2)$ are constant as functions of $\vartheta\in[-\pi,\pi]$,
\item[(ii)] $\psi(-\pi,t)=\psi(\pi,t)$, for any $t\in [-\pi/2,\pi/2]$,
\item[(iii)] $\int_R\psi(\vartheta,t)\cos(t)\,d\vartheta\,dt=0$.
\end{enumerate}

Fix $\psi\in C^2(\rrr)$ as above and define $\psi_s(\vartheta,t):=(\psi(\vartheta,t)+\psi(\vartheta,-t))/2$ and $\psi_a(\vartheta,t):=(\psi(\vartheta,t)-\psi(\vartheta,-t))/2$. Clearly $\psi=\psi_s+\psi_a$. Moreover, we have $\psi_s(\vartheta,-t)=\psi_s(\vartheta,t)$ and $\psi_a(\vartheta,-t)=-\psi_a(\vartheta,t)$, for any $t\in[-\pi/2,\pi/2]$. It follows from Fubini's theorem that $\mathcal{Q}(\psi_s,\psi_a)=0$ and that $\psi_s$ satisfies equality (iii) above. In particular, $\mathcal{Q}(\psi,\psi)=\mathcal{Q}(\psi_s,\psi_s)+\mathcal{Q}(\psi_a,\psi_a)$. Note that
\[
\mathcal{Q}(\psi_a,\psi_a)=\int_{-\pi}^\pi\left[\int_{-\pi/2}^{\pi/2}\left\{(\psi_a)_\vartheta'(t)^2-(\psi_a)_\vartheta(t)^2\right\}dt\right]d\vartheta,
\]
where $(\psi_a)_\vartheta(t):=\psi_a(\vartheta,t)$ and the derivatives are taken with respect to $t$. Let $(\overline{\psi}_a)_\vartheta$ be the continuous extension of $(\psi_a)_\vartheta$ to $[-\pi/2,3\pi/2]$ which is symmetric with respect to $t=\pi/2$. This function is piecewise $C^1$ with mean zero. So, we can apply Wirtinger's inequality for mean zero functions (see for instance \cite[p.~141]{mitrinovic}) to obtain
\[
2\int_{-\pi/2}^{\pi/2}(\psi_a)_\vartheta'(t)^2\,dt=
\int_{-\pi/2}^{3\pi/2}(\overline{\psi}_a)_\vartheta'(t)^2\,dt\geq\int_{-\pi/2}^{3\pi/2}(\overline{\psi}_a)_\vartheta(t)^2\,dt=2\int_{-\pi/2}^{\pi/2}(\psi_a)_\vartheta(t)^2\,dt,
\]   
and so $\mathcal{Q}(\psi_a,\psi_a)\geq 0$. Hence, to prove the theorem it suffices to see that $\mathcal{Q}(\psi_s,\psi_s)\geq 0$. Indeed, we will prove that $\mathcal{Q}(\psi,\psi)\geq 0$, for any $\psi\in C^2(\rrr)$ satisfying properties (i), (ii), (iii) and:
\begin{itemize}
\item[(iv)] $\psi(\vartheta,-t)=\psi(\vartheta,t)$, for any $(\vartheta,t)\in\rrr$.
\end{itemize}

Consider a function $\psi$ as above. The Fourier series of $\psi$ with respect to the basic orthogonal system in $L^2(\rrr)$ is the following sequence of functions (see \cite[Sect.~7.4]{tolstov}):
\begin{align}
\label{eq:fourier} 
S_{mn}(\vartheta,t):=\sum_{i=0}^{m}\sum_{j=0}^n\lambda_{ij}&\big\{a_{ij}\,\cos(i\vartheta)\,\cos(2jt)+b_{ij}\,\sin(i\vartheta)\,\cos(2jt)
\\
\nonumber
&+c_{ij}\,\cos(i\vartheta)\,\sin(2jt)+d_{ij}\,\sin(i\vartheta)\,\sin(2jt)\big\},
\end{align}
where $\la_{ij}$ is the real constant given by
\begin{equation}
\label{eq:laij}
\la_{ij}:=
\begin{cases}
1/4,\quad\text{for }i=j=0, 
\\
1/2,\quad\text{for }i=0, j\geq 1 \text{ or } i\geq 1, j=0,
\\
1,\quad \quad \text{for }i,j\geq 1,
\end{cases}
\end{equation}
and the coefficients $a_{ij}$, $b_{ij}$, $c_{ij}$, $d_{ij}$ are defined by
\begin{align*}
a_{ij}:=&\frac{2}{\pi^2}\int_\rrr\psi(\vartheta,t)\,\cos(i\vartheta)\,\cos(2jt)\,d\vartheta\,dt, \quad b_{ij}:=\frac{2}{\pi^2}\int_\rrr\psi(\vartheta,t)\,\sin(i\vartheta)\,\cos(2jt)\,d\vartheta\,dt,
\\
c_{ij}:=&\frac{2}{\pi^2}\int_\rrr\psi(\vartheta,t)\,\cos(i\vartheta)\,\sin(2jt)\,d\vartheta\,dt, \quad d_{ij}:=\frac{2}{\pi^2}\int_\rrr\psi(\vartheta,t)\,\sin(i\vartheta)\,\sin(2jt)\,d\vartheta\,dt.
\end{align*}
Note that $c_{ij}=d_{ij}=0$ as an elementary consequence of Fubini's theorem and (iv). On the other hand, the fact that $\psi(\vartheta,-\pi/2)=\psi(\vartheta,\pi/2)$ allows us to reproduce the computations in \cite[Sect.~3.10]{tolstov} to deduce that the Fourier series of $\ptl\psi/\ptl t$ coincides with
\[
\frac{\ptl S_{mn}}{\ptl t}(\vartheta,t)=\sum_{i=0}^{m}\sum_{j=0}^n\,(-2j)\,\lambda_{ij}\,\big\{a_{ij}\,\cos(i\vartheta)\,\sin(2jt)+b_{ij}\,\sin(i\vartheta)\,\sin(2jt)\big\}.
\]
Recall that for any $\phi\in L^2(\rrr)$ the associated Fourier series converges to $\phi$ in $L^2(\rrr)$. Thus, we have
\begin{align}
\label{eq:q1}
\int_\rrr\psi(\vartheta,t)\,h(\vartheta,t)\,d\vartheta\,dt=\sum_{m=0}^{\infty}\sum_{n=0}^\infty\,\lambda_{mn}\,&\bigg\{a_{mn}\int_\rrr\cos(i\vartheta)\,\cos(2jt)\,h(\vartheta,t)\,d\vartheta\,dt
\\
\nonumber
&+b_{mn}\int_\rrr\sin(i\vartheta)\,\cos(2jt)\,h(\vartheta,t)\,d\vartheta\,dt\bigg\},
\end{align}
for any $h\in L^2(\rrr)$. This implies the following Parseval identities for $\psi$ and $\ptl\psi/\ptl t$ (see \cite[Sect.~7.2]{tolstov}):
\begin{align*}
\frac{2}{\pi^2}\int_\rrr \psi^2\,d\vartheta\,dt&=\sum_{m,n=0}^{\infty}
\lambda_{mn}\,(a^2_{mn}+b^2_{mn}),
\\
\frac{2}{\pi^2}\int_\rrr\Big(\frac{\partial\psi}{\partial t}\Big)^2\,d\vartheta\,dt
&=\sum_{m,n=0}^{\infty}4n^2\,\lambda_{mn}\,(a^2_{mn}+b^2_{mn}).
\end{align*}
This ensures the absolute convergence of the numerical double series at the right hand sides. Now, the definition of the quadratic form $\mathcal{Q}(\psi,\psi)$ yields
\[
\frac{2}{\pi^2}\,\mathcal{Q}(\psi,\psi)=\sum_{m,n=0}^{\infty}
(4n^2-1)\,\lambda_{mn}\,(a^2_{mn}+b^2_{mn})=\sum_{m=0}^{\infty}
\sigma_m.
\]
Here $\{\sigma_m\}_{m\geq 0}$ is the sequence given by $\sigma_m:=\alpha_m+\beta_m$, where
\[
\alpha_m:=-\lambda_{m0}\,a^2_{m0}+ \sum_{n=1}^\infty
\,(4n^2-1)\,\lambda_{mn}\,a^2_{mn}, \quad \beta_m:=-\lambda_{m0}\,b^2_{m0}+ \sum_{n=1}^\infty
\,(4n^2-1)\,\lambda_{mn}\,b^2_{mn}.
\]
So, to finish the proof it suffices to check that $\alpha_m\geq 0$ and $\beta_m\geq 0$ for any $m\geq 0$.

It is clear that $\beta_0\geq 0$ since $b_{00}=0$. Let us see that $\alpha_0\geq 0$. By taking into account \eqref{eq:q1} and equality $\int_\rrr\psi(\vartheta,t)\cos(t)\,d\vartheta\,dt=0$ we get, after some easy computations
\[
\sum_{n=0}^\infty\la_{0n}\,a_{0n}\,\frac{(-1)^{n+1}}{4n^2-1}=0,
\]
which is equivalent by \eqref{eq:laij} to
\[
a_{00}=2\,\sum_{n=1}^{\infty}
\frac{(-1)^n}{4n^2-1}\,a_{0n}.
\]
On the other hand note that, for any $n \in \mathbb{N}$,
\begin{align*}
\Big(\sum_{i=1}^{n} \frac{2\,(-1)^i}{4i^2-1}\,a_{0i}\Big)^2&=\,4 \sum_{i,j=1}^{n}\frac{(-1)^{i+j}\,a_{0i}\,a_{0j}}{(4i^2-1)\,(4j^2-1)}\leq
2\,\sum_{i,j=1}^{n}\bigg(\frac{a^2_{0i}}{(4i^2-1)\,(4j^2-1)}+\frac{a^2_{0j}}{(4i^2-1)\,(4j^2-1)}\bigg)
\\
&=4\,\sum_{i,j=1}^n\frac{a_{0i}^2}{(4i^2-1)\,(4j^2-1)}=\,4\,\bigg(\sum_{i=1}^{n}\frac{a^2_{0i}}{4i^2-1}\bigg)\,\bigg(\sum_{j=1}^n\frac{1}{4j^2-1}\bigg)
\\
&\leq\frac{2\pi^2}{3}\,\sum_{i=1}^{n}\frac{a^2_{0i}}{4i^2-1}\leq 2\,\sum_{i=1}^n\,(4i^2-1)\,a^2_{0i},
\end{align*}
where we have used that $\sum_{j=1}^\infty j^{-2}=\pi^2/6$. Taking limits when $n\to\infty$ and using the expression for $a_{00}$ above we conclude that $\alpha_0\geq 0$.

Finally, let us see that $\alpha_m\geq 0$ and $\beta_m\geq 0$ for any $m\in\mathbb{N}$. As $\psi\in C^2(\rrr)$ we can follow the arguments in \cite[Sect.~3.10]{tolstov} to deduce that the Fourier series $S_{mn}$ of $\psi$ in \eqref{eq:fourier} converges uniformly to $\psi$ in $\rrr$, see also \cite{hardy}. Recall that $\psi(\vartheta,\pi/2)$ is a constant $c\in\rr$ which does not depend on $\vartheta\in [-\pi,\pi]$. So, we have
\[
\sum_{m,n=0}^\infty\lambda_{mn}\left\{a_{mn}\,(-1)^n\,\cos(m\vartheta)+b_{mn}\,(-1)^n\,\sin(m\vartheta)\right\}=c
\]
uniformly on $[-\pi,\pi]$. In particular, we get also convergence in $L^2(-\pi,\pi)$. So, for any $k\geq 1$, we can integrate on $[-\pi,\pi]$ the previous equation multiplied by $h(\vartheta):=\cos(k\vartheta)$ to obtain
\[
a_{m0}=2\,\sum_{n=1}^\infty (-1)^{n+1}\,a_{mn}.
\]
The same computation with $h(\vartheta):=\sin(k\vartheta)$ gives us
\[
b_{m0}=2\,\sum_{n=1}^\infty (-1)^{n+1}\,b_{mn}.
\]
The proof concludes by applying Lemma~\ref{lem:fouriersuma} below. 
\end{proof}

\begin{lemma} 
\label{lem:fouriersuma}
Let $\{x_n\}_{n\geq 0}$ be a sequence of real numbers such that  $x_0= 2\,\sum_{n=1}^\infty (-1)^{n+1}\,x_n$ and $\sum_{n=1}^\infty(4n^2-1)\,x_n^2<\infty$. Then, we have
\[
\frac{-1}{2}\,x_0^2 + \sum_{n=1}^\infty\,(4n^2-1)\,x_n^2 \geq 0.
\]
\end{lemma}

\begin{proof}
For any $n\in\mathbb{N}$, we define
\begin{align*}
s_n:=\frac{-1}{2}\,\Big(\sum_{i=1}^n 2\,(-1)^{i+1}\,x_i\Big)^2+\sum_{i=1}^n \,(4i^2-1)\,x_i^2=\sum_{i=1}^n\,(4i^2-3)\,x_i^2-4\sum_{\substack{i,j=1\\i<j}}^n(-1)^{i+j}\,x_i\,x_j.
\end{align*}
If we show that $s_n\geq 0$ for any $n\in\mathbb{N}$, then the lemma follows by taking limits in the first equality above. Indeed, we will use induction to prove that
\begin{equation}
\label{eq:induction}
s_n=\sum_{i=1}^{n-1}\Big[(2i-1)\,x_i+\sum_{k=1}^{n-i}2\,(-1)^{k+1}\,x_{i+k}\Big]^2+(2n-1)^2\,x_n^2, \quad n\in\mathbb{N}.
\end{equation}

It is easy to see that \eqref{eq:induction} is valid for $n=1,2$. Suppose that \eqref{eq:induction} holds for some $n\geq 2$. Then
\begin{align*}
s_{n+1}&=s_n+\big(4\,(n+1)^2-3\big)\,x_{n+1}^2+4\,(-1)^{n}\,x_{n+1}\,\sum_{i=1}^n\,(-1)^i\,x_i
\\
&=\sum_{i=1}^{n-1}\Big[(2i-1)\,x_i+\sum_{k=1}^{n-i}2\,(-1)^{k+1}\,x_{i+k}\Big]^2+(2n-1)^2\,x_n^2
\\
&+\big(4\,(n+1)^2-3\big)\,x_{n+1}^2+4\,(-1)^n\,x_{n+1}\,\sum_{i=1}^n\,(-1)^i\,x_i
\\
&=\sum_{i=1}^{n}\Big[(2i-1)\,x_i+\sum_{k=1}^{n+1-i}\!2\,(-1)^{k+1}\,x_{i+k}\Big]^2+(2n+1)^2\,x_{n+1}^2
\\
&+4\,(-1)^n\,x_{n+1}\,\sum_{i=1}^n\,(-1)^i\,x_i-4\,(2n-1)\,x_n\,x_{n+1}
\\
&-4\,x_{n+1}\sum_{i=1}^{n-1}\Big[(-1)^{n+i}\,(2i-1)\,x_i+\sum_{k=1}^{n-i}2\,(-1)^{n-i+k+1}\,x_{i+k}\Big].
\end{align*}
This gives us the desired expression for $s_{n+1}$ in \eqref{eq:induction} provided
\[
\sum_{i=1}^n\,(-1)^i\,x_i=(-1)^n\,(2n-1)\,x_n+\sum_{i=1}^{n-1}\Big[(-1)^{i}\,(2i-1)\,x_i+\sum_{k=1}^{n-i} 2\,(-1)^{i-k-1}\,x_{i+k}\Big].
\]
The previous equality can be checked by induction. This completes the proof.
\end{proof}

\begin{remark}
Montefalcone studied a $1$-dimensional eigenvalues problem to prove in \cite[Prop.~4.9]{montefalcone-spheres} that the Pansu spheres in the Heisenberg groups $\mathbb{H}^n$ are radially stable under a volume constraint. More precisely, it is shown that $A''(0)\geq 0$ for volume-preserving admissible variations with velocity vector $U=(\mnh\,u)\,N$ and such that $u$ is radially symmetric. 
\end{remark}

\begin{remark}
It is important to emphasize that our stability results in this section are valid for any variation of $\sla(p)$ for which formula \eqref{eq:2ndsla} holds. From Theorem~\ref{th:2ndsla} this happens for admissible variations $C^3$ off of the poles.  In Appendix~B it is shown that the family of such variations is very large. However, there are other admissible variations of $\sla(p)$ with regularity less than $C^3$ off of the poles and satisfying \eqref{eq:2ndsla}, see Remarks~\ref{re:seno} for precise references. 
\end{remark}

\section{The isoperimetric problem in sub-riemannian $3$-space forms}
\label{sec:isoperimetric}

In this section we apply our previous results to deduce some interesting consequences in relation to the \emph{isoperimetric problem}. We first introduce some basic definitions. 

Let $M$ be a Sasakian sub-Riemannian $3$-manifold. For any Borel set $\Om\subeq M$, the \emph{volume} of $\Om$ is the Riemannian volume $V(\Om)$ in $(M,g)$. This volume functional differs by a constant from the signed volume defined in \eqref{eq:volume} since its first variation coincides, up to sign, with \eqref{eq:vprima}. As in \cite{fssc} the \emph{perimeter} of $\Om$ can be introduced as
\begin{equation}
\label{eq:per}
P(\Om):=\sup\left\{\int_\Om\divv U\,dM;\,|U|\leq 1\right\},
\end{equation}
where $U$ ranges over $C^1$ horizontal vector fields with compact support on $M$. Here $dM$ and $\divv$ are the Riemannian volume and divergence in $(M,g)$, respectively. If $\Om$ is bounded by a $C^2$ surface $\Sg$, then we have $P(\Om)=A(\Sg)$ by the Riemannian divergence theorem. An \emph{isoperimetric region} or \emph{minimizer} in $M$ is a set $\Om\subeq M$ satisfying $P(\Om)\leq P(\Om')$, for any other $\Om'\subeq M$ with $V(\Om')=V(\Om)$.

If $M$ is homogeneous then isoperimetric regions of any volume exist and they are all bounded sets,  see \cite[Thm.~6.1 and Lem.~4.6]{galli-ritore}. As we pointed out in the Introduction the regularity of these minimizers is a delicate question, whose answer is far from being established. Nevertheless, if $\Om$ is a $C^2$ isoperimetric region, then it is clear that $\ptl\Om$ is a compact volume-preserving area-stationary surface. Thus, we can invoke Alexandrov's theorem (Theorem~\ref{th:alexandrov}) to get the following result, already known in the Heisenberg group $\mathbb{M}(0)$, see \cite[Thm.~7.2]{rr2}.

\begin{corollary}
\label{cor:isop}
The boundary of any $C^2$ isoperimetric region $\Om$ in $\e$ with $\kappa\leq 0$ is a spherical surface $\sla(p)$.
\end{corollary}

In order to get the same conclusion in $\e$ with $\kappa>0$ from Theorem~\ref{th:alexandrov}, we should prove first that $\ptl\Om$ is contained inside an open hemisphere of $\mathbb{S}^3$. Another possibility is to show that any $C^2$ compact volume-preserving area-stationary torus in $\e$ with $\kappa>0$ (see some examples in Remark~\ref{re:tekel}) is unstable under a volume constraint. In such a case, we would obtain that $\ptl\Om$ is topologically a sphere by Theorem~\ref{th:topology} and so, we would deduce Corollary~\ref{cor:isop} from Theorem~\ref{th:hopf}. This is a work in progress of the authors \cite{hr3}.

The previous results might suggest that the following conjecture is probably true: \emph{in the model spaces $\e$ any isoperimetric region is bounded by a spherical surface $\sla(p)$}. In the first Heisenberg group $\mathbb{M}(0)$ this conjecture was posed in 1982 by Pansu~\cite[p.~172]{pansu1} and has been supported by several partial results, see some references in the Introduction.  Our stability result in Theorem~\ref{th:stability} provides evidence that our conjecture for any $\e$ should hold. 

On the other hand, one may also ask if the spherical surfaces $\sla(p)$ will bound isoperimetric regions inside any $3$-space form. As happens in Riemannian geometry this is not true in general.

\begin{example}
\label{ex:mendicuti}
Consider the Sasakian flat cylinder $M:=\mathbb{M}(0)/G$ given in Example~\ref{ex:topology}. By construction, the projection map $\Pi:\mathbb{M}(0)\to M$ is an isometry when restricted to the slab $\rr^2\times (0,2\pi)$. We compare isoperimetrically the spherical surfaces $\sla(p)$ in $M$ with $p:=\Pi(o)$ and $\la>1/2$ with the family of tori $\Sg_R:=\mathbb{S}^1(R)\times\mathbb{S}^1=\Pi(\mathbb{S}^1(R)\times [0,2\pi])$, $R>0$. Since the volume in $\mathbb{M}(0)$ coincides with the Lebesgue measure in $\rr^3$, a straightforward computation from Lemma~\ref{lem:slaF} (iii), see also \cite[Rem.~7.3]{rr2}, implies that
\[
A(\sla(p))=\frac{\pi^2}{\la^3},\quad V(\mathcal{B}_\la(p))=\frac{3\pi^2}{8\la^4},
\]  
where $\mathcal{B}_\la(p)$ is the topological $3$-ball in $M$ enclosed by $\sla(p)$. On the other hand, we have
\[
A(\Sg_R)=4\pi^2R,\quad V(\Om_R)=2\pi^2R^2,
\]
where $\Om_R$ denotes the solid $3$-torus in $M$ bounded by $\Sg_R$. From the previous equalities it is easy to check that, for any $v\in((27/8)\pi^2,6\pi^2)$, the torus $\Sg_R$ enclosing volume $v$ has strictly less area than the sphere $\sla(p)$ of the same volume. As a consequence, none of the spheres $\sla(p)$ with $1/2<\la<1/\sqrt[4]{9}$ bounds a minimizer in $M$. Moreover, if we assume $C^2$ regularity of isoperimetric regions in $M$, then we can use the topological restriction in Theorem~\ref{th:topology} together with Hopf's theorem (Theorem~\ref{th:hopf}) to deduce the existence of isoperimetric tori in $M$.
\end{example}

In the previous example the spheres $\sla(p)$ fail to be isoperimetric when the enclosed volume is sufficiently large. We finish this section with the following question: are the spheres $\sla(p)$ of an arbitrary homogeneous Sasakian $3$-manifold $M$ boundaries of isoperimetric regions for small volumes?

\appendix
\section{Proof of the second variation formula}
\label{sec:appendix1}
\setcounter{theorem}{0}
\setcounter{equation}{0}  
\setcounter{subsection}{0}

In this section we prove Theorem~\ref{th:2ndsla}. We first derive a second variation formula for admissible variations possibly moving the singular set of arbitrary CMC surfaces in Sasakian sub-Riemannian $3$-manifolds.

\begin{theorem}
\label{2ndgeneral}
Let $\Sg$ be an orientable $C^2$ surface, possibly with boundary, immersed in a Sasakian sub-Riemannian $3$-mani\-fold $M$. Suppose that $\Sg-\Sg_0$ is $C^3$ and take a variation $\varphi:I\times\Sg\to M$ which is $C^3$ off of $\Sg_0$. Denote by $U:=uN+Q$ the associated velocity vector field, where $N$ is the unit normal to $\Sg$ and $Q:=U^\top$. If $\Sg$ has constant mean curvature $H$ and the variation is admissible, then the functional $A+2HV$ is twice differentiable at the origin, and we have
\begin{align}
\label{eq:gen2nd}
(A+2HV)''(0)&=\int_{\Sg}|N_{h}|^{-1}\left\{Z(u)^2-
\big(|B(Z)+S|^2+4\,(K-1)\,|N_{h}|^2\big)\,u^2\right\}d\Sg
\\
\nonumber
&+\int_{\Sg}\divv_\Sg\big\{\escpr{N,T}\,\big(1-\escpr{B(Z),S}\big)\,u^2\,Z\big\}\,d\Sg
\\
\nonumber
&+\int_{\Sg}\divv_\Sg\big\{\escpr{N,T}\,\big(2H\,\mnh\,u^2-w\big)\,S\big\}\,d\Sg
\\
\nonumber
&+\int_\Sg\divv_\Sg\big(\mnh\,W^\top\big)\,d\Sg+\int_{\Sg}\divv_\Sg\big(h_1 Z+h_2\,S\big)\,d\Sg,
\end{align}
provided all the terms are locally integrable with respect to $d\Sg$. In the previous formula $\{Z,S\}$ is the tangent orthonormal basis defined in \eqref{eq:nuh} and \eqref{eq:ese}, $B$ is the Riemannian shape operator, $K$ is the Webster scalar curvature, the function $w$ is the normal component of the acceleration vector field $W:=D_UU$, and $h_1,h_2$ are given by
\begin{align*}
h_1:=2&\,\big\{H\,\escpr{Q,Z}+\escpr{N,T}\,\escpr{D_SQ,Z}+\mnh^{-1}\,\escpr{Q,S}\,\big(\escpr{B(Z),S}+\escpr{N,T}^2\big)\big\}\,u
\\
&+\mnh\,\big(\escpr{Q,Z}\,\escpr{D_SQ,S}-\escpr{Q,S}\,\escpr{D_SQ,Z}\big)
\\
&+\escpr{N,T}\,\escpr{Q,Z}^2\,\big(1-\escpr{B(Z),S}\big)-\escpr{N,T}\,\escpr{Q,Z}\,\escpr{Q,S}\,\escpr{B(S),S},
\end{align*}
\vspace{-0,5cm}
\begin{align*}
\hspace{-0,85cm}h_2:=-2&\,\big\{H\,\escpr{Q,S}+\escpr{N,T}\,\escpr{D_ZQ,Z}-\mnh\,\escpr{Q,Z}\big\}\,u
\\
&+\mnh\,\big(\escpr{Q,S}\,\escpr{D_ZQ,Z}-\escpr{Q,Z}\,\escpr{D_ZQ,S}\big)
\\
&+2H\,\mnh\,\escpr{N,T}\,\escpr{Q,Z}^2+\escpr{N,T}\,\escpr{Q,Z}\,\escpr{Q,S}\,\big(1+\escpr{B(Z),S}\big).
\end{align*}
If $U=uN$ then the last integral in \eqref{eq:gen2nd} vanishes. If $\Sg$ has empty boundary and $\varphi$ is supported on $\Sg-\Sg_0$, then all the divergence terms in \eqref{eq:gen2nd} vanish.
\end{theorem}

The theorem was proved in \cite[Thm.~5.2]{rosales} for some particular variations with support in $\Sg-\Sg_0$, velocity vector $U=uN$ and acceleration vector $W=wN$. We will follow the arguments there, with the corresponding modifications that appear since $W^\top$ and $Q$ need not vanish on $\Sg$. The idea is to apply differentiation under the integral sign so that, after a long calculus, we will obtain \eqref{eq:gen2nd} with the help of Lemma~\ref{lem:aux4} below.

\begin{proof}[Proof of Theorem~\ref{2ndgeneral}]
We extend $U$ along the variation by $U(\varphi_{s}(p)):=(d/dt)|_{t=s}\,\varphi_{t}(p)$. Let $N$ be a vector field whose restriction to any $\Sg_{s}:=\var_s(\Sg)$ is a unit normal vector.

We first compute $A''(0)$.  Equation \eqref{eq:duis} yields
\[
A(s)=\int_{\Sg-\Sg_{0}}\big(\mnh\circ\var_{s}\big)
\,|\text{Jac}\,\var_{s}|\,d\Sg.
\]
Since the variation is admissible we can apply Lemma~\ref{lem:difint} to get 
\begin{equation}
\label{eq:2nd1} 
A''(0)=\int_{\Sg-\Sg_{0}}\left\{\mnh''(0)+2\,\mnh'(0)\,|\text{Jac}\,\var_{s}|'(0)+\mnh\,|\text{Jac}\,\var_{s}|''(0)\right\}d\Sg,
\end{equation}
where the derivatives are taken with respect to $s$. Let us compute the different terms in \eqref{eq:2nd1}. 

The calculus of $|\text{Jac}\,\var_{s}|'(0)$ and $|\text{Jac}\,\var_{s}|''(0)$ can be found in \cite[Sect.~9]{simon} and \cite[Lem.~5.4]{rosales}. On the one hand, we have
\begin{align}
\label{eq:jac1}
|\text{Jac}\,\var_{s}|'(0)=&\divv_{\Sg}U=(-2H_{R})\,u +\divv_\Sg Q=-\big(2H\,\mnh+\escpr{B(S),S}\big)\,u+\divv_\Sg Q
\end{align}
where $-2H_{R}:=\divv_{\Sg}N$ is the Riemannian mean curvature of $\Sg$, and we have used that $2H_{R}=\escpr{B(Z),Z}+\escpr{B(S),S}=2H\,\mnh+\escpr{B(S),S}$. On the other hand, if for any $p\in\Sg-\Sg_0$, we choose the orthonormal basis $\{e_1,e_2\}:=\{Z_p,S_p\}$ of $T_p\Sg$, and we take into account that $D_eU=e(u)\,N_p-u(p)\,B(e)+D_eQ$ for any $e\in T_p\Sg$, then we deduce
\begin{align} 
\label{eq:enemigo}
|\text{Jac}\,\var_{s}|''(0)&=\divv_{\Sg}W+(\divv_{\Sg}U)^2 +\sum_{i=1}^2|(D_{e_{i}}U)^\bot|^2
\\
\nonumber
&-\sum_{i=1}^2\escpr{R(U,e_{i})U,e_{i}}
-\sum_{i,j=1}^2\escpr{D_{e_{i}}U,e_{j}}\,\escpr{D_{e_{j}}U,e_{i}}
\\
\nonumber 
&=\divv_{\Sg}W+|\nabla_{\Sg}u|^2+|B(Q)|^2+2\,\escpr{\nabla_{\Sg}u,B(Q)}-\ric(U,U)
\\
\nonumber 
&+\escpr{R(U,N)U,N}+2\,\big(\escpr{D_Z U,Z}\,\escpr{D_S U,S}-\escpr{D_Z U,S}\,\escpr{D_S U,Z}\big),
\end{align}
where $\nabla_{\Sg}u$ is the gradient of $u$ relative to $\Sg$ and $\ric$ is the Ricci tensor in $(M,g)$. In the previous formula the curvature term $\escpr{R(U,N)U,N}-\ric(U,U)$ can be computed from \eqref{eq:ruvt}, \eqref{eq:ruvu} and \eqref{eq:ricvv}. After simplifying, this term equals
\[
4\,(1-K)\,\big(\escpr{N,T}^2\,|Q|^2+2\,\mnh\,\escpr{N,T}\,\escpr{Q,S}\,u+\mnh^2\,u^2\big)-|Q|^2-2u^2,
\]
which together with the fact that
\[
\escpr{D_{e_i} U,e_j}=-\escpr{B(e_i),e_j}\,u+\escpr{D_{e_i}Q,e_j},
\]
gives us the identity
\begin{align}
\label{eq:jac3} 
|\text{Jac}\,\var_{s}|''(0)&=\divv_{\Sg}W+|\nabla_{\Sg}u|^2+|B(Q)|^2+2\,\escpr{\nabla_\Sg u,B(Q)}-|Q|^2
\\
\nonumber 
&+2\,\big(2H\,\mnh\,\escpr{B(S),S}-\escpr{B(Z),S}^2-1\big)\,u^2
\\
\nonumber 
&+4\,(1-K)\,\big(\escpr{N,T}^2\,|Q|^2+2\,\mnh\,\escpr{N,T}\,\escpr{Q,S}\,u+\mnh^2\,u^2\big)
\\
\nonumber 
&+2\,\escpr{B(Z),S}\,\big(\escpr{D_Z Q,S}+\escpr{D_S
Q,Z}\big)\,u
\\
\nonumber 
&-2\,\escpr{B(S),S}\,\escpr{D_ZQ,Z}\,u-4H\,\mnh\,\escpr{D_SQ,S}\,u
\\
\nonumber 
&+2\,\big(\escpr{D_ZQ,Z}\,\escpr{D_SQ,S}-\escpr{D_Z
Q,S}\,\escpr{D_SQ,Z}\big).
\end{align}

Now we obtain expressions for $\mnh'(0)$ and $\mnh''(0)$.  From \eqref{eq:vmnh} and \eqref{eq:conmute} it follows that
\[
\mnh'(s)=U(\mnh)=\escpr{D_{U}N,\nuh}+\escpr{N,T}\,\escpr{U,Z}.
\]
The second equality in \eqref{eq:relations} together with the fact that 
\begin{equation}
\label{eq:dun}
D_{U}N=-\nabla_{\Sg}u-B(Q),
\end{equation}
implies that
\begin{equation}
\label{eq:dmnh}
|N_{h}|'(0)=-\escpr{N,T}\,\big(S(u)+\escpr{B(Q),S}-\escpr{Q,Z}\big).
\end{equation}
Moreover, we have
\begin{equation}
\label{eq:d2mnh1}
|N_{h}|''(0)=\escpr{D_{U}D_{U}N,\nuh}+\escpr{D_{U}N,D_{U}\nuh}+U\big(\escpr{N,T}\big)\,\escpr{Q,Z}
+\escpr{N,T}\,U\big(\escpr{U,Z}\big).
\end{equation}
By using \eqref{eq:dvnuh}, \eqref{eq:dun} and the third relation in \eqref{eq:relations}, we get
\begin{align}
\label{eq:d2mnh12}
\escpr{D_{U}N,D_{U}\nuh}&=\mnh^{-1}\,Z(u)^2+\escpr{N,T}\,Z(u)\,u
\\
\nonumber 
&+\mnh^{-1}\,\big(\escpr{N,T}^2\,\escpr{Q,S}+2\,\escpr{B(Q),Z}\big)\,Z(u)
\\
\nonumber
&+\mnh\,\escpr{Q,Z}\,S(u)+\escpr{N,T}\,\escpr{B(Q),Z}\,u
\\
\nonumber 
&+\mnh^{-1}\big(\escpr{B(Q),Z}^2+\escpr{N,T}^2\,\escpr{Q,S}\,\escpr{B(Q),Z}\big)+\mnh\,\escpr{Q,Z}\,\escpr{B(Q),S}.
\end{align}
It is also easy to check from \eqref{eq:vnt} that
\begin{equation}
\label{eq:d2mnh13}
U\big(\escpr{N,T}\big)=\mnh\,\big(S(u)+\escpr{B(Q),S}-\escpr{Q,Z}\big).
\end{equation}
On the other hand
\begin{align}
\label{eq:d2mnh14}
U\big(\escpr{U,Z}\big)&=\escpr{W,Z}+\escpr{D_UZ,\nuh}\,\escpr{U,\nuh}+\escpr{D_UZ,T}\,\escpr{U,T}
\\
\nonumber
&=\escpr{W,Z}-\escpr{Z,D_U\nuh}\,\escpr{U,\nuh}-\escpr{Z,D_UT}\,\escpr{U,T}
\\
\nonumber 
&=\escpr{W,Z}+\big(u+\mnh^{-1}\,\escpr{N,T}\,\escpr{Q,S}\big)\big(Z(u)+\escpr{B(Q),Z}+\escpr{Q,S}\big),
\end{align}
where the last equality comes from \eqref{eq:dvnuh} and \eqref{eq:relations}. It remains to compute $D_{U}D_{U}N$. For a fixed point $p\in\Sg-\Sg_{0}$ we denote $E_i(s):=e_i(\varphi_s)$. It is clear that $E_i(0)=e_i$ and $[U,E_i]=0$. Moreover, $\{E_1(s),E_2(s)\}$ is a basis of the tangent space to $\Sg_s$ at $\varphi_s(p)$. Note that $\{e_{1},e_{2},N_p\}$ is an orthonormal basis of $T_pM$. As a consequence
\[
D_{U}D_{U}N=\sum_{i=1}^2\escpr{D_{U}D_{U}N,E_{i}}\,E_{i}
+\escpr{D_{U}D_{U}N,N}\,N.
\]
Since $\escpr{N,E_{i}}=0$ and $[U,E_i]=0$, then we get
\begin{align*}
\escpr{D_{U}D_{U}N,E_{i}}&=-2\,\escpr{D_{U}N,D_{U}E_i}-\escpr{N,D_{U}D_{U}E_i}=-2\,\escpr{D_{U}N,D_{e_i}U}-\escpr{N,D_{U}D_{e_i}U}
\\
\nonumber
&=-2\,\escpr{D_{U}N,D_{e_{i}}U}+\escpr{R(U,E_{i})U,N}-\escpr{D_{e_{i}}W,N}
\\
\nonumber
&=\,2\,\escpr{\nabla_{\Sg}u,D_{e_{i}}U}+2\,\escpr{B(Q),D_{e_i}U}+\escpr{R(U,E_{i})U,N}-\escpr{D_{e_{i}}W,N},
\end{align*}
where we have employed \eqref{eq:dun}. Moreover, since $|N|^2=1$ on $\Sg$, we deduce
\[
\escpr{D_{U}D_{U}N,N}=-|D_{U}N|^2=-|\nabla_{\Sg}u|^2-|B(Q)|^2-2\,\escpr{\nabla_\Sg u,B(Q)}.
\]
Recall that $e_{1}=Z_{p}$ and $e_{2}=S_{p}$. Then, the previous equalities together with $\escpr{S,\nuh}=\escpr{N,T}$ lead us to the expression
\begin{align}
\label{eq:d2mnh11}
\escpr{D_{U}D_{U}N,\nuh}&=\escpr{N,T}\,\escpr{D_U D_UN,S}+\mnh\,\escpr{D_U D_U N,N}
\\
\nonumber 
&=\escpr{N,T}\,\big(\!-2\,\escpr{\nabla_\Sg u,B(S)}\,u+2\,\escpr{\nabla_\Sg u,D_S Q}-2\,\escpr{B(Q),B(S)}\,u
\\
\nonumber 
&+2\,\escpr{B(Q),D_SQ}+\escpr{R(U,S)U,N}-\escpr{D_SW,N}\big)
\\
\nonumber 
&-\mnh\,\big(|\nabla_\Sg u|^2+|B(Q)|^2+2\,\escpr{\nabla_\Sg u,B(Q)}\big).
\end{align}
The curvature term above can be derived from \eqref{eq:ruvt}, \eqref{eq:ruvu} and \eqref{eq:ricvv}, so that we obtain
\[
\escpr{R(U,S)U,N)}=4\,(K-1)\,\mnh\,\escpr{N,T}\,\escpr{Q,Z}^2-\escpr{Q,S}\,u.
\]

Equations \eqref{eq:d2mnh11}, \eqref{eq:d2mnh12}, \eqref{eq:d2mnh13} and \eqref{eq:d2mnh14} allow us to compute $\mnh''(0)$ from \eqref{eq:d2mnh1}. We use the resulting formula together with \eqref{eq:dmnh}, \eqref{eq:jac1} and \eqref{eq:jac3}. By taking into account the equalities
\begin{alignat*}{4}
&\quad \nabla_\Sg u&&=Z(u)\,Z+S(u)\,S,\hspace{1.8cm} B(S)&&=\escpr{B(Z),S}\,Z+\escpr{B(S),S}\,S,
\\
& \divv_\Sg Q&&=\escpr{D_ZQ,Z}+\escpr{D_SQ,S}, \quad (D_SQ)^\top&&=\escpr{D_SQ,Z}\,Z+\escpr{D_SQ,S}\,S,
\end{alignat*}
and simplifying, we obtain
\begin{align}
\label{eq:defini1}
&\mnh''(0)+2\,\mnh'(0)\,|\text{Jac}\,\var_{s}|'(0)+\mnh\,|\text{Jac}\,\var_{s}|''(0)
\\
\nonumber
&=\mnh^{-1}\,Z(u)^2+2\,\escpr{N,T}\,\big(1-\escpr{B(Z),S}\big)\,Z(u)\,u+4H\,\mnh\,\escpr{N,T}\,S(u)\,u
\\
\nonumber 
&+\mnh\,\divv_\Sg W-\escpr{N,T}\,\escpr{D_SW,N}+\escpr{N,T}\,\escpr{W,Z}
\\
\nonumber 
&+q_1\,u^2+q_2\,u+q_3\,Z(u)+q_4\,S(u)+q_5\,\escpr{D_ZQ,Z}+q_6\,\escpr{D_SQ,S}+q_7\,
\escpr{D_SQ,Z}
\\
\nonumber 
&+2\,\mnh\,\big(\escpr{D_ZQ,Z}\,\escpr{D_SQ,S}
-\escpr{D_ZQ,S}\,\escpr{D_SQ,Z}\big)+2\,\mnh\,\escpr{B(Z),S}\,\escpr{D_ZQ,S}\,u
\\
\nonumber 
&+\mnh^{-1}\,\escpr{B(Q),Z}\,\big(\escpr{B(Q),Z}
+2\,\escpr{N,T}^2\,\escpr{Q,S}\big)+2\,\mnh\,\escpr{Q,Z}\,\escpr{B(Q),S}
\\
\nonumber 
&+\mnh^{-1}\,\escpr{Q,S}^2-2\,\mnh\,|Q|^2+4\,(1-K)\,\mnh\,\escpr{N,T}^2\,\escpr{Q,S}^2,
\end{align}
where the functions $q_i$ with $i=1,\ldots,7$ are defined by
\begin{align}
\nonumber 
q_1&=4H\,\mnh^2\,\escpr{B(S),S}-2\,\mnh\,\escpr{B(Z),S}^2+4\,(1-K)\,\mnh^3-2\,\mnh,
\\
\nonumber 
q_2&=2\,\escpr{N,T}\,\big\{\escpr{B(Q),Z}\,\big(1-\escpr{B(Z),S}\big)-\escpr{Q,Z}\,\escpr{B(S),S}
\\
\nonumber 
&+2H\,\mnh\,\big(\escpr{B(Q),S}-\escpr{Q,Z}\big)+4\,(1-K)\,\mnh^2\,\escpr{Q,S}\big\}, 
\\
\nonumber 
q_3&=2\,\big(\escpr{N,T}\,\escpr{D_SQ,Z}+\mnh^{-1}\,\escpr{B(Q),Z}+\mnh^{-1}\,\escpr{N,T}^2\,\escpr{Q,S}\big),
\\
\nonumber 
q_4&=2\,\big(\mnh\,\escpr{Q,Z}-\escpr{N,T}\,\escpr{D_ZQ,Z}\big),
\\
\nonumber
q_5&=2\,\big\{\escpr{N,T}\,\big(\escpr{Q,Z}-\escpr{B(Q),S}\big)-\mnh\,\escpr{B(S),S}\,u\big\},
\\
\nonumber 
q_6&=2\,\escpr{N,T}\,\escpr{Q,Z}-4H\,\mnh^2\,u,
\\
\nonumber
q_7&=2\,\big(\escpr{N,T}\,\escpr{B(Q),Z}+\mnh\,\escpr{B(Z),S}\,u\big).
\end{align}

Note that equations \eqref{eq:2nd1} and \eqref{eq:defini1} provide the derivative $A''(0)$. Now, we proceed to the calculus of $V''(0)$. From \eqref{eq:vprima} we have
\[
V'(s)=\int_{\Sg_{s}}\escpr{U,N}\,d\Sg_{s}=\int_{\Sg}\big{(}\escpr{U,N}\circ\varphi_{s}\big{)}\,|\text{Jac}\,
\var_{s}|\,d\Sg,
\]
which implies that
\begin{equation}
\label{eq:vdosprima}
V''(0)=\int_\Sg\big\{\escpr{U,N}'(0)+u\,|\text{Jac}\,\varphi_s|'(0)\big\}\,d\Sg.
\end{equation}
Thus, from \eqref{eq:dun} and \eqref{eq:jac1} we deduce
\begin{align}
\label{eq:defini2}
\escpr{U,N}'(0)+u\,|\text{Jac}\,\var_{s}|'(0)&=w-\escpr{\nabla_\Sg u,Q}-\escpr{B(Q),Q}
\\
\nonumber
&+u\,\divv_\Sg Q-\big(2H\,\mnh+\escpr{B(S),S}\big)\,u^2.
\end{align}

We conclude from \eqref{eq:2nd1}, \eqref{eq:defini1},
\eqref{eq:vdosprima} and \eqref{eq:defini2} that
\begin{equation}
\label{eq:nemo} 
(A+2HV)''(0)=\int_{\Sg-\Sg_0}\beta\,d\Sg,
\end{equation}
where the function $\beta$ has the expression
\begin{align}
\label{eq:nemo2}
\beta&:=\mnh^{-1}\,Z(u)^2+2\,\escpr{N,T}\,\big(1-\escpr{B(Z),S}\big)\,Z(u)\,u 
\\
\nonumber 
&+4H\,\mnh\,\escpr{N,T}\,S(u)\,u+\big(q_1-4H^2\,\mnh-2H\,\escpr{B(S),S}\big)\,u^2
\\
\nonumber 
&+2Hw+\mnh\,\divv_\Sg W-\escpr{N,T}
\,\escpr{D_SW,N}+\escpr{N,T}\,\escpr{W,Z}
\\
\nonumber 
&+q_2\,u+\big(q_3-4H\,\escpr{Q,Z}\big)\,Z(u)
+\big(q_4-4H\,\escpr{Q,S}\big)\,S(u)
\\
\nonumber 
&+\divv_\Sg(2Hu\,Q)-2H\,\escpr{B(Q),Q}+q_5\,\escpr{D_ZQ,Z}+q_6\,\escpr{D_SQ,S}+q_7\,\escpr{D_SQ,Z}
\\
\nonumber 
&+2\,\mnh\,\big(\escpr{D_ZQ,Z}\,\escpr{D_SQ,S}-\escpr{D_ZQ,S}\,\escpr{D_SQ,Z}\big)+2\,\mnh\,\escpr{B(Z),S}\,\escpr{D_ZQ,S}\,u
\\
\nonumber 
&+\mnh^{-1}\,\escpr{B(Q),Z}\,\big(\escpr{B(Q),Z}+2\,\escpr{N,T}^2\,\escpr{Q,S}\big)+2\,\mnh\,\escpr{Q,Z}\,\escpr{B(Q),S}
\\
\nonumber 
&+\mnh^{-1}\,\escpr{Q,S}^2-2\,\mnh\,|Q|^2+4\,(1-K)\,\mnh\,\escpr{N,T}^2\,\escpr{Q,S}^2.
\end{align}
Let us simplify the terms containing $W$. From \eqref{eq:vmnh} and the fact that $\nuh^\top=\escpr{N,T}\,S$, it is easy to check that
\[
\divv_\Sg\big(\mnh\,W^\top\big)=\mnh\,\divv_\Sg W^\top-\escpr{N,T}\,\escpr{B(S),W}+\escpr{N,T}\,\escpr{W,Z},
\]
and so 
\begin{align*}
2Hw&+\mnh\,\divv_\Sg W-\escpr{N,T}
\,\escpr{D_SW,N}+\escpr{N,T}\,\escpr{W,Z}
\\
&=\divv_\Sg\big(\mnh\,W^\top\big)-\escpr{N,T}\,S(w)+\big(2H\,\escpr{N,T}^2-\mnh\,\escpr{B(S),S}\big)\,w.
\end{align*}
Now, we can use the computations below \cite[Eq.~(5.20)]{rosales} to infer that the first three lines of \eqref{eq:nemo2} equal
\begin{align}
\label{eq:nemo3}
|N_{h}|^{-1}&\big\{Z(u)^2-\big(|B(Z)+S|^2+4\,(K-1)\,|N_{h}|^2\big)\,u^2\big\}
\\
\nonumber 
&+\divv_\Sg\big\{\escpr{N,T}\,\big(1-\escpr{B(Z),S}\big)\,u^2\,Z\big\}
\\
\nonumber
&+\divv_\Sg\big\{\escpr{N,T}\,\big(2H\,\mnh\,u^2-w\big)\,S\big\}+\divv_\Sg\big(\mnh\,W^\top\big).
\end{align}
To prove the statement, it suffices to show that the remainder summands in \eqref{eq:nemo2} equal $\divv_\Sg\,(h_1\,Z+h_2\,S)$. This is a long but straightforward calculus similar to the one at the end of \cite[Proof of Thm.~5.2]{rosales}. Indeed, by using Lemma~\ref{lem:aux4} below and the formulas
\begin{align}
\label{useful}
S\big(\escpr{N,T}\big)&=\mnh\,\escpr{B(S),S}, \hspace{3cm} \ S\big(\mnh\big)=-\escpr{N,T}\,\escpr{B(S),S},
\\
\nonumber
Z\big(\escpr{N,T}\big)&=\mnh\,\big(\escpr{B(Z),S}-1\big), \hspace{2.01cm} \ Z\big(\mnh\big)=\escpr{N,T}\,\big(1-\escpr{B(Z),S}\big),
\\
\nonumber 
\escpr{D_SZ,S}&=\mnh^{-1}\,\escpr{N,T}\,\big(1+\escpr{B(Z),S}\big),\quad \ \!\escpr{D_ZS,Z}=-2H\,\escpr{N,T},
\\
\nonumber
Z\big(\escpr{B(Z),S}\big)&=4\,\mnh\,\escpr{N,T}\,(1-K-H^2)-2\,\mnh^{-1}\,\escpr{N,T}\,\escpr{B(Z),S}\,\big(1+\escpr{B(Z),S}\big),
\end{align}
we can deduce \eqref{eq:gen2nd} from \eqref{eq:nemo}, \eqref{eq:nemo2} and \eqref{eq:nemo3}. This finishes the proof.
\end{proof}

\begin{lemma}
\label{lem:aux4}{$($\cite[Lem.~5.5]{rosales}$)$.}
Let $\Sg$ be an orientable $C^2$ surface immersed in a Sasakian sub-Riemannian $3$-manifold $M$. For any $\phi\in C^1(\Sg)$ we have the following equalities in $\Sg-\Sg_{0}$
\begin{align*}
\divv_{\Sg}(\phi\,Z)&=Z(\phi)+|N_{h}|^{-1}\,\escpr{N,T}\,\big(1+\escpr{B(Z),S}\big)\,\phi, \quad \divv_{\Sg}(\phi\,S)=S(\phi)-2H\,\escpr{N,T}\,\phi.
\end{align*}
\end{lemma}

\begin{remarks}
\label{re:seno}
1. The regularity hypotheses in Theorem~\ref{2ndgeneral} are necessary to compute the different terms in the proof for an arbitrary variation $\varphi$. However, it is possible to derive the second variation for some particular variations of a surface $\Sg$ with less regularity and $\Sg_0=\emptyset$, see for instance \cite[Thm.~14.5]{dgn}, \cite[Thm.~3.7]{hrr}, \cite[Thm.~5.2]{rosales}, \cite[Thm.~7.3]{galli} and \cite[Thm.~4.1]{galli-ritore2}. In our case the $C^3$ regularity of $\Sg-\Sg_0$ does not suppose loss of generality since we are only interested in stability properties of the spherical surfaces $\sla(p)$.

2. The fact that the variation $\varphi$ is admissible is only used to differentiate under the integral sign in \eqref{eq:duis}. If a variation is not admissible then $A''(0)$ may exist or not and, in case of existence, it is not easy to compute it. Some examples of these situations were discussed in the Heisenberg group~\cite[Prop.~3.11]{hrr} and in pseudo-Hermitian $3$-manifolds, see \cite[Ex.~4.3]{ch2} and \cite[Lem.~7.7]{galli}. 
\end{remarks}

Now, we proceed to obtain the second variation formula in Theorem~\ref{th:2ndsla}. For the proof we will show that all the divergence terms in \eqref{eq:gen2nd} vanish for variations of a spherical surface $\sla(p)$. The main ingredients are the integrability of $\mnh^{-1}$ in Corollary~\ref{cor:integrability} and a divergence theorem which extends the classical situation valid for vector fields supported off of the poles. 

\begin{lemma}
\label{divergence} 
Let $X$ be a bounded and tangent $C^1$ vector field on $\sla(p)$ minus the poles such that $\divv_{\sla(p)}X$ is integrable with respect to $d\sla(p)$.  Then, we have
\[
\int_{\sla(p)}\divv_{\sla(p)}X\,d\sla(p)=0.
\]
\end{lemma}

\begin{proof}
For any $\eps>0$ small enough, let $\Sg_\eps:=\sla(p)-(D_1(\eps)\cup D_2(\eps))$, where $D_i(\eps)$, $i=1,2$, are small spherical caps centered at the poles. From the divergence theorem we obtain
\begin{equation}
\label{eq:divriem}
\int_{\Sg_\eps}\divv_{\sla(p)}X\,d\sla(p)=-\sum_{i=1}^2\,\int_{\ptl D_i(\eps)}\escpr{X,\eta_i}\,dl,
\end{equation}
where $\eta_i$ is the inner conormal to $\ptl D_i(\eps)$ in $\sla(p)$, and $dl$ in the Riemannian element of length. Note that the left hand side term above tends to $\int_{\sla(p)}\divv_{\sla(p)}X\,d\sla(p)$ when $\eps\to 0$ by the dominated convergence theorem. On the other hand, since $X$ is bounded, the right hand side term goes to zero when $\eps\to 0$. This proves the claim.
\end{proof}

\begin{proof}[Proof of Theorem~\ref{th:2ndsla}]
We know that $\sla(p)$ is a $C^\infty$ surface off of the poles with constant mean curvature $\lambda$. Hence, for an admissible variation $\varphi$ which is also $C^3$ off of the poles, it is possible to apply Theorem~\ref{2ndgeneral} to compute $(A+2\lambda V)''(0)$. To prove the claim it suffices to use \eqref{eq:potential}, and to see that the divergence terms in \eqref{eq:gen2nd} are all equal to zero.

Consider the $C^1$ tangent vector field $X:=\escpr{N,T}\,\big(1-\escpr{B(Z),S}\big)\,u^2\,Z$ defined on $\sla(p)$ minus the poles. Recall that $\escpr{B(Z),S}=(1-\tau^2)\,\mnh^2$ by Lemma~\ref{lem:slaF} (iii) and so, $X$ is bounded. Moreover, by using Lemma~\ref{lem:aux4} together with some of the equalities in \eqref{useful} and the integrability of $\mnh^{-1}$ in Corollary~\ref{cor:integrability}, we deduce that $\divv_{\sla(p)}X$ is integrable with respect to $d\sla(p)$. So, the integral of $\divv_{\sla(p)}X$ vanishes as a consequence of Lemma~\ref{divergence}. The same argument holds for the vector field $X:=\escpr{N,T}\,\big(2H\,\mnh\,u^2-w\big)\,S$. Note that the integral of $\divv_{\sla(p)}\,(\mnh\,W^\top)$ also vanishes by the divergence theorem since $\mnh\,W^\top$ is a Lipschitz tangent vector field on $\sla(p)$. Hence, to finish the proof we must show that 
\begin{equation}
\label{eq:sisi}
\int_{\sla(p)}\divv_{\sla(p)}X\,d\sla(p)=0,
\end{equation}
where $X:=h_1\,Z+h_2\,S$, and the functions $h_1$, $h_2$ are defined below equation \eqref{eq:gen2nd}. At the end of the proof of Theorem~\ref{2ndgeneral} we obtained that $\divv_{\sla(p)}X$ is equal to the function $\beta$ in \eqref{eq:nemo2} by removing the first three lines. This long expression only contains bounded functions and the term $\mnh^{-1}$, which is integrable on $\sla(p)$. From here, it follows that $\divv_{\sla(p)}X$ is integrable with respect to $d\sla(p)$. Furthermore, from the definitions of $h_1$ and $h_2$ we deduce that all the terms contained in $X$ are bounded vector fields off of the poles, with the exception of 
\[
2\,\mnh^{-1}\,\escpr{Q,S}\,\big(\escpr{B(Z),S}+\escpr{N,T}^2\big)\,u\,Z.
\]
Since $\escpr{B(Z),S}=(1-\tau^2)\,\mnh^2$ and $\mnh^2+\escpr{N,T}^2=1$, then the unique unbounded vector in the previous formula is $2\,\mnh^{-1}\,\escpr{Q,S}\,u\,Z$. Hence, we cannot apply directly Lemma~\ref{divergence} to infer equality \eqref{eq:sisi}. However, the conclusion of Lemma~\ref{divergence} also holds provided the right hand side term in \eqref{eq:divriem} tends to zero as $\eps\to 0$. So, to prove \eqref{eq:sisi} we only have to see that
\begin{equation}
\label{eq:storm}
\lim_{\eps\to 0}\,\int_{\ptl D_i(\eps)}\mnh^{-1}\,\escpr{Q,S}\,\escpr{Z,\eta_i}\,u\,dl=0, \quad i=1,2,
\end{equation}
where $D_i(\eps)$, $i=1,2$, are small spherical caps centered at the poles, and $\eta_i$ is a unit vector tangent to $\sla(p)$ and normal to $\ptl D_i(\eps)$.

Fix a positive orthonormal basis $\{e_1,e_2\}$ in the contact plane $\mathcal{H}_p$. Let $F(\theta,s):=\ga_\theta(s)$ be the flow of CC-geodesics of curvature $\lambda$ with $\ga_\theta(0)=p$ and $\dot{\ga}_\theta(0)=(\cos\theta)\,e_1+(\sin\theta)\,e_2$. As in Lemma~\ref{lem:slaF} we consider the unit normal along $\sla(p)$ such that $Z=\dot{\ga}_\theta$ off of the poles. For $\eps>0$ small enough, the curve $\ptl D_1(\eps)$ is parameterized by the curve $\alpha_\eps:[0,2\pi]\to\sla(p)$ given by $\alpha_\eps(\theta):=F(\theta,\eps)$. From equations \eqref{eq:esepolar} and \eqref{eq:slan}, we get
\[
\dot{\alpha}_\eps(\theta)=\frac{\ptl F}{\ptl\theta}(\theta,s)=V_\theta(\eps)=-\big(\lambda\,v(\eps)\big)\,Z-\sqrt{v(\eps)^2+(v'(\eps)/2)^2}\,S, 
\]
where $v(\eps):=\sin^2(\tau\eps)/\tau^2$ and $\tau:=\sqrt{\la^2+\kappa}$. The previous formula implies that
\[
\eta_1:=\frac{\sqrt{v(\eps)^2+(v'(\eps)/2)^2}\,Z-\big(\lambda\,v(\eps)\big)\,S}{|V_\theta(\eps)|}
\]
is a unit normal to $\ptl D_1(\eps)$ tangent to $\sla(p)$. Thus, we have
\[
\int_{\ptl D_1(\eps)}\mnh^{-1}\,\escpr{Q,S}\,\escpr{Z,\eta_1}\,u\,dl=\frac{v(\eps)^2+(v'(\eps)/2)^2}{v(\eps)}\int_0^{2\pi}\big(\escpr{Q,S}\,u\big)(\alpha_\eps(\theta))\,d\theta.
\]
On the one hand, it is straightforward to check that 
\[
\lim_{\eps\to 0}\frac{v(\eps)^2+(v'(\eps)/2)^2}{v(\eps)}=1.
\]
On the other hand, if we denote $h_\eps(\theta):=(\escpr{Q,S}\,u)(\alpha_\eps(\theta))$, then the continuity of $Q$ and $u$ in $\sla(p)$ together with equality $S=\escpr{N,T}\,\nuh-\mnh\,T=-\escpr{N,T}\,J(\dot{\ga}_\theta)-\mnh\,T$, gives us $|h_\eps(\theta)|\leq c$, for some $c>0$ not depending on $\eps$, and $\lim_{\eps\to 0}h_\eps(\theta)=h(\theta)$, where
\[
h(\theta):=u(p)\,\big(\escpr{Q_p,e_1}\sin\theta-\escpr{Q_p,e_2}\cos\theta\big).
\]
Thus, we can apply the dominated convergence theorem to conclude that
\[
\lim_{\eps\to 0}\int_0^{2\pi}\big(\escpr{Q,S}\,u\big)(\alpha_\eps(\theta))\,d\theta=\int_0^{2\pi}h(\theta)\,d\theta=0.
\]
This yields equation \eqref{eq:storm} for $i=1$. The case $i=2$ is similar. This proves \eqref{eq:sisi} and finishes the proof of the theorem.
\end{proof}

\section{Examples of admissible variations}
\label{sec:appendix2}
\renewcommand{\thesection}{B}
\setcounter{theorem}{0}
\setcounter{equation}{0}  
\setcounter{subsection}{0}

The notion of admissible variation was introduced in Definition~\ref{def:admissible}. For such variations we can apply Lemma~\ref{lem:difint} to differentiate under the integral sign twice in \eqref{eq:duis}. In this appendix we will construct several admissible variations. We start with some immediate examples.

\begin{example}
\label{ex:killing}
Let $U$ be a sub-Riemannian Killing field on $M$. This means that any diffeomorphism $\varphi_s$ of the associated one-parameter group is an isometry of $M$. So, $|\text{Jac}\,\varphi_s|_p=1$ and $\mnh_p(s)=\mnh(p)$ for any $s\in I$ and any $p\in\Sg$. From here it is immediate that the resulting variation of $\Sg$ is admissible and the area functional is constant. Observe also that these variations need not fix the singular set $\Sg_0$. 
\end{example}

\begin{example}
\label{ex:trivial}
Suppose that $\Sg$ has $C^3$ regular set $\Sg-\Sg_0$. Consider a variation $\varphi$ which is $C^3$ on $\Sg-\Sg_0$, and such that $\mnh_p(s)>0$ for any $p\in\Sg-\Sg_0$ and any $s\in I$. Geometrically, this means that any $\varphi_s$ preserves the regular set, i.e., $\varphi_s(\Sg-\Sg_0)\subeq\Sg_s-(\Sg_s)_0$. In particular, the function $s\mapsto f(s,p):=\mnh_p(s)\,|\text{Jac}\,\varphi_s|_p$ is $C^2$ for any $p\in\Sg-\Sg_0$, which implies conditions (i), (ii) and (iv) in Definition~\ref{def:admissible}. Moreover, we have
\[
\frac{\ptl^2 f}{\ptl s^2}(s,p)=\mnh_p''(s)\,|\text{Jac}\,\varphi_s|_p+2\,\mnh_p'(s)\,|\text{Jac}\,\varphi_s|_p'(s)+\mnh_p(s)\,|\text{Jac}\,\varphi_s|_p''(s).
\]
So, a possible way of constructing admissible variations is to bound the different terms in the previous formula by locally integrable functions not depending on $s$. This is clear for $|\text{Jac}\,\varphi_s|'_p$ and $|\text{Jac}\,\varphi_s|''_p$ since they are continuous on $I\times\Sg$. On the other hand, note that
\[
\mnh_p'(s)=U_{\varphi_s(p)}(\mnh)=\escpr{D_UN_h,\nuh}
\big(\varphi_s(p)\big),
\]
whereas
\[
\mnh_p''(s)=\escpr{D_UD_UN_h,\nuh}\big(\varphi_s(p)\big)+ \escpr{D_UN_h,D_U\nuh}\big(\varphi_s(p)\big).
\]
From the expression for $D_U\nuh$ in \eqref{eq:dvnuh} we infer that, if there is $h:\Sg\to\rr$ locally integrable with respect to $d\Sg$, and such that $\mnh^{-1}_p(s)\leq h(p)$ for any $p\in\Sg-\Sg_0$ and any $s\in I$, then the restriction of $\varphi$ to $I'\times\Sg$ is admissible for any $I'\sub\sub I$. In particular, if $\varphi$ is $C^3$ on $\Sg-\Sg_0$ and compactly supported on $\Sg-\Sg_0$, then there is  $I'\sub\sub I$ such that the restriction of $\varphi$ to $I'\times\Sg$ is admissible.
\end{example}

The analytic condition $\mnh^{-1}_p(s)\leq h(p)$, which is sufficient to get an admissible variation, was found by Montefalcone~\cite[Eqs.~(25) and (26)]{montefalcone2} in the setting of Carnot groups.  Since this condition is not easy to verify in practice, we are led to produce admissible variation in more geometric ways. In precise terms we will focus on normal and vertical admissible variations moving $\Sg_0$. We need some notation and facts that will be useful in the sequel.

Let $\Sg$ be an oriented $C^2$ surface immersed in $M$. For any $C^1$ vector field $U$ with compact support on $\Sg$ we consider the $C^1$ variation $\varphi:I\times\Sg\to M$ given by $\varphi_s(p):=\exp_p(s\,U_p)$, where $I$ is an open interval containing $0$ and $\exp_p$ denotes the Riemannian exponential map of $M$ at $p$. Take a point $p\in\Sg$ and a vector $e\in T_p\Sg$. Let $\alpha:(-\eps_0,\eps_0)\to\Sg$ be a $C^1$ curve with $\alpha(0)=p$ and $\dot{\alpha}(0)=e$. We define the map $F:(-\eps_0,\eps_0)\times I\to M$ by $F(\eps,s):=\varphi_s(\alpha(\eps))$. By Remark~\ref{re:ricase} we have that $E(s):=(\ptl F/\ptl\eps)(0,s)=e(\varphi_s)$ is a $C^\infty$ vector field along the Riemannian geodesic $\ga_p(s):=\exp_p(s\,U_p)$ satisfying $[\dot{\ga}_p,E]=0$ and the Jacobi equation \eqref{eq:rijacobi}. Note also that $E(0)=e$ and $E'(0)=D_eU$. Fix a basis $\{e_1,e_2\}$ in $T_p\Sg$ and let $E_i(s)=e_i(\varphi_s)$, $i=1,2$, be the associated Jacobi fields. Then $\{E_1(s),E_2(s)\}$ is a basis of $T_{\ga_p(s)}\Sg$ and 
\begin{equation}
\label{eq:jacobiano1}
|\text{Jac}\,\varphi_s|_p=|E_1\times E_2|(s)
=\big(|E_1|^2\,|E_2|^2-\escpr{E_1,E_2}\big)^{1/2}(s), 
\end{equation}
which is a $C^\infty$ function along $\ga_p$. Moreover, the unit normal of $\Sg_s:=\varphi_s(\Sg)$ at $\varphi_s(p)$ is 
\begin{equation}
\label{eq:jacobiano2}
N_p(s)=\frac{E_1\times E_2}{|E_1\times E_2|}\,(s), 
\end{equation}
which is also $C^\infty$ along $\ga_p$. In the two previous formulas the cross product $\times$ is taken with respect to an orthonormal basis of the tangent space to $M$ along $\ga_p(s)$. From the $C^\infty$ regularity of $|\text{Jac}\,\varphi_s|_p$ and $N_p(s)$ it follows that, if $\mnh_p(s)>0$ for any $p\in\Sg-\Sg_0$ and any $s\in I$, then the conditions (i), (ii) and (iv) in Definition~\ref{def:admissible} hold. 

We are now ready to prove that the deformation of a compact surface $\Sg$ by means of Riemannian parallel surfaces is an admissible variation.

\begin{lemma}
\label{lem:parallels}
Let $\Sg$ be a compact, oriented $C^2$ surface immersed inside a Sasakian sub-Riemannian $3$-manifold $M$. Then, there is an open interval $I'\sub\rr$ such that the variation $\varphi:I'\times\Sg\to M$ defined by $\varphi_s(p):=\exp_p(sN_p)$ is admissible.
\end{lemma}

\begin{proof}
Take a point $p\in\Sg$ and a vector $e\in T_p\Sg$. Let $E(s):=e(\varphi_s)$ be the associated Jacobi field along $\ga_p(s):=\exp_p(sN_p)$. From equation~\eqref{eq:rijacobi} we get that $\escpr{E,\dot{\ga}_p}$ is an affine function along $\ga_p$. Thus equalities $E(0)=e$ and $E'(0)=D_eN$ give us $\escpr{E,\dot{\ga}_p}=0$ along $\ga_p$. As $e$ is an arbitrary vector in $T_p\Sg$, the unit normal to $\Sg_s$ at $\varphi_s(p)$ is $N_p(s)=\dot{\ga}_p(s)$. It follows that $\mnh^2_p(s)=1-\escpr{\dot{\ga}_p,T}^2(s)=\mnh^2(p)$ since $\escpr{\dot{\ga}_p,T}$ is constant along $\ga_p$. So, we have $\mnh_p(s)\,|\text{Jac}\,\varphi_s|_p=\mnh(p)\,|\text{Jac}\,\varphi_s|_p$. Hence conditions (i), (ii) and (iv) in Definition~\ref{def:admissible} are satisfied. On the other hand, from equation~\eqref{eq:enemigo} and the compactness of $\Sg$, the derivative $|\text{Jac}\,\varphi_s|_p''(s)$ is uniformly bounded as a function of $s\in I'$ and $p\in\Sg$ provided $I'$ is a small open interval containing $0$. This finishes the proof.
\end{proof}

The previous lemma together with Example~\ref{ex:trivial} allows us to construct more general admissible variations based on parallel surfaces near the singular set.

\begin{corollary}
\label{cor:parallels}
Let $M$ be a Sasakian sub-Riemannian $3$-manifold and $\Sg$ an oriented $C^2$ surface immersed in $M$ with $C^3$ regular set $\Sg-\Sg_0$. Consider a variation $\varphi:I\times\Sg\to M$ of $\Sg$ satisfying:
\begin{itemize} 
\item[(i)] there is $O\sub\sub\Sg$ with $\Sg_0\sub O$ such that $\varphi_s(p)=\exp_p(sN_p)$ for $s\in I$ and $p\in O$,
\item[(ii)] the restriction of $\varphi$ to $(\Sg-O)\times I$ is of class $C^3$. 
\end{itemize}
Then, there is an interval $I'\sub\sub I$ such that the restriction of $\varphi$ to $I'\times\Sg$ is admissible.
\end{corollary}

By a \emph{vertical variation} of $\Sg$ we mean one of the form $\varphi_s(p):=\exp_p(s\,u(p)\,T_p)$, where $u\in C_0^1(\Sg)$. This means that the surfaces $\Sg_s:=\varphi_s(\Sg)$ are all vertical graphs over $\Sg$. For the case $u=1$ the variation is admissible by Example~\ref{ex:killing} since $T$ is a sub-Riemannian Killing field and $\varphi_s$ coincides with the one-parameter group of $T$. In the next result we provide more examples of admissible vertical variations. Recall that $\{Z,S\}$ is the tangent basis to $\Sg-\Sg_0$ defined in \eqref{eq:nuh} and \eqref{eq:ese}. 

\begin{lemma} 
\label{lem:vertvar}
Let $\Sg$ be an oriented $C^2$ surface immersed inside a Sasakian sub-Riemannian $3$-mani\-fold $M$. Suppose that $\escpr{N,T}$ does not vanish along $\Sg$ and that $\mnh^{-1}$ is locally integrable with respect to $d\Sg$. Let $u\in C^1_0(\Sg)$ such that $|S(u)|\leq h\,|Z(u)|$ in $\Sg-\Sg_0$, for some bounded function $h$. Then, the vertical variation $\varphi_s(p):=\exp_p(s\,u(p)\,T_p)$ is admissible.
\end{lemma}

\begin{proof}
Take a point $p\in \Sg$. The Riemannian geodesic $\ga_p(s):=\exp_p(s\,u(p)\,T_p)$ satisfies the equality $\ga_p(s)=\alpha_p(s\,u(p))$, where $\alpha_p$ is the integral curve of $T$ through $p$. In particular, we have $\dot{\ga}_p(s)=u(p)\,T_{\ga_p(s)}$. Let $\{X(s),Y(s)\}$ be a positive orthonormal basis of the horizontal plane at $\ga_p(s)$ obtained by parallel transport of a similar basis $\{X_p,Y_p\}$ of $\h_p$. For any vector $e\in T_p\Sg$, we know that the associated vector field $E(s):=e(\varphi_s)$ satisfies the Jacobi equation \eqref{eq:rijacobi}. By using \eqref{eq:ruvt} this equation reads 
$E''(s)+u(p)^2\,E(s)_h=0$. Hence, an easy integration together with equalities $E(0)=e$ and $E'(0)=e(u)\,T_p+u(p)\,J(e)$ implies that
\begin{equation}
\label{eq:E}
E(s)=x(s)\,X(s)+y(s)\,Y(s)+t(s)\,T(s),
\end{equation}
where $T(s):=T_{\ga_{p}(s)}$, and the functions $x(s)$, $y(s)$, $t(s)$ are given by
\begin{align}
\label{eq:verticalpart}
\nonumber x(s)&:=\escpr{e,X_p}\,\cos(u(p)\,s)-\escpr{e,Y_p}\,\sin(u(p)\,s),
\\
y(s)&:=\escpr{e,Y_p}\,\cos(u(p)\,s)+\escpr{e,X_p}\,\sin(u(p)\,s),
\\
\nonumber 
t(s)&:=e(u)\,s+\escpr{e,T_p}.
\end{align}

Now suppose that $p\in\Sg-\Sg_0$ and consider the orthonormal basis $\{e_1,e_2\}$ of $T_p\Sg$ defined by $e_1=Z_p$ and $e_2=S_p$. Let $E_i(s):=e_i(\varphi_s)$, $i=1,2$, be the associated Jacobi fields along $\ga_p$. From equations \eqref{eq:jacobiano1} and \eqref{eq:jacobiano2}, we get
\begin{equation*}
f(s,p):=\mnh_p(s)\,|\text{Jac}\,\varphi_s|_p=|(E_1(s)\times E_2(s))_h|,
\end{equation*}
where the cross product is taken with respect to the basis $\{X(s),Y(s),T(s)\}$. Hence, a straightforward computation from \eqref{eq:E} and \eqref{eq:verticalpart} shows that
\[
f(s,p)=Q_p(s)^{1/2}, \quad Q_p(s):=a_p\,s^2+b_p\,s+c_p,
\]
where
\begin{equation*}
a_p:=\escpr{N_p,T_p}^2\,Z_p(u)^2+S_p(u)^2, \quad 
b_p:=-2\,\mnh(p)\,S_p(u),\quad
c_p:=\mnh^2(p).
\end{equation*}
Note that the discriminant $\text{disc}(Q_p)$ equals
\[
-4\,\mnh^2(p)\,\escpr{N_p,T_p}^2\,Z_p(u)^2,
\]
which is less than or equal to $0$. Moreover, the fact that $\escpr{N,T}$ never vanishes on $\Sg$ together with inequality $|S_p(u)|\leq h(p)\,|Z_p(u)|$ implies that $\text{disc}(Q_p)=0$ if and only if $Z_p(u)=S_p(u)=0$. From here we deduce that $Q_p(s)>0$ for any $p\in\Sg-\Sg_0$ and any $s\in\rr$. In particular, for any $p\in\Sg-\Sg_0$, the function $s\mapsto f(s,p)$ is $C^\infty$ and so, the conditions (i), (ii) and (iv) in Definition~\ref{def:admissible} are satisfied. 

Finally, take $p\in\text{supp}(u)\cap(\Sg-\Sg_0)$ and suppose $a_p\neq 0$. Since $Q_p(s)\geq Q_p(-b_p/(2a_p))=-\text{disc}(Q_p)/(4a_p)$ and $|S_p(u)|\leq h(p)\,|Z_p(u)|$ with $h$ bounded, we obtain
\begin{align*}
\left|\frac{\ptl^2 f}{\ptl s^2}(s,p)\right|&=-\frac{\text{disc}(Q_p)}{4\,Q_p(s)^{3/2}}
\leq\frac{2\,a_p^{3/2}}{(-\text{disc}(Q_p))^{1/2}}
=\frac{\,\big(\escpr{N_p,T_p}^2\,Z_p(u)^2+S_p(u)^2\big)^{3/2}}
{\mnh(p)\,|\escpr{N_p,T_p}|\,|Z_p(u)|}
\\
&\leq \frac{\,\big(\escpr{N_p,T_p}^2+h(p)^2\big)^{3/2}\,Z_p(u)^2 }{\mnh(p)\,|\escpr{N_p,T_p}|}
\leq\frac{C\,\escpr{(\nabla_\Sg u)_p,Z_p}^2}{\mnh(p)}\leq C'\,\mnh^{-1}(p),
\end{align*}
for some constants $C,C'>0$. The previous inequality also holds if $a_p=0$ since, in that case, $b_p=0$ and $(\ptl^2 f/\ptl s^2)(s,p)=0$ for any $s\in\rr$. From the hypothesis that $\mnh^{-1}$ is locally integrable with respect to $d\Sg$ we conclude that condition (iii) in Definition~\ref{def:admissible} holds, proving the claim.
\end{proof}

Note that $\escpr{N,T}\neq 0$ in small neighborhoods of $\Sg_0$. Hence we can combine Lemma~\ref{lem:vertvar} with Example~\ref{ex:trivial} to construct admissible variations based on suitable vertical deformations near $\Sg_0$.

\begin{corollary}
\label{cor:vertical}
Let $M$ be a Sasakian sub-Riemannian $3$-manifold and $\Sg$ an oriented $C^2$ surface immersed in $M$. Suppose that $\Sg-\Sg_0$ is $C^3$ and the function $\mnh^{-1}$ is locally integrable with respect to $d\Sg$. Consider a variation $\varphi:I\times\Sg\to M$ satisfying:
\begin{itemize} 
\item[(i)] there is a small open neighborhood $O$ of $\Sg_0$, a function $u\in C^1(O)$ with bounded gradient, and a function $h$ bounded on $O-\Sg_0$, such that $|S_p(u)|\leq h(p)\,|Z_p(u)|$ for any $p\in O-\Sg_0$, and $\varphi_s(p)=\exp_p(s\,u(p)\,T_p)$ for any $s\in I$ and any $p\in O$,
\item[(ii)] the restriction of $\varphi$ to $(\Sg-O)\times I$ is of class $C^3$. 
\end{itemize}
Then, there is an interval $I'\sub\sub I$ such that the restriction of $\varphi$ to $I'\times\Sg$ is admissible.
\end{corollary}

The previous results provide admissible variations of a $C^2$ volume-preserving area-stationary surface $\Sg$ with $\Sg_0\neq\emptyset$ inside a $3$-dimensional space form. As we proved in Theorems~\ref{th:isolated} and \ref{th:cmula}, such a surface is either a plane $\mathcal{L}_\la(p)$ as in \eqref{eq:lla}, a spherical surface $\sla(p)$ as in \eqref{eq:sla}, or a surface $\mathcal{C}_{\mu,\la}(\Ga)$ as in Example~\ref{ex:cmula}. Since all these surface are $C^\infty$ off of the singular set, then Corollary~\ref{cor:parallels} implies that $C^3$ perturbations of the deformation of $\Sg$ by Riemannian parallels are admissible variations. On the other hand, the construction in Corollary~\ref{cor:vertical} only applies for $\mathcal{L}_\la(p)$ and $\sla(p)$ since the function $\mnh^{-1}$ is locally integrable on these surfaces, whereas it is not on $\mathcal{C}_{\mu,\la}(\Ga)$. In Example~\ref{ex:hecho} below we give a particular but important family of functions in $\sla(p)$ satisfying the condition $|S(u)|\leq h\,|Z(u)|$, thus providing examples of admissible vertical variations. We conclude that the class of variations of a spherical surface $\sla(p)$ for which the second variation formula and the stability results in Section~\ref{sec:stability} are valid is very large.

\begin{example}
\label{ex:hecho}
Let $\sla(p)$ be a spherical surface inside a $3$-dimensional space form of curvature $\kappa$. We consider the immersion $F(\theta,s):=\ga_\theta(s)$ introduced in Lemma~\ref{lem:slaF}, and the associated unit normal $N$. We know from Corollary~\ref{cor:integrability} that $\mnh^{-1}$ is integrable with respect to $d\sla(p)$. For any $u\in C^1(\sla(p))$ we denote $\overline{u}:=u\circ F$. By Lemma~\ref{lem:slaF} (ii) and equation \eqref{eq:esepolar}, we have
\begin{align*}
Z(u)&=\frac{\ptl\overline{u}}{\ptl s}(\theta,s), \quad S(u)=\frac{-1}{\sqrt{v(s)^2+(v'(s)/2)^2}}\,\frac{\ptl\overline{u}}{\ptl\theta}(\theta,s)-\la\,\mnh (s)\,\frac{\ptl\overline{u}}{\ptl s}(\theta,s),
\end{align*} 
where $v(s):=\sin^2(\tau s)/\tau^2$ and $\tau:=\sqrt{\la^2+\kappa}$. From here we deduce that, if $\ptl\overline{u}/\ptl\theta=0$, then $S(u)=h\,Z(u)$ for some function $h$ bounded off of the poles. In particular, any variation $\varphi$ of $\sla(p)$ whose restriction to small neighborhoods of the poles is of the form $\varphi_s(p)=\exp_p(s\,u(p)\,T_p)$ will be admissible by Corollary~\ref{cor:vertical}. Note that in a model space $\e$ the hypothesis $\ptl\overline{u}/\ptl \theta=0$ means that the function $u$ is radially symmetric, i.e., invariant under Euclidean vertical rotations. 
\end{example}

\providecommand{\bysame}{\leavevmode\hbox to3em{\hrulefill}\thinspace}
\providecommand{\MR}{\relax\ifhmode\unskip\space\fi MR }
\providecommand{\MRhref}[2]{%
  \href{http://www.ams.org/mathscinet-getitem?mr=#1}{#2}
}
\providecommand{\href}[2]{#2}


\begin{thebibliography}{10}

\bibitem{alexandrov}
A.~D. Alexandrov, \emph{A characteristic property of spheres}, Ann. Mat. Pura
  Appl. (4) \textbf{58} (1962), 303--315. \MR{0143162 (26 \#722)}

\bibitem{balogh}
Z.~M. Balogh, \emph{Size of characteristic sets and functions with prescribed
  gradient}, J. Reine Angew. Math. \textbf{564} (2003), 63--83. \MR{MR2021034
  (2005d:43007)}

\bibitem{bdce}
J.~L. Barbosa, M.~P. do~Carmo, and J.~Eschenburg, \emph{Stability of
  hypersurfaces of constant mean curvature in {R}iemannian manifolds}, Math. Z.
  \textbf{197} (1988), no.~1, 123--138. \MR{MR917854 (88m:53109)}

\bibitem{bscv}
V.~Barone~Adesi, F.~Serra~Cassano, and D.~Vittone, \emph{The {B}ernstein
  problem for intrinsic graphs in {H}eisenberg groups and calibrations}, Calc.
  Var. Partial Differential Equations \textbf{30} (2007), no.~1, 17--49.
  \MR{MR2333095}

\bibitem{falbel4}
P.~Bieliavsky, E.~Falbel, and C.~Gorodski, \emph{The classification of
  simply-connected contact sub-{R}iemannian symmetric spaces}, Pacific J. Math.
  \textbf{188} (1999), no.~1, 65--82. \MR{MR1680411 (2000d:53051)}

\bibitem{blair}
D.~E. Blair, \emph{Riemannian geometry of contact and symplectic manifolds},
  Progress in Mathematics, vol. 203, Birkh\"auser Boston Inc., Boston, MA,
  2002. \MR{MR1874240 (2002m:53120)}

\bibitem{cdg1}
L.~Capogna, D.~Danielli, and N.~Garofalo, \emph{The geometric {S}obolev
  embedding for vector fields and the isoperimetric inequality}, Comm. Anal.
  Geom. \textbf{2} (1994), no.~2, 203--215. \MR{MR1312686 (96d:46032)}

\bibitem{survey}
L.~Capogna, D.~Danielli, S.~D. Pauls, and J.~T. Tyson, \emph{An introduction to
  the {H}eisenberg group and the sub-{R}iemannian isoperimetric problem},
  Progress in Mathematics, vol. 259, Birkh\"auser Verlag, Basel, 2007.
  \MR{MR2312336}

\bibitem{chanillo-yang}
S.~Chanillo and P.~Yang, \emph{Isoperimetric inequalities \& volume comparison
  theorems on {CR} manifolds}, Ann. Sc. Norm. Super. Pisa Cl. Sci. (5)
  \textbf{8} (2009), no.~2, 279--307. \MR{MR2548248}

\bibitem{cchy}
J.-H. Cheng, H.-L. Chiu, J.-F. Hwang, and P.~Yang, \emph{Umbilicity and
  characterization of {P}ansu spheres in the {H}eisenberg group},
  arXiv:1406.2444v2 (to appear in J. Reine Angew. Math.), April 2015.

\bibitem{ch}
J.-H. Cheng and J.-F. Hwang, \emph{Properly embedded and immersed minimal
  surfaces in the {H}eisenberg group}, Bull. Austral. Math. Soc. \textbf{70}
  (2004), no.~3, 507--520. \MR{MR2103983 (2005f:53010)}

\bibitem{ch2}
\bysame, \emph{Variations of generalized area functionals and {$p$}-area
  minimizers of bounded variation in the {H}eisenberg group}, Bull. Inst. Math.
  Acad. Sin. (N.S.) \textbf{5} (2010), no.~4, 369--412. \MR{2809838
  (2012g:49085)}

\bibitem{chmy}
J.-H. Cheng, J.-F. Hwang, A.~Malchiodi, and P.~Yang, \emph{Minimal surfaces in
  pseudohermitian geometry}, Ann. Sc. Norm. Super. Pisa Cl. Sci. (5) \textbf{4}
  (2005), no.~1, 129--177. \MR{MR2165405 (2006f:53008)}

\bibitem{chmy2}
\bysame, \emph{A {C}odazzi-like equation and the singular set for {$C^1$}
  smooth surfaces in the {H}eisenberg group}, J. Reine Angew. Math.
  \textbf{671} (2012), 131--198. \MR{2983199}

\bibitem{chy}
J.-H. Cheng, J.-F. Hwang, and P.~Yang, \emph{Existence and uniqueness for
  {$p$}-area minimizers in the {H}eisenberg group}, Math. Ann. \textbf{337}
  (2007), no.~2, 253--293. \MR{MR2262784}

\bibitem{chy2}
\bysame, \emph{Regularity of {$C^1$} smooth surfaces with prescribed {$p$}-mean
  curvature in the {H}eisenberg group}, Math. Ann. \textbf{344} (2009), no.~1,
  1--35. \MR{2481053 (2010b:35117)}

\bibitem{chern}
S.~S. Chern, \emph{On surfaces of constant mean curvature in a
  three-dimensional space of constant curvature}, Geometric dynamics ({R}io de
  {J}aneiro, 1981), Lecture Notes in Math., vol. 1007, Springer, Berlin, 1983,
  pp.~104--108. \MR{730266 (86b:53058)}

\bibitem{dgn}
D.~Danielli, N.~Garofalo, and D.-M. Nhieu, \emph{Sub-{R}iemannian calculus on
  hypersurfaces in {C}arnot groups}, Adv. Math. \textbf{215} (2007), no.~1,
  292--378. \MR{MR2354992}

\bibitem{d2}
M.~Derridj, \emph{Sur un th\'eor\`eme de traces}, Ann. Inst. Fourier (Grenoble)
  \textbf{22} (1972), no.~2, 73--83. \MR{MR0343011 (49 \#7755)}

\bibitem{dcriem}
M.~P. do~Carmo, \emph{Riemannian geometry}, Mathematics: Theory \&
  Applications, Birkh\"auser Boston Inc., Boston, MA, 1992, Translated from the
  second Portuguese edition by Francis Flaherty. \MR{MR1138207 (92i:53001)}

\bibitem{falbel1}
E.~Falbel and C.~Gorodski, \emph{Sub-{R}iemannian homogeneous spaces in
  dimensions {$3$} and {$4$}}, Geom. Dedicata \textbf{62} (1996), no.~3,
  227--252. \MR{MR1406439 (97g:53060)}

\bibitem{folland}
G.~B. Folland, \emph{Real analysis}, second ed., Pure and Applied Mathematics
  (New York), John Wiley \& Sons Inc., New York, 1999, Modern techniques and
  their applications, A Wiley-Interscience Publication. \MR{1681462
  (2000c:00001)}

\bibitem{fssc}
B.~Franchi, R.~Serapioni, and F.~Serra~Cassano, \emph{Rectifiability and
  perimeter in the {H}eisenberg group}, Math. Ann. \textbf{321} (2001), no.~3,
  479--531. \MR{MR1871966 (2003g:49062)}

\bibitem{galli}
M.~Galli, \emph{First and second variation formulae for the sub-{R}iemannian
  area in three-dimensional pseudo-{H}ermitian manifolds}, Calc. Var. Partial
  Differential Equations \textbf{47} (2013), no.~1-2, 117--157. \MR{3044134}

\bibitem{galli2}
\bysame, \emph{On the classification of complete area-stationary and stable
  surfaces in the subriemannian {S}ol manifold}, Pacific J. Math. \textbf{271}
  (2014), no.~1, 143--157. \MR{3259763}

\bibitem{galli-ritore}
M.~Galli and M.~Ritor\'e, \emph{Existence of isoperimetric regions in contact
  sub-{R}iemannian manifolds}, J. Math. Anal. Appl. \textbf{397} (2013), no.~2,
  697--714. \MR{2979606}

\bibitem{galli-ritore2}
\bysame, \emph{Area stationary and stable surfaces of class ${C}^1$ in the
  sub-{R}iemannian {H}eisenberg group $\mathbb{H}^1$}, arXiv:1410.3619v2,
  November 2014.

\bibitem{gromov-cc}
M.~Gromov, \emph{Carnot-{C}arath\'eodory spaces seen from within},
  Sub-Riemannian geometry, Progress in Mathematics, vol. 144, Birkh\"auser,
  Basel, 1996, pp.~79--323. \MR{MR1421823 (2000f:53034)}

\bibitem{hardy}
G.~H. Hardy, \emph{On double {F}ourier series and especially which represent
  the double zeta function with real and incommensurable parameters}, Quart. J.
  Math. Oxford Ser. \textbf{37} (1906), 53--79.

\bibitem{hp1}
R.~K. Hladky and S.~D. Pauls, \emph{Constant mean curvature surfaces in
  sub-{R}iemannian geometry}, J. Differential Geom. \textbf{79} (2008), no.~1,
  111--139. \MR{MR2401420}

\bibitem{hp2}
\bysame, \emph{Variation of perimeter measure in sub-{R}iemannian geometry},
  Int. Electron. J. Geom. \textbf{6} (2013), no.~1, 8--40. \MR{3048517}

\bibitem{hopf}
H.~Hopf, \emph{Differential geometry in the large}, second ed., Lecture Notes
  in Mathematics, vol. 1000, Springer-Verlag, Berlin, 1989, Notes taken by
  Peter Lax and John W. Gray, With a preface by S. S. Chern, With a preface by
  K. Voss. \MR{1013786 (90f:53001)}

\bibitem{hughen}
K.~Hughen, \emph{The geometry of sub-{R}iemannian three manifolds}, Ph.D.
  thesis, Duke University, 1995.

\bibitem{hrr}
A.~Hurtado, M.~Ritor\'e, and C.~Rosales, \emph{The classification of complete
  stable area-stationary surfaces in the {H}eisenberg group $\mathbb{H}^1$},
  Adv.~Math. \textbf{224} (2010), 561--600.

\bibitem{hr1}
A.~Hurtado and C.~Rosales, \emph{Area-stationary surfaces inside the
  sub-{R}iemannian three-sphere}, Math. Ann. \textbf{340} (2008), no.~3,
  675--708. \MR{MR2358000 (2008i:53038)}

\bibitem{hr3}
\bysame, \emph{The ${C}^2$ isoperimetric problem in the sub-{R}iemannian
  $3$-sphere}, in preparation, 2015.

\bibitem{kobayashi}
S.~Kobayashi and K~Nomizu, \emph{Foundations of differential geometry. {V}ol.
  {I}}, Wiley Classics Library, John Wiley \& Sons Inc., New York, 1996,
  Reprint of the 1963 original, A Wiley-Interscience Publication. \MR{1393940
  (97c:53001a)}

\bibitem{leomas}
G.~P. Leonardi and S.~Masnou, \emph{On the isoperimetric problem in the
  {H}eisenberg group {${\Bbb H}\sp n$}}, Ann. Mat. Pura Appl. (4) \textbf{184}
  (2005), no.~4, 533--553. \MR{MR2177813}

\bibitem{lr}
G.~P. Leonardi and S.~Rigot, \emph{Isoperimetric sets on {C}arnot groups},
  Houston J. Math. \textbf{29} (2003), no.~3, 609--637 (electronic).
  \MR{MR2000099 (2004d:28008)}

\bibitem{mitrinovic}
D.~S. Mitrinovi{{\'c}}, \emph{Analytic inequalities}, Springer-Verlag, New
  York-Berlin, 1970. \MR{0274686 (43 \#448)}

\bibitem{montefalcone}
F.~Montefalcone, \emph{Hypersurfaces and variational formulas in
  sub-{R}iemannian {C}arnot groups}, J. Math. Pures Appl. (9) \textbf{87}
  (2007), no.~5, 453--494. \MR{MR2322147 (2008d:53035)}

\bibitem{montefalcone-spheres}
\bysame, \emph{Stability of {H}eisenberg isoperimetric profiles},
  arXiv:1110.0707, November 2011.

\bibitem{montefalcone2}
\bysame, \emph{Stable {H}-{M}inimal {H}ypersurfaces}, J. Geom. Anal.
  \textbf{25} (2015), no.~2, 820--870. \MR{3319952}

\bibitem{montgomery}
R.~Montgomery, \emph{A tour of subriemannian geometries, their geodesics and
  applications}, Mathematical Surveys and Monographs, vol.~91, American
  Mathematical Society, Providence, RI, 2002. \MR{MR1867362 (2002m:53045)}

\bibitem{monti}
R.~Monti, \emph{Brunn-{M}inkowski and isoperimetric inequality in the
  {H}eisenberg group}, Ann. Acad. Sci. Fenn. Math. \textbf{28} (2003), no.~1,
  99--109. \MR{MR1976833 (2004c:28021)}

\bibitem{monti-rearrangements}
\bysame, \emph{Rearrangements in metric spaces and in the {H}eisenberg group},
  J. Geom. Anal. \textbf{24} (2014), no.~4, 1673--1715. \MR{3261714}

\bibitem{mscv}
R.~Monti, F.~Serra~Cassano, and D.~Vittone, \emph{A negative answer to the
  {B}ernstein problem for intrinsic graphs in the {H}eisenberg group},
  Boll.~Unione~Mat.~Ital. (9) (2008), no.~3, 709--728, ISSN 1972-6724.

\bibitem{pansu1}
P.~Pansu, \emph{An isoperimetric inequality on the {H}eisenberg group}, Rend.
  Sem. Mat. Univ. Politec. Torino (1983), no.~Special Issue, 159--174 (1984),
  Conference on differential geometry on homogeneous spaces (Turin, 1983).
  \MR{MR829003 (87e:53070)}

\bibitem{pauls-regularity}
S.~D. Pauls, \emph{{$H$}-minimal graphs of low regularity in {$\Bbb H\sp 1$}},
  Comment. Math. Helv. \textbf{81} (2006), no.~2, 337--381. \MR{MR2225631
  (2007g:53032)}

\bibitem{pinkall}
U.~Pinkall, \emph{Hopf tori in {$S^3$}}, Invent. Math. \textbf{81} (1985),
  no.~2, 379--386. \MR{799274 (86k:53075)}

\bibitem{r2}
M.~Ritor\'e, \emph{Examples of area-minimizing surfaces in the sub-{R}iemannian
  {H}eisenberg group $\mathbb{H}^1$ with low regularity}, Calc. Var. Partial
  Differential Equations \textbf{34} (2009), no.~2, 179--192. \MR{MR2448649
  (2009h:53062)}

\bibitem{ritore-calibrations}
\bysame, \emph{A proof by calibration of an isoperimetric inequality in the
  {H}eisenberg group {${\Bbb H}^n$}}, Calc. Var. Partial Differential Equations
  \textbf{44} (2012), no.~1-2, 47--60. \MR{2898770}

\bibitem{rr1}
M.~Ritor\'e and C.~Rosales, \emph{Rotationally invariant hypersurfaces with
  constant mean curvature in the {H}eisenberg group {$\Bbb H\sp n$}}, J. Geom.
  Anal. \textbf{16} (2006), no.~4, 703--720. \MR{MR2271950}

\bibitem{rr2}
\bysame, \emph{Area-stationary surfaces in the {H}eisenberg group {$\Bbb H\sp
  1$}}, Adv. Math. \textbf{219} (2008), no.~2, 633--671. \MR{MR2435652}

\bibitem{rosales}
C.~Rosales, \emph{Complete stable {CMC} surfaces with empty singular set in
  {S}asakian sub-{R}iemannian $3$-manifolds}, Calc. Var. Partial Differential
  Equations \textbf{43} (2012), no.~3--4, 311--345.

\bibitem{rumin}
M.~Rumin, \emph{Formes diff\'erentielles sur les vari\'et\'es de contact}, J.
  Differential Geom. \textbf{39} (1994), no.~2, 281--330. \MR{MR1267892
  (95g:58221)}

\bibitem{simon}
L.~Simon, \emph{Lectures on geometric measure theory}, Proceedings of the
  Centre for Mathematical Analysis, Australian National University, vol.~3,
  Australian National University Centre for Mathematical Analysis, Canberra,
  1983. \MR{MR756417 (87a:49001)}

\bibitem{singer-thorpe}
I.~M. Singer and J.~A. Thorpe, \emph{Lecture notes on elementary topology and
  geometry}, Springer-Verlag, New York, 1976, Reprint of the 1967 edition,
  Undergraduate Texts in Mathematics. \MR{0413152 (54 \#1273)}

\bibitem{tanno2}
S.~Tanno, \emph{The topology of contact {R}iemannian manifolds}, Illinois J.
  Math. \textbf{12} (1968), 700--717. \MR{0234486 (38 \#2803)}

\bibitem{tanno}
\bysame, \emph{Sasakian manifolds with constant {$\phi$}-holomorphic sectional
  curvature}, T\^ohoku Math. J. (2) \textbf{21} (1969), 501--507. \MR{MR0251667
  (40 \#4894)}

\bibitem{tolstov}
G.~P. Tolstov, \emph{Fourier series}, Dover Publications, Inc., New York, 1976,
  Second English translation, Translated from the Russian and with a preface by
  Richard A. Silverman. \MR{0425474 (54 \#13429)}

\end{thebibliography}
\end{document}